\documentclass{article}
\usepackage[journal=JEMS,lang=american]{ems-journal-jems} 
\makeatletter
\def\ps@titlepage{%
  \let\@oddfoot\@empty
  \let\@evenfoot\@empty
  \def\@oddhead{\jems@head{}{\thepage}}%
  \let\@evenhead\@oddhead
}
\makeatother
\usepackage{bm,mathtools}
\theoremstyle{plain}%
\newtheorem{theorem}{Theorem}[section]
\newtheorem{proposition}[theorem]{Proposition}
\newtheorem{lemma}[theorem]{Lemma}

\theoremstyle{definition}%
\newtheorem{remark}[theorem]{Remark}%
\newtheorem{definition}[theorem]{Definition}%
\newtheorem{assumption}[theorem]{Assumption}%

\allowdisplaybreaks[3]
\newcommand{\N}{{\mathbb N}}

\newcommand{\R}{{\mathbb R}}
\newcommand{\Z}{{\mathbb Z}}
\newcommand{\cB}{{\mathcal B}}

\newcommand{\cD}{{\mathcal D}}
\newcommand{\cE}{{\mathcal E}}
\newcommand{\cF}{{\mathcal F}}
\newcommand{\cG}{{\mathcal G}}
\newcommand{\cH}{{\mathcal H}}
\newcommand{\cL}{{\mathcal L}}
\newcommand{\cM}{{\mathcal M}}
\newcommand{\cN}{{\mathcal N}}
\newcommand{\cP}{{\mathcal P}}
\newcommand{\cQ}{{\mathcal Q}}
\newcommand{\cW}{{\mathcal W}}
\newcommand{\dd}{{\mathrm{d}}}
\renewcommand{\a}{\alpha}\renewcommand{\b}{\beta}

\newcommand{\gm}{\gamma}\newcommand{\dl}{\delta}
\newcommand{\eps}{\varepsilon}
\newcommand{\lm}{\lambda}\newcommand{\kp}{\kappa}
\newcommand{\zt}{\zeta}
\newcommand{\sg}{\sigma}
\newcommand{\ph}{\varphi}
\newcommand{\om}{\omega}
\newcommand{\Gm}{\Gamma}
\newcommand{\Lm}{\Lambda}
\newcommand{\Sg}{\Sigma}
\newcommand{\Om}{\Omega}
\newcommand{\la}{\langle}\newcommand{\ra}{\rangle}
\newcommand{\wg}{\wedge}
\newcommand{\bff}{{\bm{f}}}
\newcommand{\bfh}{{\bm{h}}}
\newcommand{\bfr}{{\bm{r}}}
\newcommand{\bfM}{{\mathbf{M}}}
\newcommand{\bfone}{{\mathbf{1}}}
\newcommand{\sfd}{{\mathsf{d}}}
\newcommand{\HS}{\mathrm{HS}}
\newcommand{\maruM}{{\stackrel{\,\circ}{\smash{\cM}\rule{0pt}{1.1ex}}}}
\DeclareRobustCommand{\supp}{\mathop{\mathrm{supp}}}
\DeclareRobustCommand{\Cp}{\mathrm{Cap}}
\DeclareRobustCommand{\osc}{\mathop{\mathrm{osc}}}
\DeclareRobustCommand{\rank}{\mathop{\mathrm{rank}}}
\DeclareRobustCommand{\tr}{\mathop{\mathrm{tr}}}
\DeclareRobustCommand{\Mat}{\mathrm{Mat}}

\DeclareRobustCommand{\muesssup}{\mathop{\mu\text{-}\mathrm{ess\,sup}}}
\DeclareRobustCommand{\muessinf}{\mathop{\mu\text{-}\mathrm{ess\,inf}}}
\DeclareRobustCommand{\muosc}{\mathop{\mu\text{-}\mathrm{osc}}}
\DeclareRobustCommand{\nuesssup}{\mathop{\nu\text{-}\mathrm{ess\,sup}}}
\DeclareRobustCommand{\diam}{\operatorname{diam}}
\DeclareRobustCommand{\Int}{\operatorname{int}}
\DeclareRobustCommand{\dist}{\operatorname{dist}}
\DeclareRobustCommand{\Wh}{\operatorname{Wh}}

\numberwithin{equation}{section}

\begin{document}

\title{Martingale dimensions for a class of metric measure spaces}
\titlemark{Martingale dimensions for a class of metric measure spaces}


\emsauthor{1}{
	\givenname{Masanori}
	\surname{Hino}
	\mrid{358504}
	\zblid{hino.masanori}
	\orcid{0009-0004-8835-7115}}{M.~Hino}

\Emsaffil{1}{
	\pretext{}
	\department{Department of Mathematics, Graduate School of Science}
	\organisation{Kyoto University}
	\rorid{02kpeqv85}
	\address{Kitashirakawaoiwakecho, Sakyo-ku}
	\zip{606-8502}
	\city{Kyoto}
	\country{Japan}
	\posttext{}
	\affemail{hino@math.kyoto-u.ac.jp}
	\furtheremail{}}
	

\classification[60J46, 60J60, 60G44, 28A80]{31E05}


\begin{abstract}
We study the martingale dimension of diffusion processes associated with strongly local regular Dirichlet forms on metric measure spaces. The martingale dimension is an intrinsic quantity describing the probabilistic complexity of a diffusion, but its precise determination in non-smooth settings is highly nontrivial.
In this paper, we establish a purely analytic criterion, formulated in terms of energy measures and relative capacities, which guarantees that the AF-martingale dimension equals one. Our approach does not rely on self-similarity, heat kernel estimates, or global geometric assumptions, and applies to settings where such structures are unavailable.
The proof is based on a blow-up analysis of harmonic functions combined with a push-forward argument for energy measures, which excludes the coexistence of multiple independent martingale directions. 
As applications, we verify the assumptions of the theory for a class of inhomogeneous Sierpinski gaskets, including genuinely non-self-similar examples, and for the two-dimensional standard Sierpinski carpet, where the criterion yields a new capacity-based proof of the one-dimensionality of the martingale structure. 
In these examples, self-similarity or local symmetry enters only in verifying the hypotheses; the mechanism that forces the one-dimensionality is the analytic balance condition itself.
\end{abstract}

\maketitle

\section{Introduction}
The local structure of diffusion processes on irregular spaces has been a central topic in analysis and probability theory. On a metric measure space, several notions of dimension coexist, reflecting different aspects of geometry and stochastic behavior. Among them, the martingale dimension, introduced in the framework of Dirichlet forms, plays a distinguished role as an intrinsic measure of the probabilistic complexity of a diffusion.

While the martingale dimension agrees with the manifold dimension for non-degenerate diffusions on smooth Riemannian manifolds, its determination in non-smooth spaces is highly nontrivial and remains poorly understood in general. 
In this direction, several results have been obtained for specific classes of fractal spaces \cite{Kus89,Hi08,Hi13}. The first result identifying the martingale dimension in a non-smooth setting was obtained by Kusuoka~\cite{Kus89}, who proved that the AF-martingale dimension of Brownian motion on Sierpinski gaskets of arbitrary dimension equals one. Here the AF-martingale dimension is defined in terms of additive functionals.
In \cite{Hi08,Hi13}, the AF-martingale dimension was proved to be one for the canonical diffusion processes on p.c.f.\ self-similar sets, while the AF-martingale dimension was proved to be dominated by the spectral dimension for the Brownian motion on Sierpinski carpets. These results, though obtained by detailed arguments based on the self-similarity of the underlying spaces, suggest a close relationship between the martingale dimension and other notions of dimension despite the lack of smooth structure. 

More recently, Murugan and collaborators~\cite{Mu24,Mu25,EM25} studied the relationship between the AF-martingale dimension and the analytic dimension under two-sided Gaussian or sub-Gaussian heat kernel bounds. In particular, they showed that the AF-martingale dimension coincides with the analytic dimension under Gaussian estimates and that it is finite under sub-Gaussian bounds.
While these results provide important insight into the probabilistic structure of diffusions, they are obtained under global heat kernel assumptions.

Therefore, a natural problem is to identify a general mechanism
for determining the exact value of the martingale dimension
without assuming geometric symmetries or heat kernel bounds.
The purpose of the present paper is to show that the AF-martingale dimension can be determined under a purely analytic condition formulated in terms of energy measures and relative capacities, without assuming any form of self-similarity, volume growth, or heat kernel estimates. More precisely, we prove that under this condition the AF-martingale dimension must equal one, independently of the global geometric structure of the space. The condition we introduce is local in nature and applies to a broad class of strongly local regular Dirichlet forms, including settings in which the underlying geometry is highly irregular or inhomogeneous.

A key ingredient of the proof is a blow-up analysis of harmonic
functions combined with a push-forward argument for energy
measures.
If the martingale dimension were at least two, we could construct
normalized harmonic coordinates on a sequence of shrinking
domains.
After localization by equilibrium potentials, the push-forwards
of the two diagonal energy measures would converge to the same
non-zero measure $\xi(x)\,\dd x$ on $\R^2$, whereas the
push-forwards of the total variation measures of the off-diagonal
terms would converge to zero.
The resulting limiting energy structure is therefore isotropic.
At a suitable point in the coordinate image, logarithmic cutoff
functions have arbitrarily small energy with respect to the
limiting density $\xi$.
Pulling these functions back through the harmonic coordinates
and localizing them in the original domains produces subsets
of the corresponding cores whose relative capacities, divided
by the capacities of the cores, tend to zero.
Each such subset has positive energy measure and hence contains
a point outside the exceptional set in (A3)~(c) or (A3')~(e).
The uniform lower bound on point capacities then gives a
contradiction.
This argument excludes martingale dimension two or higher
without using self-similarity or an explicit renormalization
structure.

A few words are in order on the nature of the balance conditions. Condition (A3)~(b) (resp.\ (A3')~(d)) is an oscillation-energy estimate of resistance type for harmonic functions, localized to each piece of the partition, and it is what allows the harmonic functions to be normalized as coordinates in the blow-up argument. Condition (A3)~(c) (resp.\ (A3')~(e)), a scale-uniform lower bound on the relative capacity of points, is used at only one point of the proof, namely to rule out that the sets produced by the logarithmic cutoff become negligible in capacity. 
We expect that the latter condition can be relaxed, although this would require a different argument at the step mentioned above; we leave this for future work.

As applications, we verify the assumptions of our theory in three settings: p.c.f.\ self-similar sets with regular harmonic structures, where the known result \cite{Hi08,Hi13} is recovered within the present framework; a class of inhomogeneous Sierpinski gaskets, including genuinely non-self-similar members beyond the scope of earlier results based on global self-similarity; and the two-dimensional standard Sierpinski carpet, an infinitely ramified fractal for which the window-type condition~(A3') is verified mainly on the basis of a local resistance estimate.

The remainder of the paper is organized as follows. In Section~\ref{sec:def} we review the notion of the AF-martingale dimension and its representation in terms of Dirichlet forms and energy measures. Section~\ref{sec:results} introduces our analytic assumptions, states the main theorem (Theorem~\ref{th:main}), and establishes several preliminary results. Section~\ref{sec:main} contains the proof of Theorem~\ref{th:main}. 
Examples are given in Section~\ref{sec:example}, where the theorem is applied to p.c.f.\ self-similar sets (Section~\ref{sec:pcf}), to inhomogeneous Sierpinski gaskets
(Section~\ref{sec:inh}), and to the two-dimensional standard Sierpinski carpet (Section~\ref{sec:carpet}); the last example is treated by means of the window-type condition~(A3'). The argument in Section~\ref{sec:carpet} also extends to higher-dimensional generalized Sierpinski carpets with spectral dimension less than two (Remark~\ref{rem:SC}), although for these carpets,
as for the planar one, the conclusion itself is known from \cite{Hi13}.

A preliminary announcement of the main results was given in~\cite{Hi25}; the present paper provides a generalization of the main theorem (introduction of (A3')) and contains the full proofs.

\subsection*{Notation}
\begin{itemize}
\item $\Z_+=\{m\in\Z\mid m\ge0\}$.
\item $a\vee b=\max\{a,b\}$, $a\wg b=\min\{a,b\}$.
\item $A\triangle B$ denotes the symmetric difference of two sets $A$ and $B$.
\item $\bfone_A$ denotes the indicator function of a set $A$.
\item For a map $\Phi$ and a set $A$, $\Phi|_A$ denotes the map $\Phi$ whose defining set is restricted to $A$.
\item $\psi^* f$ denotes the pullback of $f$ by the map $\psi$, that is, $\psi^* f=f\circ\psi$.
\item For a subset $A$ of a topological space, $\overline{A}$ denotes the closure of $A$.
\item For a topological space $K$, $C(K)$ denotes the totality of  real continuous functions on $K$, and $C_c(K)$ denotes the totality of functions in $C(K)$ with compact support.
\item $C^1_b(\R^d)$ denotes the totality of real and bounded $C^1$-functions on $\R^d$ with bounded first order derivatives.
\item $\cL^d$ denotes the $d$-dimensional Lebesgue measure.
\item $W^{r,p}(\R^d)$ denotes the $L^p$-Sobolev space with order $r$ on $\R^d$.
\item $\|\cdot\|_{L^p(K,\mu)}$ denotes the $L^p$-norm of the $L^p$-space $L^p(K,\mu)$, $1\le p\le \infty$.
\item The symbol $\fint$ denotes the normalized integral, that is, $\fint_A f\,\dd\nu=\nu(A)^{-1}\int_A f\,\dd\nu$.
\item For a signed measure $\nu$, $|\nu|$ denotes the total variation measure of $\nu$.
\item For a measure $\nu$ and a map $\Phi$, $\Phi_* \nu$ denotes the push-forward measure of $\nu$ by $\Phi$.
\item For a measure $\nu$ and a measurable set $A$, $\nu|_A$ denotes the measure $\nu$ whose domain is restricted to $A$.
\item $\|F\|_\infty$ denotes the supremum norm of a function $F$.
\item $|a|$ denotes the Euclidean norm of $a\in\R^d$.
\item $\|A\|_{\HS}$ denotes the Hilbert--Schmidt norm of a matrix $A$.
\item $\dl_{ij}$ denotes the Kronecker delta, that is, $\dl_{ij}=\begin{cases}1&(i=j)\\0&(i\ne j).\end{cases}$
\end{itemize}
Note that we tacitly assume all functions in $\cF$, the domain of a regular Dirichlet form, are taken in their quasi-continuous versions unless otherwise mentioned (see the beginning of Section~\ref{sec:results}).

\section{Martingale dimensions and their representation in terms of Dirichlet forms}\label{sec:def}
In order to rigorously define the AF-martingale dimension, we begin by recalling the framework of symmetric Dirichlet forms and their associated diffusion processes, following~\cite{Hi10,FOT11}. This analytic foundation will allow us to represent martingale additive functionals in terms of energy measures, which is essential for formulating and proving our main results in later sections.

Let $K$ be a locally compact separable metrizable space. Its one-point compactification is denoted by $K_\Delta=K\cup\{\Delta\}$. The Borel $\sg$-field on $K$ (resp.\ $K_\Delta$) is denoted by $\cB$ (resp.\ $\cB_\Delta$).
Let $\mu$ be a positive Radon measure on $K$ with full support. 
We recall the definition of a symmetric regular Dirichlet form $(\cE,\cF)$ on $L^2(K,\mu)$. It consists of a bilinear form and a function space satisfying the following properties:
\begin{itemize}
\item $\cF$ is a dense linear subspace of $L^2(K,\mu)$.
\item $\cE$ is a non-negative definite symmetric bilinear form on $\cF$.
\item (Closedness) By equipping $\cF$ with the inner product $(f,g)_\cF:=\cE(f,g)+\int_K f g\,\dd\mu$, $\cF$ becomes a Hilbert space.
\item (Markovian property) For every $f\in\cF$, $\hat f:=(0\vee f)\wg 1$ belongs to $\cF$ and $\cE(\hat f,\hat f)\le \cE(f,f)$.
\item (Regularity) $\cF\cap C_c(K)$ is dense in $\cF$ with the topology derived by $(\cdot,\cdot)_\cF$ and dense in $C_c(K)$ with the uniform topology.
\end{itemize}
A non-positive self-adjoint operator $L$ on $L^2(K,\mu)$ is associated with $(\cE,\cF)$ in the sense that the domain of $\sqrt{-L}$ is equal to $\cF$ and 
\[
\cE(f,g)=\int_K (\sqrt{-L}f)(\sqrt{-L}g)\,\dd\mu\quad\text{for $f,g\in \cF$}.
\]
The operator $L$ generates the semigroup $\{T_t\}_{t\ge0}$ on $L^2(K,\mu)$ by defining $T_t=e^{t L}$. 
For a subset $A$ of $K$, we define the (1-)\emph{capacity} $\Cp_1(A)$ of $A$ as
\[
\Cp_1(A)=\inf\left\{\cE(f,f)+\int_K f^2\,\dd\mu\mathrel{}\middle|\mathrel{}\parbox{0.5\hsize}{$f\in\cF$ and $f\ge 1$ $\mu$-a.e.\ on some open set containing $A$}\right\},
\]
where $\inf\emptyset=+\infty$.
A Borel measure $\lm$ on $K$ is called \emph{smooth} if the following two conditions are satisfied.
\begin{enumerate}
\item $\lm$ charges no set of zero capacity, that is, $\lm(A)=0$ for all Borel subsets $A$ of $K$ with $\Cp_1(A)=0$.
\item There exists an increasing sequence $\{F_n\}_{n=1}^\infty$ of closed sets of $K$ such that $\lm(F_n)<\infty$ for all $n$ and $\lim_{n\to\infty}\Cp_1(C\setminus F_n)=0$ for any compact subset $C$ of $K$.
\end{enumerate}

In what follows, we further assume that $(\cE,\cF)$ is \emph{strongly local}. That is, $\cE(u,v)=0$ for $u,v\in\cF$ if $v$ is constant on a neighborhood of $\supp[u]$, where 
\[
\supp[u]=\left\{x\in K\mathrel{}\middle|\mathrel{} \int_U |u|\,\dd\mu>0\text{ for any neighborhood $U$ of $x$}\right\}.
\]
Then, from \cite[Theorem~4.5.3]{FOT11}, there corresponds to a $\mu$-symmetric diffusion process $\bfM=(\Om,\cF_\infty,\{\cF_t\}_{t\ge0},\{X_t\}_{t\ge0},\{P_x\}_{x\in K_\Delta})$ on $K$ satisfying $P_x(X_{\zeta-}\in K,\ \zeta<\infty)=0$ for every $x\in K$.
More precisely, we can construct a diffusion process $\{X_t\}_{t\ge0}$ on $K$ defined on a filtered probability space $(\Om,\cF_\infty,\{\cF_t\}_{t\ge0})$ with a family of probability measures $\{P_x\}_{x\in K_\Delta}$ and shift operators $\{\theta_t\}_{t\in[0,+ \infty]}$ such that the following hold.
\begin{itemize}
\item $(\Om,\cF_\infty,\{\cF_t\}_{t\ge0})$ is a filtered probability space and the filtration $\{\cF_t\}_{t\ge0}$ is right continuous ($\bigcap_{t>s}\cF_t=\cF_s$ for all $s\ge0$).
\item For each $x\in K_\Delta$, $P_x$ is a probability measure on $ (\Om,\cF_\infty)$.
\item For each $t\ge0$, $X_t\colon \Om\to K_\Delta$ is $\cF_t/\cB_\Delta$-measurable. We set $X_\infty(\om)=\Delta$ for $\om\in\Om$.
\item For each $t\ge0$ and $A\in\cB$, $P_x(X_t\in A)$ is $\cB$-measurable in $x\in K$.
\item For any $t\ge0$, $P_\Delta(X_t=\Delta)=1$.
\item (Normality) For any $x\in K$, $P_x(X_0=x)=1$.
\item $X_t(\om)=\Delta$ for all $t\ge\zeta(\om)$, where $\zeta(\om)=\inf\{t\ge0\mid X_t(\om)=\Delta\}$ is the life time of $\{X_t\}_{t \ge0}$.
\item (Continuity of sample paths) For each $\om\in\Om$, the map $[0,\infty)\ni t\mapsto X_t(\om)\in K_\Delta$ is continuous.
\item Each $\theta_t$ is a measurable map from $(\Om,\cF_\infty)$ to itself. For each $s\ge0$ and $t\in[0,+\infty]$, $X_s\circ\theta_t=X_{s+t}$.
\item (Strong Markov property) For $A\in \cB_\Delta$, $s\ge0$, any $\{\cF_t\}_{t\ge0}$-stopping time $\sg$, and any probability measure $\lm$ on $(K_\Delta,\cB_\Delta)$,
\[
  P_\lm(X_{\sg+s}\in A\mid \cF_\sg)=P_{X_\sg}(X_s\in A)\quad P_\lm\text{-a.s.}
\]
Here, $P_\lm$ is the probability measure on $(\Om,\cF_\infty)$ defined as \begin{equation}\label{eq:Plm}
P_\lm(\Lambda)=\int_{K_\Delta}P_x(\Lambda)\,\lm(\dd x),\qquad \Lambda\in \cF_\infty.
\end{equation}
\item (Correspondence with Dirichlet forms) For any $t\ge0$ and $f\in L^2(K,\mu)$ that are $\cB$-measurable, $T_t f(x)=E_x[f(X_t)]$ holds for $\mu$-a.e.\,$x\in K$. Here $E_x$ denotes the expectation with respect to $P_x$, and $f$ extends to a function on $K_\Delta$ by setting $f(\Delta)=0$.
\end{itemize}

A subset $A$ of $K_\Delta$ is called \emph{nearly Borel} measurable if, for any Borel probability measure $\lm$ on $K_\Delta$, there exist Borel subsets $A_1$ and $A_2$ of $K_\Delta$ such that $A_1\subset A\subset A_2$ and\break $P_\lm(\{\text{There exists some $t\ge0$ such that $X_t\in A_2\setminus A_1$}\})=0$.
A subset $N$ of $K$ is called \emph{exceptional} if there exists a nearly Borel set $\tilde N$ including $N$ such that $P_\mu(\sg_{\tilde N}<\infty)=0$, where $\sg_{\tilde N}(\om)=\inf\{t>0\mid X_t(\om)\in \tilde N\}$. From \cite[Theorem~4.2.1]{FOT11}, a subset $N$ of $K$ is exceptional if and only if $\Cp_1(N)=0$.
We say that statements $P(x)$ depending on $x\in K$ hold for \emph{quasi-every} $x$ (q.e.\,$x$ in abbreviation) if $P(x)$ holds for all $x\in K\setminus N$ for some exceptional set $N$.

Without loss of generality, we assume that the filtration $\{\cF_t\}_{t\ge0}$ is the minimum completed admissible filtration. That is, $\cF_t$ $(t\ge0)$ and $\cF_\infty$ are defined in the following way.
\begin{itemize}
\item Let $\cF_t^0=\sg(\{X_s\mid 0\le s\le t\})$ and $\cF_\infty^0=\sg(\{X_s\mid s\ge 0\})$.
\item Let $\cP(K_\Delta)$ be the set of all probability measures on $(K_\Delta,\cB_\Delta)$.
\item For $\lm\in\cP(K_\Delta)$ let $P_\lm$ be the probability measure on $(\Om,\cF_\infty^0)$ defined as in \eqref{eq:Plm}.
\item For $\lm\in\cP(K_\Delta)$, $\cF_\infty^\lm$ denotes the completion of $\cF_\infty^0$ with respect to $P_\lm$, and $\cF_t^\lm$ denotes the completion of $\cF_t^0$ in $\cF_\infty^\lm$ with respect to $P_\lm$. That is, $\cF_t^\lm=\{\Lm\in\cF_\infty^\lm\mid \text{there exists} \break\text{$\Lm'\in\cF_t^0$ such that $P_\lm(\Lm\bigtriangleup \Lm')=0$}\}$.
\item Let $\cF_t=\bigcap_{\lm\in\cP(K_\Delta)}\cF_t^\lm$ and $\cF_\infty=\bigcap_{\lm\in\cP(K_\Delta)}\cF_\infty^\lm$.
\end{itemize}

An \emph{additive functional} (abbreviated as AF) $A_t(\om)$, $t\ge0$, $\om\in\Om$ is a $[-\infty,+\infty]$-valued function such that the following hold.
\begin{itemize}
\item For each $t\ge0$, $A_t(\cdot)$ is $\cF_t$-measurable.
\item There exist a set $\Lm\in\cF_\infty$ and an exceptional set $N$ such that $P_x(\Lm)=1$ for all $x\in K\setminus N$, $\theta_t(\Lm)\subset \Lm$ for all $t>0$, and for each $\om\in\Lm$ the following hold.
\begin{itemize}
\item $A_\cdot(\om)$ is right continuous and has the left limit on $[0,\zt(\om))$, where $\zt$ is the life time.
\item $A_0(\om)=0$, $|A_t(\om)|<\infty$ for $t<\zt(\om)$, and $A_t(\om)=A_{\zt(\om)}(\om)$ for $t\ge\zt(\om)$.
\item $A_{t+s}(\om)=A_s(\om)+A_t(\theta_s(\om))$ for $t,s\ge0$.
\end{itemize}
\end{itemize}
The sets $\Lm$ and $N$ are called a defining set and an exceptional set of the AF $A_t(\om)$, respectively.
Two AFs $A^{(1)}$ and $A^{(2)}$ are called \emph{equivalent} if $P_x(A^{(1)}_t=A^{(2)}_t)=1$ for q.e.\,$x\in K$ for each $t>0$. 
We identify equivalent AFs.

An AF $A_t(\om)$ is called a \emph{positive} (resp.\ \emph{continuous}) additive functional if there exists a defining set $\Lm$ satisfying $A_t(\om)\in[0,+\infty]$ for all $t\in[0,\infty)$ (resp.\ $A_t(\om)$ is continuous on $[0,\infty)$) for $\om\in\Lm$. 
A positive continuous additive functional is abbreviated as PCAF.
Any PCAF $A$ admits its \emph{Revuz measure} $\mu_A$. That is, there exists a unique smooth measure $\mu_A$ on $K$ such that for any $t>0$ and any non-negative Borel functions $f$ and $h$ on $K$,
\[
\int_K h(x) E_x\biggl[\int_0^t f(X_s)\,\dd A_s\biggr]\,\mu(\dd x)=\int_0^t \int_K f(x) E_x[h(X_s)]\,\mu_A(\dd x)\,\dd s. 
\]
(See \cite[Section~5.1]{FOT11}.)

Let $\cM$ be the space of all finite c\`adl\`ag additive functionals $M$ such that for each $t>0$, $E_x[M_t^2]<\infty$ and $E_x[M_t]=0$ q.e.\,$x\in K$.
An element of $\cM$ is called a martingale additive functional (MAF for short).
From the assumption of strong locality of $(\cE,\cF)$, every element of $\cM$ is in fact a continuous additive functional by \cite[Lemma~5.5.1~(ii)]{FOT11}.

Each $M\in\cM$ admits a unique PCAF $\la M \ra$ such that $E_x[\la M\ra_t]=E_x[M_t^2]$ for q.e.\,$x\in K$ for each $t>0$.
For $M,L\in\cM$, let
\[
\la M,L\ra_t=\frac12(\la M+L\ra_t-\la M\ra_t-\la L\ra_t)
\]
and define a signed measure $\mu_{\la M,L\ra}$ on $K$ as
\[
\mu_{\la M,L\ra}=\frac12(\mu_{\la M+L\ra}-\mu_{\la M\ra}-\mu_{\la L\ra}).
\]

For $M\in\cM$, the energy $e(M)$ of $M$ is defined as
\[
e(M)=\lim_{t\downarrow0}\frac1{2t}E_\mu[M_t^2]\left(=\sup_{t>0}\frac1{2t}E_\mu[M_t^2]\le\infty\right).
\]
Here, $E_\mu$ denotes the integration with respect to $P_\mu$.
We set $\maruM=\{M\in\cM\mid e(M)<\infty\}$.
By letting 
\[
e(M,M')=\lim_{t\downarrow0}\frac1{2t}E_\mu[M_t M'_t],\qquad M,M'\in\maruM,
\]
$(\maruM,e)$ becomes a real Hilbert space by \cite[Theorem~5.2.1]{FOT11}.
Here, two elements of $\maruM$ are identified if they are equivalent. 
Each $M\in \maruM$ associates a finite Borel measure $\mu_{\la M\ra}$ on $K$ in the Revuz correspondence with the quadratic variation $\la M\ra$ of $M$.
For $M\in \maruM$ and $f\in L^2(K,\mu_{\la M\ra})$, we can define the \emph{stochastic integral} $f\bullet M\in\maruM$ of $f$ with respect to $M$, which is characterized by the identity
\[
  e(f\bullet M,L)=\frac12\int_K f(x)\,\mu_{\la M,L\ra}(\dd x)
  \quad \text{for all }L\in\maruM.
\]
See \cite[Section~5.6]{FOT11} for further details.

We now define the concept of AF-martingale dimensions.
\begin{definition}[{\cite[Definition~3.3]{Hi10}}]
The \emph{AF-martingale dimension} $d_{\mathrm{m}}$ of $\bfM$ or of $(\cE,\cF)$ is the smallest number $d$ such that, there exist $d$ elements $M^{(1)},M^{(2)},\dots,M^{(d)}$ in $\maruM$ such that each $M\in\maruM$ is expressed as
\[
M=\sum_{i=1}^d (\ph^{(i)}\bullet M^{(i)})
\]
for some $\ph^{(i)}\in L^2(K,\mu_{\la M^{(i)}\ra})$, $i=1,2,\dots,d$.
If such a number does not exist, $d_{\mathrm{m}}$ is defined as $\infty$. Moreover, $d_{\mathrm{m}}$ is defined as $0$ if $\maruM=\{0\}$.
\end{definition}
The AF-martingale dimension informally represents the number of independent noises contained in the diffusion process. Its value is not readily determined from the definition. We introduce an analytic characterization of the AF-martingale dimension.
For a bounded function $f$ in $\cF$, there exists a unique finite Borel measure $\nu_f$ on $K$ satisfying
\[
  \int_K \ph\,\dd\nu_f=2\cE(f\ph,f)-\cE(\ph,f^2)\quad
  \text{for all }\ph\in\cF\cap C_c(K).
\]
For general $f\in\cF$, we define $\nu_f$ by $\nu_f(B)=\lim_{n\to\infty}\nu_{f_n}(B)$ for Borel sets $B$, where $f_n=(-n)\vee f\wg n$. 
We call $\nu_f$ the energy measure of $f$ (see \cite[Section~3.2]{FOT11}), which plays a central role in the subsequent arguments. For $f,g\in\cF$, the mutual energy measure $\nu_{f,g}$ is a signed measure on $K$ that is defined as
\[
\nu_{f,g}=\frac12(\nu_{f+g}-\nu_{f}-\nu_g).
\]

For a Borel measure $\nu$ on $K$ and a (signed) measure $\hat\nu$ on $K$, $\hat\nu\ll\nu$ means that $\hat\nu$ is absolutely continuous with respect to $\nu$, that is, $\nu(B)=0$ implies $\hat\nu(B)=0$ for Borel subsets $B$ of $K$.
We now recall the following proposition.
\begin{proposition}[{\cite[Lemma~2.3]{Hi10}, see also \cite[Lemma~2.2]{Na85}}]
There exists a (finite) Borel measure $\nu$ on $K$ such that the following holds:
\begin{enumerate}
\item For every $f\in\cF$, $\nu_f\ll\nu$.
\item If another measure $\nu'$ satisfies {\rm(i)} with $\nu$ replaced by $\nu'$, then $\nu\ll\nu'$.
\end{enumerate}
\end{proposition}
Such a measure $\nu$ is referred to as the \emph{minimal energy-dominant measure} (\cite[Definition~2.1]{Hi10}).
It is easy to see that $\nu_{f,g}\ll\nu$ for any $f,g\in\cF$.
\begin{definition}[{\cite[Definition~2.9]{Hi10}}]
The \emph{index} of $(\cE,\cF)$ is defined as the smallest integer $p$ such that for every $N\in\N$ and $f_1,\dots,f_N\in\cF$,
\[
\rank\left(\frac{\dd\nu_{f_i,f_j}}{\dd\nu}(x)\right)_{i,j=1}^N\le p,\quad \nu\text{-a.e.}\,x.
\]
If such $p$ does not exist, the index is defined as $\infty$.
\end{definition}
It is evident that the index is defined independently of the choice of a minimal energy-dominant measure.

The following key result links the analytic structure of Dirichlet forms to probabilistic martingale dimensions.
\begin{theorem}[{\cite[Theorem~3.4]{Hi10}}]\label{th:char}
The AF-martingale dimension of $\bfM$ coincides with the index of $(\cE,\cF)$.
\end{theorem}
From the next section, we treat the index of $(\cE,\cF)$, not the AF-martingale dimension itself.

\section{Main result and preliminary arguments}\label{sec:results}
In this section, we introduce the analytic framework required to state and prove our main result, which identifies sufficient conditions on energy measures and relative capacities ensuring that the AF-martingale dimension collapses to one. We begin by recalling some auxiliary notions, and then formulate the key assumptions.

A real function $f$ on $K$ is called quasi-continuous (resp.\ quasi-continuous in the restricted sense) if, for any $\eps>0$, there exists an open set $G$ of $K$ such that $\Cp_1(G)<\eps$ and $f|_{K\setminus G}$ is continuous (resp.\ $f|_{K_\Delta\setminus G}$ is continuous by letting $f(\Delta)=0$). From \cite[Theorem~2.1.3]{FOT11}, each $f\in\cF$ admits a quasi-continuous $\mu$-modification $\tilde f$ in the sense that $f=\tilde f$ $\mu$-a.e.\ and $\tilde f$ is quasi-continuous in the restricted sense. In what follows, when we consider functions in $\cF$, \emph{we assume that they are always taken to be a quasi-continuous $\mu$-modification}.
We recall the following fact.
\begin{lemma}[{\cite[Lemma~2.1.4]{FOT11}}]\label{lem:aeqe}
Let $U$ be an open subset of $K$ and $f$ a quasi-continuous function on $U$. If $f\ge0$ $\mu$-a.e.\ on $U$, then $f\ge0$ q.e.\ on $U$.
\end{lemma}
 From this lemma, if $\tilde f$ and $\hat f$ are both quasi-continuous $\mu$-modifications of a function $f$ on $K$, then $\tilde f=\hat f$ q.e.

A function $h$ in $\cF$ is called \emph{harmonic} on an open subset $U$ of $K$ if $h$ attains the infimum of 
\[
\inf\{\cE(f,f)\mid f\in\cF,\ f=h\text{ q.e.\ on }K\setminus U\}.
\]
The totality of functions in $\cF$ that are harmonic on $U$ is denoted by $\cH_U$.

For an open subset $U$ of $K$ and a subset $V$ of $U$, the relative ($0$-)capacity $\Cp(V;U)$ is defined as
\[
\Cp(V;U)=\inf\left\{\cE(f,f)\mathrel{}\middle|\mathrel{}\parbox{0.6\textwidth}{$f\in\cF$, $f\ge1$ $\mu$-a.e.\ on some open set containing $V$, and $f=0$ q.e.\ on $K\setminus U$}\right\}.
\]

For a Borel function $f$ on $K$ and a Borel subset $A$ of $K$ with $\mu(A)>0$, we define the $\mu$-oscillation of $f$ on $A$ by
\[
\muosc_A f=\muesssup_A f-\muessinf_A f.
\]

In the following, we fix a minimal energy-dominant measure $\nu$.

A collection of subsets $\{U_\lm\}_{\lm\in\Lm}$ of $K$ is called a \emph{partition} of $K$ if the following hold.
\begin{enumerate}
\item Each $U_\lm$ is a non-empty relatively compact open set of $K$.
\item The sets $\{U_\lm\}_{\lm\in \Lm}$ are disjoint in $\lm$ and $(\mu+\nu)\Bigl(K\setminus \bigsqcup\nolimits_{\lm\in \Lm}U_\lm\Bigr)=0$.
\end{enumerate}
Since $K$ is separable, $\Lm$ is at most countable.

We introduce the following assumptions.
\begin{assumption}\label{assumption}
\begin{enumerate}[({A}1)]
\item[(A1)] $\cE\not\equiv0$. That is, there exists $f\in\cF$ such that $\cE(f,f)>0$.
\item[(A2)] There exists a sequence $\{U_\lm^{(1)}\}_{\lm\in \Lm_1}$, $\{U_\lm^{(2)}\}_{\lm\in \Lm_2}$, $\{U_\lm^{(3)}\}_{\lm\in \Lm_3},\dots$ of partitions of $K$ such that the following hold.
\begin{enumerate}
\item For each $n\in\N$, $\{U_\lm^{(n+1)}\}_{\lm\in\Lm_{n+1}}$ is a refined partition of $\{U_\lm^{(n)}\}_{\lm\in\Lm_n}$ in the sense that, for each $\lm\in\Lm_{n+1}$, $U_\lm^{(n+1)}\subset U_{\lm'}^{(n)}$ for some $\lm'\in\Lm_n$.
\item There exists a metric $\sfd$ on $K$ that is compatible with its topology such that, for any compact subset $S$ of $K$,
\[
\lim_{n\to\infty}\sup_{\lm\in\Lm_n;\; U_\lm^{(n)}\cap S\ne\emptyset}\diam U_\lm^{(n)}=0,
\]
where $\diam$ denotes the diameter with respect to $\sfd$.
\end{enumerate}
\item[(A3)](Energy--relative capacity balance) There exists $C>0$ such that the following hold: For each $n\in \N$ and $\lm\in \Lm_n$, there exists a subset $V_\lm^{(n)}$ of $U_\lm^{(n)}$ that is closed in $K$ such that for any $h\in\cH_{U_\lm^{(n)}}\cap L^\infty(K,\mu)$,
\begin{enumerate} 
\item $\nu_h(U_\lm^{(n)})\le C\nu_h(V_\lm^{(n)})$,
\item $\Cp(V_\lm^{(n)};U_\lm^{(n)})\left(\muosc_{U_\lm^{(n)}} h\right)^2\le C \nu_h(U_\lm^{(n)})$,
\item $\Cp(\{x\};U_\lm^{(n)})\ge C^{-1} \Cp(V_\lm^{(n)};U_\lm^{(n)})$ for q.e.\,$x\in V_\lm^{(n)}$.
\end{enumerate}
\item[(A3')](Energy--relative capacity balance for windows) There exist $J\in\N$ and $C>0$ such that, for every $n\in\N$ and $\lm\in\Lm_n$, there exists an at most countable family
\[
\cW_\lm=\{(F_\theta,O_\theta)\}_{\theta\in\Theta_\lm}
\]
of pairs of subsets of $U_\lm^{(n)}$, called \emph{windows}, satisfying the following conditions:
\begin{enumerate}
\item For each $\theta\in\Theta_\lm$, $F_\theta$ is a compact set, $O_\theta$ is an open set, $F_\theta\subset O_\theta\subset U_\lm^{(n)}$, and $0<\Cp(F_\theta;O_\theta)<\infty$,
\item $\Cp_1\left(U_\lm^{(n)}\setminus\bigcup_{\theta\in\Theta_\lm}F_\theta\right)=0$,
\item $\sum_{\theta\in\Theta_\lm}\bfone_{O_\theta}\le J$ on $K$,
\item $\Cp(F_\theta;O_\theta)\left(\muosc_{O_\theta}h\right)^2\le C\nu_h(O_\theta)$ for every $\theta\in\Theta_\lm$ and every $h\in\cH_{U_\lm^{(n)}}\cap L^\infty(K,\mu)$,
\item For every $\theta\in\Theta_\lm$, $\Cp(\{x\};O_\theta)\ge C^{-1}\Cp(F_\theta;O_\theta)$ for q.e.\,$x\in F_\theta$.
\end{enumerate}
\end{enumerate}
\end{assumption}
The main result in this paper is stated as follows.
\begin{theorem}\label{th:main}
Assume {\rm(A1)} and {\rm(A2)} in Assumption~\ref{assumption}, and either {\rm(A3)} or {\rm(A3')}.
Then, the AF-martingale dimension $d_\mathrm{m}$ is equal to one.
\end{theorem}
Our analysis is confined to the case of one-dimensional martingale structures. 
To obtain upper bounds larger than one, the balance assumptions
and the final capacity argument would need to be reformulated, which is left for future work.
\begin{remark}
\begin{enumerate}
\item Among the assumptions, conditions (A3) and (A3') play the central role. They can be viewed as localized balance conditions between the potential-theoretic size of subsets, measured by relative capacity, and the distribution of energy of harmonic functions. More specifically, (A3)~(c) and (A3')~(e) express a scale-uniform form of non-polarity of points: the relative capacity of a point is bounded from below by a fixed proportion of that of the corresponding core or window. This lower bound is precisely what prevents points from becoming negligible under the blow-up procedure in the proof of Theorem~\ref{th:main}. Such conditions are expected to hold on a class of ``analytically low-dimensional'' state spaces; see Section~\ref{sec:example} for typical fractal examples. Neither (A3) nor (A3') obviously implies the other: (A3) provides a single distinguished core in each $U_\lm^{(n)}$ without any covering requirement, whereas (A3') uses a covering family of windows and imposes the capacity comparisons only window-wise.

\item One could also consider modifying the conditions (A2) and (A3) (or (A3')) in Assumption~\ref{assumption} by replacing $U_\lm^{(n)}$ and $V_\lm^{(n)}$ by metric balls with arbitrary center points and formulating them in such a way that the concept of partition is not used.
Such formulation is more standard in the field of analysis on metric measure spaces. There are two reasons for formulating the assumptions as (A2) and (A3) (or (A3')) here. First, they are easier to verify in typical examples. Second, this formulation naturally reflects the inductive and scale-refining structure inherent in the proof of Theorem~\ref{th:main}, which would be obscured under metric ball-based conditions.
\end{enumerate}
\end{remark}
We prepare the proof of Theorem~\ref{th:main} in the remainder of this section.

First we collect some basic properties on the energy measures for later use.
\begin{proposition}[cf.\ {\cite[Section~3.2]{FOT11}} and {\cite[Theorem~4.3.8]{CF12}}]\label{prop:sl}
\begin{enumerate}
\item For $f,g\in \cF$ and $A\in\cB$, 
\[
\left|\nu_f(A)^{1/2}-\nu_g(A)^{1/2}\right|\le\nu_{f-g}(A)^{1/2}\quad\text{and}\quad\left|\nu_{f,g}(A)\right|\le\nu_f(A)^{1/2}\nu_g(A)^{1/2}.
\]
\item For $f,g\in\cF$, $\ph\in L^2(K,\nu_f)$, and $\psi\in L^2(K,\nu_g)$, it holds that $\ph\psi\in L^1(K,|\nu_{f,g}|)$ and
\[
\left|\int_K \ph\psi\,\dd\nu_{f,g}\right|\le\left(\int_K \ph^2\,\dd\nu_f\right)^{1/2}\left(\int_K \psi^2\,\dd\nu_g\right)^{1/2}.
\]
\item $\nu_f(K)=2\cE(f,f)$ for $f\in\cF$. Thus, $\nu_{f,g}(K)=2\cE(f,g)$ for $f,g\in \cF$.
\item If $f\in\cF$ is constant q.e.\ on a Borel set $A$, then $\nu_f(A)=0$.
\item For $m\in\N$, $\Phi\in C^1_b(\R^m)$ with $\Phi(0,\dots,0)=0$, and $f_1,\dots, f_m,g\in\cF$, it holds that $\Phi(f_1,\dots,f_m)\in\cF$ and 
\begin{equation}\label{eq:derivation}
 \dd \nu_{\Phi(f_1,\dots,f_m),g}=\sum_{i=1}^m\frac{\partial\Phi}{\partial x_i}(f_1,\dots,f_m)\,\dd\nu_{f_i,g}.
\end{equation}
\end{enumerate}
\end{proposition}
We should remark that the strong locality of $(\cE,\cF)$ is crucial for Proposition~\ref{prop:sl}~(iii)--(v). We also note that taking a quasi-continuous modification of $f_i$ is important in the expression of \eqref{eq:derivation}.

For an open subset $U$ of $K$ and $f\in\cF$, we define
\[
\cH_U(f)=\{h\in\cH_U\mid h=f \text{ q.e.\ on $K\setminus U$}\}.
\]
In the following five results, we suppose that $a$, $b$, and $c$ are real numbers and $U$ is a relatively compact open subset of $K$.
\begin{lemma}\label{lem:3.1}
Suppose that $a\le 0\le b$, $f\in\cF$ and $a\le f\le b$ q.e.\ on $K\setminus U$.
Then, there exists $h\in \cH_U(f)$ such that $a\le h\le b$ q.e.\ on $K$.
\end{lemma}
\begin{proof}
Let $E=\inf\{\cE(g,g)\mid g\in\cF\text{ and $g=f$ q.e.\ on $K\setminus U$}\}$.
Take a sequence $\{g_m\}_{m=1}^\infty$ from $\cF$ such that $g_m=f$ q.e.\ on $K\setminus U$ for each $m$ and $\lim_{m\to\infty}\cE(g_m,g_m)=E$.
Let $h_m=(a\vee g_m)\wg b$. Then, $h_m=f$ q.e.\ on $K\setminus U$ and $\lim_{m\to\infty}\cE(h_m,h_m)=E$ from the Markovian property of $(\cE,\cF)$ (see, e.g., \cite[Theorem~1.4.1]{FOT11}).
Moreover, $\{h_m\}_{m=1}^\infty$ is bounded in $L^2(K,\mu)$. Indeed,
\begin{align*}
\|h_m\|_{L^2(K,\mu)}^2&=\int_U h_m^2\,\dd\mu+\int_{K\setminus U} h_m^2\,\dd\mu\\
&\le (a^2\vee b^2)\mu(U)+\|f\|_{L^2(K,\mu)}^2.
\end{align*}
Therefore, $\{h_m\}_{m=1}^\infty$ is bounded in $\cF$.
By the Banach--Alaoglu theorem, we can take a subsequence of $\{h_m\}_{m=1}^\infty$ converging weakly in $\cF$.
From the Banach--Saks theorem, the Ces\`aro mean of a further subsequence converges in $\cF$. Then, a subsequence converges q.e.\ from \cite[Theorem~2.1.4]{FOT11}. Its limit $h$ satisfies that $\cE(h,h)=E$ and $a\le h\le b$ q.e. Therefore, this $h$ satisfies the desired property.
\end{proof}
\begin{lemma}\label{lem:3.2}
Let $h\in\cF$. Then, $h$ is harmonic on $U$ if and only if $\cE(h,g)=0$ for every $g\in\cF$ with $g=0$ q.e.\ on $K\setminus U$.
\end{lemma}
This is standard and its proof is omitted.
\begin{lemma}\label{lem:3.3}
Let $f,g\in\cF$. If $f=c$ $\mu$-a.e.\ on $U$ and $g=0$ q.e.\ on $K\setminus U$, then $\cE(f,g)=0$.
\end{lemma}
\begin{proof}
By Lemma~\ref{lem:aeqe}, $f = c$ q.e.\ on $U$. From Proposition~\ref{prop:sl}~(iv), $\nu_f(U) = 0$ and $\nu_g(K\setminus U) = 0$ hold. By Proposition~\ref{prop:sl}~(i), $|\nu_{f,g}(U)|\le\nu_f(U)^{1/2}\nu_g(U)^{1/2} = 0$ and $|\nu_{f,g}(K\setminus U)| \le\nu_f(K\setminus U)^{1/2}\nu_g(K\setminus U)^{1/2}= 0$; hence $2\cE(f,g) = \nu_{f,g}(K) = 0$ by Proposition~\ref{prop:sl}~(iii).
\end{proof}
By using these lemmas, we have the following.
\begin{lemma}\label{lem:3.4}
Let $f_1, f_2\in\cF$, $f_2=f_1+c$ $\mu$-a.e.\ on $U$, and $h_1\in \cH_U(f_1)$. Then, there exists $h_2\in \cH_U(f_2)$ such that $h_2=h_1+c$ q.e.\ on $U$.
\end{lemma}
\begin{proof}
We set $h_2=h_1-f_1+f_2$.
From Lemmas~\ref{lem:3.2} and \ref{lem:3.3}, $h_2$ is harmonic on $U$.
Since $h_2=f_2$ q.e.\ on $K\setminus U$ and $h_2=h_1+c$ $\mu$-a.e.\ on $U$, $h_2$ satisfies the desired properties from Lemma~\ref{lem:aeqe}.
\end{proof}
\begin{proposition}\label{prop:3.5}
Suppose $a\le b$.
Let $f\in\cF\cap C(K)$ and $a\le f\le b$ on the boundary $\partial U$ of $U$.
Then, there exists $h\in \cH_U(f)$ such that $a\le h\le b$ q.e.\ on $U$.
\end{proposition}
\begin{proof}
First, suppose that $a\le 0\le b$. Let $\eps>0$. 
We can take a relatively compact open set $W$ such that $\partial U\subset W$ and $a-\eps\le f\le b+\eps$ on $W$. 
Let $M=\sup_{x\in W}|f(x)|<\infty$. 
There exists $\ph\in\cF\cap C(K)$ such that $0\le\ph\le 1$ on $K$, $\ph=1$ on the closure $\overline U$ of $U$, and $\ph=0$ on $K\setminus (U\cup W)$.
Also, there exists $\psi\in\cF\cap C(K)$ such that $0\le\psi\le 1$ on $K$, $\psi=1$ on $U\setminus W$, and $\psi=0$ on $K\setminus U$.
Define $\hat f=\{(-M)\vee f\wedge M\}\cdot\ph\cdot(1-\psi)\in\cF$.
Then, $a-\eps\le \hat f\le b+\eps$ on $K\setminus U$.
From Lemma~\ref{lem:3.1}, there exists $\hat h\in \cH_U(\hat f)$ such that $a-\eps\le \hat h\le b+\eps$ q.e.\ on $K$.
Since $\hat f=f\ph$ on $K\setminus U$, $\hat h\in \cH_U(f\ph)$ also holds.
Since $f\ph=f$ on $U$, Lemma~\ref{lem:3.4} with $c=0$ ensures that there exists $h_\eps\in \cH_U(f)$ such that $h_\eps=\hat h$ q.e.\ on $U$. In particular, $a-\eps\le h_\eps\le b+\eps$ q.e.\ on $U$.
Since the sequence $\{h_{1/l}\}_{l=1}^\infty$ is bounded in $\cF$, there exists a subsequence that converges weakly to some $h\in\cF$. Since $h$ is also given by the Ces\`aro mean of a further subsequence, $h\in \cH_U(f)$ and $a\le h\le b$ q.e.\ on $U$.

Next, suppose $0\le a\le b$.
Take $\ph\in\cF\cap C(K)$ such that $0\le\ph\le1$ on $K$ and $\ph=1$ on $\overline U$.
Let $g=f-a\ph$. Since $0\le g\le b-a$ on $\partial U$, there exists $\hat h\in \cH_U(g)$ such that $0\le \hat h\le b-a$ q.e.\ on $U$ from the result in the first paragraph.
Let $h=\hat h+a\ph$. Then $h\in \cH_U(f)$ and $a\le h\le b$ q.e.\ on $U$.

When $a\le b\le0$, just apply the above case for $-f$ in place of $f$.
\end{proof}
\begin{lemma}\label{lem:A2}
Under Assumption~{\rm (A2)}, the following hold.
\begin{enumerate}
\item The $\sg$-field $\cB'$ generated by $\{U_\lm^{(n)}\mid n\in\N,\ \lm\in \Lm_n\}$ coincides with the Borel $\sg$-field $\cB$ of $K$ up to $(\mu+\nu)$-null sets. That is, for any $A\in\cB$, 
there exists $A'\in\cB'$ such that $A\triangle A'$ is $(\mu+\nu)$-null.
\item For any compact subset $S$ of $K$, there exists $n_0\in \N$ such that 
\[
\mu\left(\bigcup_{\lm\in\Lm_{n_0};\; \overline{U_\lm^{(n_0)}}\cap S\ne\emptyset}U_\lm^{(n_0)}\right)<\infty.
\]
\end{enumerate}
\end{lemma}
\begin{proof}
(i) Let $O$ be an arbitrary open subset of $K$.
There exist compact sets $\{K_j\}_{j\in\N}$ included in $O$ such that $(\mu+\nu)\left(O\setminus \bigcup_{j\in\N}K_j\right)=0$.
Let $j\in\N$. Since the distance between $K_j$ and $K\setminus O$ is strictly positive, from (A2), $A_j:=\bigcup_{\lm\in\Lm_{n(j)};\; U_\lm^{(n(j))}\cap K_j\ne\emptyset}U_\lm^{(n(j))}$ satisfies $K_j\subset A_j\subset O$ up to $(\mu+\nu)$-null set for sufficiently large $n(j)\in\N$.
Therefore, $O$ is equal to $\bigcup_{j\in\N}A_j$ up to $(\mu+\nu)$-null set. This implies the assertion.

(ii) Since $K$ is locally compact, there exists a relatively compact open set $O$ such that $S\subset O$. Note that $\mu(O)<\infty$. 
Since the distance of $S$ and $K\setminus O$ is strictly positive, (A2) implies that $\bigcup_{\lm\in\Lm_n;\; U_\lm^{(n)}\cap S\ne\emptyset}\overline{U_\lm^{(n)}}\subset O$ for sufficiently large $n$, which has a $\mu$-finite mass.
\end{proof}

Under Assumption~(A2), for $n\in\N$, let 
\[
\cH_n=\{h\in\cF\mid \text{$h$ is harmonic on $U_\lm^{(n)}$ for each $\lm\in \Lambda_n$}\}\left(=\bigcap_{\lm\in\Lm_n}\cH_{U_\lm^{(n)}}\right)
\]
and $\cH_*=\bigcup_{n\in\N} \cH_n$.
Note that $\cH_n\subset\cH_{n+1}$ for $n\in\N$ by (A2)~(a) and Lemma~\ref{lem:3.2}.
\begin{proposition}\label{prop:dense}
Assume {\rm(A2)}. Then,
$\cH_*\cap L^\infty(K,\mu)$ is dense in $\cF$.
\end{proposition}
\begin{proof}
Let $f\in\cF\cap C_c(K)$. It suffices to prove that $f$ can be approximated in $\cF$ by elements of $\cH_*\cap L^\infty(K,\mu)$.
Let $n_0\in\N$ be taken as in Lemma~\ref{lem:A2}~(ii) with $S=\supp f$. Let $n\ge n_0$. 
Note that $\Lm_n$ is at most countable. Since the situation is simpler when $\Lm_n$ is a finite set, we discuss the case when $\Lm_n$ is an infinite set. We may suppose $\Lm_n=\N$.
For each $k\in \N$, we take $g_k\in \cH_{U^{(n)}_k}(f)$ such that
\[
\inf_{\partial U^{(n)}_k}f\le \muessinf_{U^{(n)}_k}g_k\le \muesssup_{U^{(n)}_k}g_k\le \sup_{\partial U^{(n)}_k}f
\quad\text{if $\partial U^{(n)}_k\ne\emptyset$}
\]
and
\[
g_k=\fint_{U^{(n)}_k}f\,\dd\mu\text{ on }U^{(n)}_k
\quad\text{if $\partial U^{(n)}_k=\emptyset$}.
\]
Such choices are possible from Proposition~\ref{prop:3.5}.

Let $h_0=f$ and for $k\in\N$ take $h_k\in \cH_{U^{(n)}_k}(h_{k-1})$ such that $h_k=g_k$ on $U^{(n)}_k$, inductively.
This is possible from Lemma~\ref{lem:3.4} with $c=0$ and $h_{k-1}=f$ $\mu$-a.e.\ on $U^{(n)}_k$ for each $k$.
Then,
\begin{align*}
\|h_k\|_{L^\infty(K,\mu)}&\le \|f\|_\infty,\\
\|h_k\|_{L^2(K,\mu)}^2&\le \|f\|_\infty^2\mu\left(\bigcup_{l\in\Lm_{n_0};\; \overline{U_l^{(n_0)}}\cap {\supp f\ne\emptyset}}U_l^{(n_0)}\right)=:M<\infty,
\end{align*}
and
\[
\cE(h_k,h_k)\le\cE(h_{k-1},h_{k-1})\le\cdots\le\cE(f,f).
\]
Therefore, $\{h_k\}_{k=1}^\infty$ is bounded both in $L^\infty(K,\mu)$ and in $\cF$.
Since $h_k$ converges $\mu$-a.e.\ as $k\to\infty$, its limit, say $h^{(n)}$, is the weak limit of $\{h_k\}_{k=1}^\infty$ in $\cF$.
Then, $\|h^{(n)}\|_{L^\infty(K,\mu)}\le \|f\|_\infty$ and
\begin{equation}\label{eq:bound0}
\cE(h^{(n)},h^{(n)})+\|h^{(n)}\|_{L^2(K,\mu)}^2\le \cE(f,f)+M.
\end{equation}
For each $k\in\N$, $h^{(n)}=h_k=g_k$ $\mu$-a.e.\ on $U^{(n)}_k$.
In particular, $h^{(n)}\in\cH_n$ (by Lemmas~\ref{lem:3.2} and \ref{lem:3.3}), 
\begin{equation}\label{eq:bound1}
\inf_{\partial U^{(n)}_k}f\le \muessinf_{U^{(n)}_k}h^{(n)}\le \muesssup_{U^{(n)}_k}h^{(n)}\le \sup_{\partial U^{(n)}_k}f
\quad\text{if $\partial U^{(n)}_k\ne\emptyset$}
\end{equation}
and
\begin{equation}\label{eq:bound2}
h^{(n)}=\fint_{U^{(n)}_k}f\,\dd\mu\text{ on }U^{(n)}_k
\quad\text{if $\partial U^{(n)}_k=\emptyset$}.
\end{equation}
Note that the complement of $\bigcap_{n\in\N}\bigcup_{k\in\Lm_n}U^{(n)}_k$ is a $\mu$-null set.
Suppose $x\in\bigcap_{n\in\N}\bigcup_{k\in\Lm_n}U^{(n)}_k$.
For each $n$, the unique $k\in \Lm_n$ such that $x\in U^{(n)}_k$ will be denoted by $k_n$.
Then, $U^{(1)}_{k_1}\supset U^{(2)}_{k_2}\supset\cdots\supset U^{(n)}_{k_n}\supset\cdots$. 
From (A2) with $S=\overline{U^{(1)}_{k_1}}$, the diameter of $U^{(n)}_{k_n}$ with respect to the metric $\sfd$ in (A2)~(b) converges to $0$.
By keeping \eqref{eq:bound1} and \eqref{eq:bound2} in mind, $h^{(n)}(x)\to f(x)$ as $n\to\infty$ for $\mu$-a.e.\,$x$.
Therefore, $h^{(n)}\to f$ in $L^2(K,\mu)$ as $n\to\infty$.

Since $\{h^{(n)}\}_{n=n_0}^\infty$ is bounded in $\cF$ from \eqref{eq:bound0}, $h^{(n)}$ converges to $f$ weakly in $\cF$. Therefore, the Ces\`aro means of some subsequence of $\{h^{(n)}\}_{n=n_0}^\infty$, which belong to $\cH_*\cap L^\infty(K,\mu)$, converge to $f$ in $\cF$.
\end{proof}

We provide a few propositions on the capacity.
\begin{proposition}\label{prop:cap}
Let $U$ be an open subset of $K$, and let $W$ be a compact subset of $U$.
Let $\{A_n\}_{n=1}^\infty$ be a sequence of subsets of $W$. If $\lim_{n\to\infty}\Cp_1(A_n)=0$, then $\lim_{n\to\infty}\Cp(A_n;U)=0$.
\end{proposition}
\begin{proof}
Take $\chi\in\cF\cap C(K)$ such that the support of $\chi$ is included in $U$, $0\le \chi\le1$, and $\chi=1$ on an open set including $W$. There exists a sequence $\{u_n\}_{n=1}^\infty$ of $\cF$ such that $0\le u_n\le1$ $\mu$-a.e.\ on $K$ and $u_n=1$ $\mu$-a.e.\ on an open set including $A_n$ for each $n$, and $\lim_{n\to\infty}\bigl(\cE(u_n,u_n)+\|u_n\|_{L^2(K,\mu)}^2\bigr)=0$. Then, $\Cp(A_n;U)\le \cE(\chi u_n,\chi u_n)$ for each $n$.
Let $\{n_k\}_{k=1}^\infty$ be an arbitrary divergent sequence of natural numbers. From \cite[Theorem~2.1.4]{FOT11}, there exists a subsequence $\{n_{k'}\}$ such that $u_{n_{k'}}\to0$ q.e.\ as $k'\to\infty$.
Then, $\int_K u_{n_{k'}}^2\,\dd\nu_\chi\to0$ as $k'\to\infty$.
Therefore,
\begin{align*}
2\cE(\chi u_{n_{k'}},\chi u_{n_{k'}})
&=\int_K u_{n_{k'}}^2\,\dd\nu_\chi+2\int_K \chi u_{n_{k'}}\,\dd\nu_{\chi,u_{n_{k'}}}+\int_K \chi^2\,\dd\nu_{u_{n_{k'}}}\\
&\le \left\{\left(\int_K u_{n_{k'}}^2\,\dd\nu_\chi\right)^{1/2}+\left(\int_K \chi^2\,\dd\nu_{u_{n_{k'}}}\right)^{1/2}\right\}^2\\
&\le \left\{\left(\int_K u_{n_{k'}}^2\,\dd\nu_\chi\right)^{1/2}+(2\cE(u_{n_{k'}},u_{n_{k'}}))^{1/2}\right\}^2\\
&\to0\qquad\text{as $k'\to\infty$.}
\end{align*}
This means that $\lim_{n\to\infty}\cE(\chi u_n,\chi u_n)=0$, which implies the assertion.
\end{proof}
\begin{proposition}
\label{prop:qe}
Let $U$ be a relatively compact open subset of $K$ and $A$ a subset of $K$ whose closure $\overline{A}$ is included in $U$.
\begin{enumerate}
\item Suppose that $g\in\cF$ satisfies that $g\ge 1$ q.e.\ on $A$ and $g=0$ q.e.\ on $K\setminus U$.
Then, $\Cp(A;U)\le \cE(g,g)$.
\item There exists $f\in\cF$ such that $0\le f\le 1$ q.e.\ on $K$, $f=1$ q.e.\ on $A$, $f=0$ q.e.\ on $K\setminus U$, and $\cE(f,f)=\Cp(A;U)$.
\end{enumerate}
\end{proposition}
\begin{proof}
(i) From the Markovian property, we may assume that $0\le g\le1$.
Take an open subset $V$ of $U$ such that $\overline{A}\subset V\subset \overline{V}\subset U$. 
For each $n\in\N$, there exists an open subset $O_n$ of $K$ such that $\Cp_1(O_n)<1/n$, $g|_{K\setminus O_n}$ is continuous and $g=1$ on $A\setminus O_n$.
Set $B_n=O_n\cap V$.
Proposition~\ref{prop:cap}, applied with $W=\overline{V}$, gives $\Cp(B_n;U)\to0$ as $n\to\infty$.
Choose $\{e_n\}_{n=1}^\infty\subset\cF$ such that $0\le e_n\le1$ on $K$, $e_n=1$ $\mu$-a.e.\ on $B_n$ and $e_n=0$ q.e.\ on $K\setminus U$ for each $n$, and $\cE(e_n,e_n)\to0$ as $n\to\infty$. 
For $n\in\N$,
\[
  G_n:=\{x\in V\setminus B_n\mid g(x)>1-1/n\}\cup B_n
\]
is open, $A\subset G_n$, $g+e_n>1-1/n$ $\mu$-a.e.\ on $G_n$, and $g+e_n=0$ q.e.\ on $K\setminus U$.
Therefore, for $n\ge2$,
\begin{align*}
\Cp(A; U)&\le \left(1-\frac1n\right)^{-2}\cE(g+e_n,g+e_n)\\
&\le \left(1-\frac1n\right)^{-2}\left(\sqrt{\cE(g,g)}+\sqrt{\cE(e_n,e_n)}\right)^{2}.
\end{align*}
Letting $n\to\infty$, we get the conclusion.

(ii) There exists a sequence of functions $\{f_n\}_{n=1}^\infty$ in $\cF$ such that $0\le f_n\le1$ q.e.\ on $K$, $f_n=1$ $\mu$-a.e.\ on an open set including $A$, and $f_n=0$ q.e.\ on $K\setminus U$ for each $n$, and $\lim_{n\to\infty}\cE(f_n,f_n)=\Cp(A;U)$. Since $\cE(f_n,f_n)+\|f_n\|_{L^2(K,\mu)}^2$ is bounded in $n$, some subsequence $\{f_{n'}\}$ converges weakly in $\cF$. Its limit $f$ satisfies that $\cE(f,f)\le \Cp(A;U)$. Since the Ces\`aro means of some subsequence of $\{f_{n'}\}$ converges to $f$ in $\cF$, from Lemma~\ref{lem:aeqe} and \cite[Theorem~2.1.4]{FOT11}, $0\le f\le 1$ q.e.\ on $K$, $f=1$ q.e.\ on $A$, and $f=0$ q.e.\ on $K\setminus U$.
By combining this with the assertion in (i), we arrive at the conclusion.
\end{proof}
The function $f$ in Proposition~\ref{prop:qe}~(ii) will be denoted by $e_{A;U}$. Although such a function may not be unique, we write $e_{A;U}$ to describe one of such $f$, whose choice does not affect the arguments that follow.
\begin{remark}[Comments for experts in the Dirichlet form theory]
Since the part form of $(\cE,\cF)$ on $U$ is not assumed to be transient, the converse of Proposition~\ref{prop:cap} is not true in general. In particular, Proposition~\ref{prop:qe} does not follow directly from the corresponding claim with $\Cp(A;U)$ replaced by the (1-)capacity of $A$ with respect to the part form on $U$.
\end{remark}
At the end of this section, we prepare lemmas from the measure theory.
A sequence of signed Borel measures $\{\eta_k\}_{k=1}^\infty$ on a metric space $Y$ is said to converge weakly to a signed Borel measure $\eta_\infty$, if, by definition, $\int_Y f\,\dd\eta_k\to\int_Y f\,\dd\eta_\infty$ as $k\to\infty$ for every bounded continuous function $f$ on $Y$.
In what follows, we treat only the situation where $\eta_k$ concentrates on a compact subset $W$ independent of $k$. In such a case, the weak convergence is consistent with the weak-$*$ convergence by identifying the space of signed measures on $W$ with the topological dual space of $C(W)$.
\begin{lemma}\label{lem:measure}
Let $Y$ be a separable metric space.
Suppose that a sequence of finite Borel measures $\{\eta_k\}_{k=1}^\infty$ on $Y$ converges weakly to a non-zero finite measure $\eta_\infty$.
Let $a\in Y$ belong to the support of $\eta_\infty$.
Then, there exists $k_1\in\N$ and a sequence $\{a^{(k)}\}_{k=k_1}^\infty$ in $Y$ such that $a^{(k)}$ converges to $a$ as $k\to\infty$ and $a^{(k)}$ belongs to the support of $\eta_k$ for any $k\ge k_1$.
\end{lemma}
\begin{proof}
For $m\in\N$, let $B_m$ denote the open ball on $Y$ with center $a$ and radius $1/m$.
Take a sequence $\{\ph_n\}_{n=1}^\infty$ in $C(Y)$ such that each $\ph_n$ is nonnegative and $\{\ph_n\}_{n=1}^\infty$ is monotonically increasing in $n$ and converges pointwise to $\bfone_{B_m}$.
Then, 
\[
\varliminf_{k\to\infty}\eta_k(B_m)
\ge\varliminf_{k\to\infty}\int_Y \ph_n\,\dd\eta_k
=\int_Y \ph_n\,\dd\eta_\infty.
\]
By letting $n\to\infty$, $\varliminf_{k\to\infty}\eta_k(B_m)\ge\eta_\infty(B_m)>0$.
Therefore, there exists $k_m\in\N$ such that $\eta_k(B_m)>0$ for all $k\ge k_m$.
We can take $\{k_m\}_{m=1}^\infty$ so that it is strictly increasing.
For each $k\ge k_1$, there exists a unique $m\in\N$ such that $k_m\le k<k_{m+1}$.
Since $\eta_k(B_m)>0$, the support $S_k$ of $\eta_k|_{B_m}$ is non-empty.
Select a point $a^{(k)}$ from $S_k$. Then, $\{a^{(k)}\}_{k=k_1}^\infty$ satisfies the required condition.
\end{proof}
\begin{lemma}\label{lem:selection}
Assume {\rm(A3')}.
Let $n\in\N$, $\lm\in\Lm_n$, and $\eta$ be a finite Borel measure on $K$ that charges no set of zero $\Cp_1$-capacity and satisfies $\eta(U_\lm^{(n)})>0$.
Further, let $\eps>0$ and let $B$ be a Borel subset of $U_\lm^{(n)}$ with $\eta(U_\lm^{(n)}\setminus B)\le\eps\eta(U_\lm^{(n)})$.
Then, there exists $\theta\in\Theta_\lm$ such that
\[
\eta(F_\theta)>0,\quad
\eta(O_\theta)\le 3J\eta(F_\theta),\quad
\eta(O_\theta\setminus B)\le 3J\eps\eta(F_\theta).
\]
\end{lemma}

\begin{proof}
Let $\Theta_\lm^+=\{\theta\in\Theta_\lm\mid \eta(F_\theta)>0\}$.
Since $U_\lm^{(n)}=\bigcup_{\theta\in\Theta_\lm}F_\theta$ up to $\eta$-null set by (A3')~(b) and the assumption on $\eta$, we have
\begin{equation}\label{eq:theta}
\sum_{\theta\in\Theta_\lm^+}\eta(F_\theta)
=\sum_{\theta\in\Theta_\lm}\eta(F_\theta)
\ge \eta\left(\bigcup_{\theta\in\Theta_\lm} F_\theta\right)=\eta(U_\lm^{(n)}).
\end{equation}
On the other hand, from (A3')~(a)--(c),
\[
\sum_{\theta\in\Theta_\lm}\eta(O_\theta)
=\int_{U_\lm^{(n)}}\left(\sum_{\theta\in\Theta_\lm}\bfone_{O_\theta}\right)\dd\eta
\le J\eta(U_\lm^{(n)})
\]
and
\[
\sum_{\theta\in\Theta_\lm}\eta(O_\theta\setminus B)
=\int_{U_\lm^{(n)}\setminus B}\left(\sum_{\theta\in\Theta_\lm}\bfone_{O_\theta}\right)\dd\eta
\le J\eta(U_\lm^{(n)}\setminus B)
\le J\eps\eta(U_\lm^{(n)}).
\]
Therefore,
\[
  \sum_{\theta\in\Theta_\lm^+}\{\eta(O_\theta)+\eps^{-1}\eta(O_\theta\setminus B)\}\le 2J\eta(U_\lm^{(n)}).
\]
If $\eta(O_\theta)+\eps^{-1}\eta(O_\theta\setminus B)\ge 3J\eta(F_\theta)$ holds for every $\theta\in\Theta_\lm^+$, then summation over $\Theta_\lm^+$ and \eqref{eq:theta} give $2J\eta(U_\lm^{(n)})\ge 3J\eta(U_\lm^{(n)})$, which contradicts $\eta(U_\lm^{(n)})>0$. 
Therefore, there exists $\theta\in\Theta_\lm^+$ such that $\eta(O_\theta)+\eps^{-1}\eta(O_\theta\setminus B)<3J\eta(F_\theta)$. This $\theta$ satisfies the required properties.
\end{proof}

\section{Proof of the main theorem}\label{sec:main}
In this section, we prove Theorem~\ref{th:main}.
We first establish the window estimates needed for the proof of Theorem~\ref{th:main}.
\begin{lemma}\label{lem:winA}
Assume {\rm(A1)--(A3)} with the constant $C$. Let $n\in\N$ and
$\lm\in\Lambda_n$, and set $(F,O):=(V_\lm^{(n)},U_\lm^{(n)})$. 
Let $d\in\N$ and $\bfh=(h_1,\dots,h_d)\in\bigl(\cH_{U_\lm^{(n)}}\cap L^\infty(K,\mu)\bigr)^d$, and suppose that $\nu_\bfh(O)>0$, where
$\nu_\bfh:=(1/d)\sum_{i=1}^d\nu_{h_i}$. Then, the following
hold.
\begin{enumerate}
\item $F$ is a compact subset of the open set $O$, and
$0<\Cp(F;O)<\infty$.
\item $\nu_\bfh(F)\ge C^{-1}\nu_\bfh(O)>0$.
Moreover, if a Borel subset $B$ of $O$ and $\eps>0$ satisfy
$\nu_\bfh(O\setminus B)\le\eps\nu_\bfh(O)$,
then $\nu_\bfh(O\setminus B)\le C\eps\nu_\bfh(F)$.
\item $\Cp(F;O)\bigl(\muosc_O h_i\bigr)^2\le C\nu_{h_i}(O)$ for every $i=1,\dots,d$.
\item $\Cp(\{x\};O)\ge C^{-1}\Cp(F;O)$ for q.e.\,$x\in F$.
\end{enumerate}
\end{lemma}
\begin{proof}
Assertion~(iii) is (A3)~(b), and the first inequality of (ii) follows by applying (A3)~(a) to each $h_i$ and averaging; the remaining assertions of (ii) are immediate consequences.
Assertion~(iv) is nothing but (A3)~(c).
We show (i). The only nontrivial part is the strict positivity of $\Cp(F;O)$.
Assume that $\Cp(F;O)=0$. From Proposition~\ref{prop:qe}~(ii) $\cE(e_{F;O},e_{F;O})=0$. Then, $\nu_{e_{F;O}}$ is a zero measure.
Take $i$ with $\nu_{h_i}(O)>0$ and $\chi\in C_b^1(\R)$ with $\chi(t)=t$ on
$[-\|h_i\|_{L^\infty(K,\mu)}-1,\ \|h_i\|_{L^\infty(K,\mu)}+1]$. By Lemma~\ref{lem:3.2},
$\cE(h_i,e_{F;O}h_i)=0$. On the other hand, Proposition~\ref{prop:sl}~(v) with
$\Phi(x,y)=\chi(x)\chi(y)$ gives
\begin{align*}
2\cE(h_i,e_{F;O}h_i)
&=\int_K e_{F;O}\,\dd\nu_{h_i}+\int_K h_i\,\dd\nu_{h_i,e_{F;O}}\\
&=\int_K e_{F;O}\,\dd\nu_{h_i}\ge\nu_{h_i}(F)\ge C^{-1}\nu_{h_i}(O)>0,
\end{align*}
since $\nu_{h_i,e_{F;O}}=0$ by Proposition~\ref{prop:sl}~(i).
This is a contradiction.
\end{proof}
\begin{lemma}\label{lem:winB}
Assume {\rm(A1)}, {\rm(A2)}, and {\rm(A3')} with the constants $J$ and $C$. Let
$n\in\N$ and $\lm\in\Lambda_n$, and write
$U:=U_\lm^{(n)}$. Let $d\in\N$ and
$\bfh=(h_1,\dots,h_d)\in(\cH_U\cap L^\infty(K,\mu))^d$ with
$\nu_\bfh(U)>0$, where $\nu_\bfh:=(1/d)\sum_{i=1}^d\nu_{h_i}$. Furthermore, let $\eps>0$ and let $B$ be a Borel subset of $U$ with
$\nu_\bfh(U\setminus B)\le\eps\nu_\bfh(U)$. Then, there exists
$(F,O)\in\cW_\lm$ such that the following hold.
\begin{enumerate}
\item $F$ is a compact subset of the open set $O$,
$O\subset U$, and $0<\Cp(F;O)<\infty$.
\item $\nu_\bfh(F)>0$, $\nu_\bfh(O)\le 3J\nu_\bfh(F)$, and $\nu_\bfh(O\setminus B)\le 3J\eps\nu_\bfh(F)$.
\item $\Cp(F;O)\bigl(\muosc_O h_i\bigr)^2\le C\nu_{h_i}(O)$ for every $i=1,\dots,d$.
\item $\Cp(\{x\};O)\ge C^{-1}\Cp(F;O)$ for q.e.\,$x\in F$.
\end{enumerate}
\end{lemma}

\begin{proof}
The measure $\nu_\bfh$ is a finite Borel measure charging no set of zero capacity. 
Applying Lemma~\ref{lem:selection} to the family $\cW_\lm$ with
$\eta:=\nu_\bfh$, we obtain $(F,O)\in\cW_\lm$ satisfying (ii). Assertion~(i) follows from (A3')~(a),
(iii) is (A3')~(d) applied to each $h_i\in\cH_U$, and (iv) is (A3')~(e).
\end{proof}

The following is the key proposition for proving Theorem~\ref{th:main}. The functions $\bfh^{(k)}$ there can be thought of as partial coordinate systems normalized to balance the total energy measure. The weak convergence properties of their energy interactions encapsulate the obstruction to having (A3)~(c) or (A3')~(e).

We set
\[
C_0=\begin{cases}
C&\text{when (A3) is assumed,}\\
3J&\text{when (A3') is assumed.}
\end{cases}
\]
\begin{proposition}\label{prop:key}
We assume {\rm(A1)} and {\rm(A2)} in Assumption~\ref{assumption}, and either {\rm(A3)} or {\rm(A3')}. 
Suppose $d\in\N$ satisfies $d\le d_\mathrm{m}$.
Then, there exist a strictly increasing sequence $\{n_k\}_{k=1}^\infty$ of natural numbers, $\lm_k\in \Lm_{n_k}$, pairs $(F_k,O_k)$ of subsets of $U_{\lm_k}^{(n_k)}$, and $\bfh^{(k)}=(h_1^{(k)},h_2^{(k)},\dots,h_d^{(k)})\in(\cH_{U_{\lm_k}^{(n_k)}})^d$ for $k\in\N$ such that the following hold.
\begin{itemize}
\item Under {\rm(A3)}, $(F_k,O_k)=(V_{\lm_k}^{(n_k)},U_{\lm_k}^{(n_k)})$; under {\rm(A3')}, $(F_k,O_k)\in\cW_{\lm_k}$.
In both cases, $F_k$ is a compact subset of the open set $O_k$ and $0<\Cp(F_k;O_k)<\infty$.
\item Every $h_i^{(k)}$ $(i=1,2,\dots,d,\ k\in\N)$ is a bounded function, and $|\b_k\bfh^{(k)}|\le R_0$ q.e.\ on $O_k$, where $R_0=d\sqrt{C_0C}$. 
\item For each $k\in\N$, $\nu_{\bfh^{(k)}}(F_k)=1$ and $\nu_{\bfh^{(k)}}(O_k)\le C_0$.
\item For each $k\in\N$, $\Cp(\{x\};O_k)\ge C^{-1}\Cp(F_k;O_k)$ for q.e.\,$x\in F_k$.
\item For each $i,j=1,2,\dots,d$, the measure $\bigl(\b_k\bfh^{(k)}|_{O_k}\bigr)_*\bigl|e_k^2 \nu_{h_i^{(k)},h_j^{(k)}}\bigr|$ converges weakly to
\[
\begin{cases}
\xi(x)\,\cL^d(\dd x)& (i=j)\\
0 &(i\ne j)
\end{cases}
\]
as $k\to\infty$ for some $\xi\in L^1(\R^d,\cL^d)$ (independent of $i$) with $\sqrt{\xi}\in W^{1,2}(\R^d)$.
In particular, the same limits hold for the corresponding signed measures $\bigl(\b_k\bfh^{(k)}|_{O_k}\bigr)_*\bigl(e_k^2 \nu_{h_i^{(k)},h_j^{(k)}}\bigr)$.
\item The measure $\bigl(\b_k\bfh^{(k)}|_{O_k}\bigr)_*\bigl(\b_k^{-2}\nu_{e_k}\bigr)$ on $\R^d$ converges weakly to some finite Borel measure $\tilde\kp$ as $k\to\infty$.
\end{itemize}
Here, 
\begin{align}
\nu_{\bfh^{(k)}}&=\frac1d \sum_{i=1}^d \nu_{h_i^{(k)}},\label{eq:nubfh}\\
\b_k&=\Cp(F_k;O_k)^{1/2},\label{eq:bkdef}\\
\shortintertext{and}
e_k&=e_{F_k;O_k}.\label{eq:ekdef}
\end{align}
\end{proposition}
\begin{proof}
We divide the proof into six steps.
\smallskip

\textbf{Step 1 (Selecting harmonic functions $\{f_i\}_{i=1}^d$). }We first construct a family of bounded harmonic functions $f_1,f_2,\dots,f_d$ that exhibit maximal linear independence in the sense of energy measures. These functions will serve as a local coordinate system in the blow-up argument below.
 
We fix a minimal energy-dominant measure $\nu$ for $(\cE,\cF)$.
We take a sequence of functions $\{\hat f_i\}_{i=1}^\infty$ from $\cH_*\cap L^\infty(K,\mu)$ such that $\{\hat f_i\mid i\in\N\}$ is dense in $\cF$. This is possible from Proposition~\ref{prop:dense} and the separability of $\cF$.
For each $i,j\in\N$, let $\hat Z_{i,j}$ denote a $\nu$-version of the Radon--Nikodym derivative $\dd\nu_{\hat f_i,\hat f_j}/\dd\nu$. 
We may assume that $(\hat Z_{i,j}(x))_{i,j=1}^N$ is symmetric and non-negative definite matrix for every $x\in K$ and every $N\in\N$.
From \cite[Proposition~2.10]{Hi10}, 
\[
d_{\mathrm{m}}=\nuesssup_{x\in K} \sup_{N\in\N}\rank(\hat Z_{i,j}(x))_{i,j=1}^N.
\]
Since $d\le d_{\mathrm{m}}$, there exists $N\in\N$ such that $\nu(\{x\in K\mid\rank(\hat Z_{i,j}(x))_{i,j=1}^N\ge d\})>0$. Therefore, there exist $1\le \a_1<\a_2<\dots<\a_d\le N$ such that $\nu(\hat B)>0$ with $\hat B=\{x\in K\mid$\break $ (\hat Z_{\a_i,\a_j}(x))_{i,j=1}^d\text{ is invertible}\}$.
We write $Z_{i,j}$ for $\hat Z_{\a_i,\a_j}$ and $f_i$ for $\hat f_{\a_i}$, respectively. There exists some $m_0\in\N$ such that $f_i\in \cH_{m_0}$ for every $i=1,2,\dots,d$.
We write $\bff=(f_1,f_2,\dots,f_d)$ and $\nu_\bff=(1/d)\sum_{i=1}^d\nu_{f_i}$. Then,
\[
\frac{\dd\nu_\bff}{\dd\nu}(x)=\frac1d \tr\bigl(Z_{i,j}(x)\bigr)_{i,j=1}^d,\quad \nu\text{-a.e.}\,x.
\]
The right-hand side is positive if $\bigl(Z_{i,j}(x)\bigr)_{i,j=1}^d$ is invertible. Therefore, $\hat B\subset \{\dd\nu_\bff/\dd\nu>0\}$ up to $\nu$-null set. Thus, $\nu_\bff(\hat B)>0$ from $\nu(\hat B)>0$.
Let
\[
\Mat(d):=\{\text{all real square matrices of size $d$}\}\cong\R^{d\times d}.
\]
We define a $\Mat(d)$-valued function $\Phi_\bff$ on $K$ as
\[
\Phi_\bff:=\begin{cases}
\left(\dfrac{\dd\nu_{f_i,f_j}}{\dd\nu_\bff}\right)_{i,j=1}^d&\text{on }\left\{\dfrac{\dd\nu_\bff}{\dd\nu}>0\right\}\\
O&\text{otherwise.}
\end{cases}
\]
Then, since
\[
\Phi_\bff=\left(Z_{i,j}(x)\middle/\frac{\dd\nu_\bff}{\dd\nu}(x)\right)_{i,j=1}^d
\quad\text{on }\left\{\frac{\dd\nu_\bff}{\dd\nu}>0\right\},
\]
it holds that $\nu_\bff(\{x\in K\mid\Phi_\bff(x)\text{ is invertible}\})>0$.

Take $a>0$ such that $\nu_\bff(\{x\in K\mid \det\Phi_\bff(x)\ge a\})>0$.
Let $B=\{x\in K\mid \det\Phi_\bff(x)\ge a\}\setminus K_*$, where $K_*:=\bigcup_{n\in\N}\left(K\setminus \bigcup_{\lm\in \Lm_n}U^{(n)}_\lm\right)$ is a $(\mu+\nu)$-null set. Then, $\nu_\bff(B)>0$.
Take an element $L$ from the support of the non-zero measure $(\Phi_\bff|_B)_*(\nu_\bff|_B)$ on $\Mat(d)$.
We record a property of $L$ that will be used below. Let
\[
S_a := \{ A\in \Mat(d) \mid \text{$A$ is symmetric and non-negative definite, and $\det A \ge a$}\}.
\]
This is a closed subset of $\Mat(d)$, and
$\Phi_\bff(B)\subset S_a$, so that the measure $(\Phi_\bff|_B)_*(\nu_\bff|_B)$ is concentrated on the closed set $S_a$ and its support is contained in $S_a$. In particular, $L$ is symmetric with $\det L \ge a > 0$ and non-negative definite, hence positive-definite.
\smallskip

\textbf{Step 2 (Reselection of $\{f_i\}_{i=1}^d$). }We will prove that we may assume $L=I$, the identity matrix, by changing $f_1,\dots,f_d$ suitably.
Since $L$ is symmetric and positive-definite as shown above, there exists an orthogonal matrix $\Gm=(\gm_{ij})_{i,j=1}^d$ diagonalizing $L$ such that
\[
\Gm^{\top} L\Gm=\begin{pmatrix}\lm_1&&\raisebox{-1.7ex}[0pt][0pt]{\phantom{$\lm_d$}\llap{\smash{\huge$0$}}}\\&\ddots\\\mbox{\rlap{\smash{\huge$0$}}\quad}&&\lm_d\end{pmatrix}\quad\mbox{with }\lm_i>0,\ i=1,2,\dots,d.
\]
Define $\check\bff=(\check f_1,\check f_2,\dots,\check f_d)\in(\cH_{m_0})^d$ by
$
  \check f_i=\lm_i^{-1/2}\sum_{k=1}^d \gm_{ki}f_k$ for $i=1,2,\dots,d$.
Then,
\[
  \nu_{\check f_i,\check f_j}=\lm_i^{-1/2}\lm_j^{-1/2}\sum_{k,l=1}^d \gm_{ki}\gm_{lj}\,\nu_{f_k,f_l},\qquad\text{$i,j=1,2,\dots,d$},
\]
which implies that
\[
 \left(\frac{\dd\nu_{\check f_i,\check f_j}}{\dd\nu_{\bff}}(x)\right)_{i,j=1}^d
 =\Lm^{-1/2}\Gm^{\top}\Phi_\bff(x) \Gm\Lm^{-1/2},\quad\text{$\nu_{\bff}\mbox{-a.e.}\,x$}.
\]
Here, $\Lm^{-1/2}$ denotes the diagonal matrix whose diagonal elements are $\lm_1^{-1/2},\lm_2^{-1/2},\break\dots,\lm_d^{-1/2}$.

For $M\in \Mat(d)$ and $r>0$, define
\[
D(M,r)=\left\{A\in \Mat(d)\mid \|A-M\|_{\HS}<r\right\}.
\]
By noting the identity $\Lm^{-1/2}\Gm^{\top}L \Gm\Lm^{-1/2}=I$, for each $\eps\in(0,1)$ there exists $\dl>0$ such that
\begin{equation}\label{eq:L1}
  \left(\frac{\dd\nu_{\check f_i,\check f_j}}{\dd\nu_{\bff}}(x)\right)_{i,j=1}^d\in D(I,\eps)\quad \text{for every }x\in(\Phi_\bff)^{-1}(D(L,\dl)).
\end{equation}
In particular, for $x\in (\Phi_\bff)^{-1}(D(L,\dl))$,
\[
\frac{\dd\nu_{\check f_i}}{\dd\nu_{\bff}}(x)\in(1-\eps,1+\eps),
\qquad i=1,2,\dots,d,
\]
which implies 
\begin{equation}\label{eq:L2}
\frac{\dd\nu_{\check \bff}}{\dd\nu_{\bff}}(x)\in(1-\eps,1+\eps),
\end{equation}
where $\nu_{\check \bff}=(1/d)\sum_{i=1}^d \nu_{\check f_i}$.
Note that $\nu_{\check \bff}\bigl((\Phi_\bff)^{-1}(D(L,\dl))\bigr)>0$ due to the choice of $L$.
Combining \eqref{eq:L1}, \eqref{eq:L2}, and the inequality
\[
 \|\alpha^{-1}(I+A)-I\|_\HS\le |\a^{-1}-1|\sqrt{d}+\a^{-1}\|A\|_\HS,
 \qquad A\in\Mat(d),\ \a>0,
\]
we can confirm that
\begin{align}
\Phi_{\check\bff}(x)&:=\left(\frac{\dd\nu_{\check f_i,\check f_j}}{\dd\nu_{\check\bff}}(x)\right)_{i,j=1}^d
=\left(\frac{\dd\nu_{\check \bff}}{\dd\nu_{\bff}}(x)\right)^{-1}\left(\frac{\dd\nu_{\check f_i,\check f_j}}{\dd\nu_{\bff}}(x)\right)_{i,j=1}^d\nonumber\\*
&\in D\left(I,\left[\left\{(1-\eps)^{-1}-1\}\vee\{1-(1+\eps)^{-1}\right\}\right]\sqrt{d}+(1-\eps)^{-1}\eps\right),\nonumber\\*
&\hspace{15em}x\in(\Phi_\bff)^{-1}(D(L,\dl)).
\label{eq:L3}
\end{align}
Since $\eps>0$ is arbitrary, \eqref{eq:L2} and \eqref{eq:L3} imply that $I$ belongs to the support of the non-zero measure $(\Phi_{\check\bff})_*\nu_{\check\bff}$.
Therefore, by considering $\check f_1,\check f_2,\dots,\check f_d$ instead of $f_1,f_2,\dots,f_d$ and reselecting $a$ ($a=1/2$ will suffice), we may  assume that $L=I$. In what follows, we always assume $L=I$.
\smallskip

\textbf{Step 3 (Selection of the domains $U_{\lm_k}^{(n_k)}$). }
For $k\in\N$, let 
\begin{equation}\label{eq:Bk}
B_k:=(\Phi_\bff|_B)^{-1}(D(I,1/k))\subset B.
\end{equation}
Then, $\nu_\bff(B_k)=\bigl((\Phi_\bff|_B)_*(\nu_\bff|_B)\bigr)(D(I,1/k))>0$.

For $n\in\N$, let $\cG^{(n)}$ be the sub $\sg$-field on $K$ generated by $\{U_\lm^{(m)}\mid 1\le m\le n,\ \lm\in \Lm_m\}$.
We define a function $Y^{(k,n)}$ on $K$ as
\[
Y^{(k,n)}(x):=\begin{cases}
\dfrac{\nu_\bff(U_\lm^{(n)}\cap B_k)}{\nu_\bff(U_\lm^{(n)})} & \text{if $x\in U_\lm^{(n)}$ and $\nu_\bff(U_\lm^{(n)})>0$ for some $\lm\in\Lm_n$}\\
0& \text{otherwise.}
\end{cases}
\]
From the refinement property (A2)~(a), $Y^{(k,n)}$ is equal to the conditional expectation of $\bfone_{B_k}$ given $\cG^{(n)}$ with respect to $\bar\nu_\bff:=\nu_\bff(K)^{-1}\nu_\bff$. From Assumption (A2)~(a), $\{Y^{(k,n)}\}_{n=1}^\infty$ is a $\{\cG^{(n)}\}_{n=1}^\infty$-martingale with respect to $\bar\nu_\bff$. From the martingale convergence theorem, $Y^{(k,n)}$ converges to $\bfone_{B_k}$ $\nu_\bff$-a.e.\ as $n\to\infty$ from Lemma~\ref{lem:A2}~(i).
In particular, for each $k\in\N$, there exist $x_k\in B_k$ and $N_k\in\N$ such that $Y^{(k,n)}(x_k)\ge 1-2^{-k}$ for all $n\ge N_k$. Then, there exists a strictly increasing sequence $\{n_k\}_{k=1}^\infty$ of natural numbers such that $Y^{(k,n_k)}(x_k)\ge 1-2^{-k}$ for all $k\in\N$.
We may assume that $n_1\ge m_0$.
We write $\lm_k$ for $\lm$ such that $x_k\in U_\lm^{(n_k)}$.%
\smallskip

\textbf{Step 4 (Selection of the windows $(F_k,O_k)$ and normalization of $\bfh^{(k)}$). }
Let $k\in\N$ and write $U_k:=U_{\lm_k}^{(n_k)}$ for brevity.
Note that $\nu_{\bff}(U_k)>0$, since
$Y^{(k,n_k)}(x_k)\ge1-2^{-k}>0$, and that
\[
\nu_{\bff}(U_k\setminus B_k)\le2^{-k}\nu_{\bff}(U_k).
\]
Since $n_k\ge m_0$, iterated application of (A2)~(a) provides
$\lm'\in\Lambda_{m_0}$ with $U_k\subset U_{\lm'}^{(m_0)}$.
Note that $f_i\in\cH_{m_0}\subset\cH_{U_k}$ for
every $i=1,2,\dots,d$ by Lemma~\ref{lem:3.2}.

We apply Lemma~\ref{lem:winA} (when (A3) is assumed) or Lemma~\ref{lem:winB} (when (A3') is assumed) with $(n,\lm)=(n_k,\lm_k)$, $\bfh:=\bff$, $B:=B_k\cap U_k$, and
$\eps:=2^{-k}$, and denote by $(F_k,O_k)$ the resulting pair;
under (A3), $(F_k,O_k)=(V^{(n_k)}_{\lm_k},U_k)$, and under
(A3'), $(F_k,O_k)\in\cW_{\lm_k}$. In either case,
$F_k$ is a compact subset of the open set $O_k\subset U_k$, and the
following hold:
\begin{enumerate}[(a)]
\item $0<\Cp(F_k;O_k)<\infty$;
\item $\nu_\bff(F_k)>0$,
$\nu_\bff(O_k)\le C_0\nu_\bff(F_k)$, and $\nu_\bff(O_k\setminus B_k)\le C_02^{-k}\nu_\bff(F_k)$;
\item $\Cp(F_k;O_k)\bigl(\muosc_{O_k}f_i\bigr)^2\le C\nu_{f_i}(O_k)$ for $i=1,2,\dots,d$;
\item $\Cp(\{x\};O_k)\ge C^{-1}\Cp(F_k;O_k)$ for q.e.\,$x\in F_k$.
\end{enumerate}
By (b), $F_k\ne\emptyset$, so that $O_k$ is a non-empty open set and $\mu(O_k)>0$ because $\mu$ has full support.

Take $\psi_k\in\cF\cap C(K)$ such that $0\le\psi_k\le1$ on $K$
and $\psi_k=1$ on $U_k$. We write $h_i^{(k)}$
($k\in\N$, $i=1,2,\dots,d$) for $c_k(f_i-d_{i,k}\psi_k)$ and
let $\bfh^{(k)}=(h_1^{(k)},h_2^{(k)},\dots,h_d^{(k)})$, where
the real constants $c_k>0$ (depending only on $k$) and $d_{i,k}$
(depending on $i$ and $k$) are selected so that
\begin{equation}\label{eq:shift}
\muessinf_{O_k}h_i^{(k)}\le 0 \le \muesssup_{O_k}h_i^{(k)},
\qquad i=1,2,\dots,d,
\end{equation}
and
\begin{equation}\label{eq:normalize}
\nu_{\bfh^{(k)}}(F_k)=1.
\end{equation}
Here, recall \eqref{eq:nubfh}. Such choices are possible. Indeed, since
$\psi_k=1$ on $U_k$, Proposition~\ref{prop:sl}~(iv) gives
$\nu_{\psi_k}(U_k)=0$, and hence, by Proposition~\ref{prop:sl}~(i),
\[
\nu_{h_i^{(k)},h_j^{(k)}}=c_k^2\,\nu_{f_i,f_j}
\quad\text{on }U_k,\qquad i,j=1,2,\dots,d;
\]
in particular, $\nu_{\bfh^{(k)}}=c_k^2\nu_\bff$ on $U_k$. 
Since $\nu_\bff(F_k)>0$ by (b), the choice $c_k:=\nu_\bff(F_k)^{-1/2}$ attains \eqref{eq:normalize}, and the
choice $d_{i,k}:=\muessinf_{O_k}f_i$, which is finite because $f_i$ is bounded and $\mu(O_k)>0$, attains \eqref{eq:shift}. 
Every $h_i^{(k)}$ is bounded because $f_i$ is bounded and $0\le\psi_k\le1$.
Moreover, $h_i^{(k)}\in\cH_{U_k}$ from Lemmas~\ref{lem:3.2} and \ref{lem:3.3}.

Note that \eqref{eq:shift} implies
\begin{equation}\label{eq:suposc}
\muesssup_{O_k}{}\bigl|h_i^{(k)}\bigr|
\le\muosc_{O_k}h_i^{(k)},
\qquad i=1,2,\dots,d.
\end{equation}
From \eqref{eq:normalize} and (b), $\nu_{\bfh^{(k)}}(O_k)\le C_0\nu_{\bfh^{(k)}}(F_k)=C_0$ and $\nu_{h_i^{(k)}}(O_k)\le d \nu_{\bfh^{(k)}}(O_k)\le d C_0$ for each $i$; hence, for $k\in\N$ and $i,j=1,2,\dots,d$,
\begin{equation}\label{eq:normalizeij}
|\nu_{h_i^{(k)},h_j^{(k)}}|(O_k)
\le \nu_{h_i^{(k)}}(O_k)^{1/2}\nu_{h_j^{(k)}}(O_k)^{1/2}
\le d C_0.
\end{equation}
Moreover, from (b) and $\nu_{\bfh^{(k)}}=c_k^2\nu_\bff$ on $U_k$,
\begin{equation}\label{eq:Y}
\nu_{\bfh^{(k)}}(O_k\setminus B_k)\le C_02^{-k}.
\end{equation}

Let $e_k$ be defined as in \eqref{eq:ekdef}. That is, by
Proposition~\ref{prop:qe}~(ii), which is applicable because $O_k\subset U_k$ is relatively compact and $F_k$ is a compact
subset of $O_k$, we have $0\le e_k\le1$, $e_k=1$ q.e.\ on $F_k$, $e_k=0$ q.e.\ on $K\setminus O_k$, and
$\cE(e_k,e_k)=\Cp(F_k;O_k)=\b_k^2$, where $\b_k$ is defined by \eqref{eq:bkdef} and $\b_k\in(0,\infty)$ by (a).
Property~(c) persists under the rescaling: since $\psi_k=1$ on
$U_k\supset O_k$, we have $\muosc_{O_k}h_i^{(k)}=c_k\muosc_{O_k}f_i$ and $\nu_{h_i^{(k)}}(O_k)=c_k^2\nu_{f_i}(O_k)$, so that
\[
\b_k^2\bigl(\muosc_{O_k}h_i^{(k)}\bigr)^2
\le C\nu_{h_i^{(k)}}(O_k)\le dC_0C,
\qquad i=1,2,\dots,d.
\]
Together with \eqref{eq:suposc}, this gives
$|\b_k h_i^{(k)}|\le(dC_0C)^{1/2}$ $\mu$-a.e.\ on $O_k$ for each
$i$, hence $|\b_k\bfh^{(k)}|\le d\sqrt{C_0C}=R_0$ $\mu$-a.e.\ on $O_k$. From Lemma~\ref{lem:aeqe},
\begin{equation}
|\beta_k\boldsymbol h^{(k)}|\le R_0\quad\text{q.e.\ on }O_k.
\label{eq:unif}
\end{equation}
Property (d) is the fourth item in the statement of Proposition~\ref{prop:key}.

Let $\ph\in C^1_b(\R^d)$ with $\ph(0,\dots,0)=0$, and take $g=\ph(\b_k\bfh^{(k)})\in\cF\cap L^\infty(K,\mu)$.
For each $i=1,2,\dots,d$,
\begin{align*}
0&=2\cE(g e_k^2,h_i^{(k)})\\
&\qquad\parbox{30em}{(because Lemma~\ref{lem:3.2} can be applied, since $g e_k^2=0$ q.e.\ on $K\setminus U_{\lm_k}^{(n_k)}$ and $h_i^{(k)}$ is harmonic on $U_{\lm_k}^{(n_k)}$)}\\
&=\int_{K}\dd\nu_{g e_k^2,h_i^{(k)}}\quad\text{(from Proposition~\ref{prop:sl}~(iii))}\\
&=\int_{K}g\,\dd\nu_{e_k^2,h_i^{(k)}}+\int_{K} e_k^2\,\dd\nu_{g,h_i^{(k)}}\quad\text{(from Proposition~\ref{prop:sl}~(v))}\\
&=\int_{K}g\,\dd\nu_{e_k^2,h_i^{(k)}}+\sum_{j=1}^d \int_{K} e_k^2\frac{\partial \ph}{\partial x_j}(\b_k\bfh^{(k)})\b_k\,\dd\nu_{h_j^{(k)},h_i^{(k)}}\\*
&\qquad\text{(from Proposition~\ref{prop:sl}~(v))}\\
&=\int_{\R^d}\ph\,\dd\bigl((\b_k\bfh^{(k)})_*\nu_{e_k^2,h_i^{(k)}}\bigr)
+\int_{\R^d}\frac{\partial\ph}{\partial x_i}\b_k\,\dd\bigl((\b_k\bfh^{(k)})_*(e_k^2\nu_{h_i^{(k)}})\bigr)\\*
&\quad+\sum_{j:j\ne i}\int_{\R^d}\frac{\partial\ph}{\partial x_j}\b_k\,\dd\bigl((\b_k\bfh^{(k)})_*(e_k^2\nu_{h_i^{(k)},h_j^{(k)}})\bigr).
\end{align*}
We note the (signed) measures $\nu_{e_k^2,h_i^{(k)}}$ and $e_k^2\nu_{h_i^{(k)},h_j^{(k)}}$ vanish on $K\setminus O_k$. 
Indeed, $e_k = 0$ q.e.\ on $K\setminus O_k$ and energy measures charge no set of zero capacity, so $e_k^2 \nu_{h_i^{(k)}, h_j^{(k)}}$ vanishes there. Moreover, since $0 \le e_k \le 1$, applying Proposition~\ref{prop:sl}~(v) with $\Phi\in C_b^1(\R)$ such that $\Phi(t) = t^2$ on $[0,1]$ yields $\nu_{e_k^2} = 4 e_k^2 \,\nu_{e_k}$, which vanishes on $K\setminus O_k$ by the same reason; hence, by Proposition~\ref{prop:sl}~(i), $|\nu_{e_k^2, h_i^{(k)}}(A)|\le\nu_{e_k^2}(A)^{1/2} \nu_{h_i^{(k)}}(A)^{1/2} = 0$ for every Borel set $A\subset K\setminus O_k$.

Dividing both sides by $\b_k$, we obtain
\begin{align}
0&=\int_{\R^d}\ph\,\dd\bigl((\b_k\bfh^{(k)}|_{O_k})_*(\b_k^{-1}\nu_{e_k^2,h_i^{(k)}})\bigr)
+\int_{\R^d}\frac{\partial\ph}{\partial x_i}\,\dd\bigl((\b_k\bfh^{(k)}|_{O_k})_*(e_k^2\nu_{h_i^{(k)}})\bigr)\nonumber\\*
&\quad+\sum_{j:j\ne i}\int_{\R^d}\frac{\partial\ph}{\partial x_j}\,\dd\bigl((\b_k\bfh^{(k)}|_{O_k})_*(e_k^2\nu_{h_i^{(k)},h_j^{(k)}})\bigr).
\label{eq:phi}
\end{align}

\textbf{Step 5 (Proof of the absolute continuity of some limit measure). }
We consider taking limits of the right-hand side of \eqref{eq:phi} as $k\to\infty$.
For $i=1,2,\dots,d$, we have
\begin{align*}
\left|\b_k^{-1}\nu_{e_k^2,h_i^{(k)}}\right|(K)
&=\b_k^{-1}\bigl|\nu_{e_k^2,h_i^{(k)}}\bigr|(O_k)\\
&\le \b_k^{-1}\nu_{e_k^2}(O_k)^{1/2}\nu_{h_i^{(k)}}(O_k)^{1/2}\quad\text{(from Proposition~\ref{prop:sl}~(i))}\\
&\le \b_k^{-1}\left(\int_{O_k}4e_k^2\,\dd\nu_{e_k}\right)^{1/2}(d C_0)^{1/2}\\*
&\hspace{3cm}\text{(from Proposition~\ref{prop:sl}~(v) and \eqref{eq:normalizeij})}\\
&\le 2\b_k^{-1}\nu_{e_k}(K)^{1/2}(d C_0)^{1/2}\\
&= 2\b_k^{-1}\bigl(2\cE(e_k,e_k)\bigr)^{1/2}(d C_0)^{1/2}\quad\text{(from Proposition~\ref{prop:sl}~(iii))}\\
&= 2(2dC_0)^{1/2}.
\end{align*}
Also, for $i,j=1,2,\dots,d$, we have
\[
\bigl|e_k^2\nu_{h_i^{(k)},h_j^{(k)}}\bigr|(K)\le \bigl|\nu_{h_i^{(k)},h_j^{(k)}}\bigr|(O_k)
\le d C_0
\]
from \eqref{eq:normalizeij}.
Thus, the total masses of the measures $|\b_k^{-1}\nu_{e_k^2,h_i^{(k)}}|$ and $|e_k^2\nu_{h_i^{(k)},h_j^{(k)}}|$ are bounded in $i$, $j$, and $k$.
Since $\b_k \bfh^{(k)}(x)$ belongs to the closed ball $W_{R_0}:=\{z\in\R^d\mid|z|\le R_0\}$ for q.e.\ $x\in O_k$ by \eqref{eq:unif}, we can take a subsequence $\{k'\}$ of $\{k\}$ and finite Borel measures $\kp_{i,j}$ $(i,j=1,2,\dots,d)$ and signed measures $\hat\kp_i$ $(i=1,2,\dots,d)$ on $\R^d$ such that the supports of $\kp_{i,j}$ and $\hat\kp_i$ are all included in $W_{R_0}$ and
\begin{align*}
\bigl(\b_{k'} \bfh^{(k')}|_{O_{k'}}\bigr)_*\bigl(\b_{k'}^{-1}\nu_{e_{k'}^2,h_i^{(k')}}\bigr) &\to \hat\kp_i\quad (i=1,2,\dots,d),\\
\bigl(\b_{k'} \bfh^{(k')}|_{O_{k'}}\bigr)_*\bigl|e_{k'}^2\nu_{h_i^{(k')},h_j^{(k')}}\bigr|&\to \kp_{i,j}\quad (i,j=1,2,\dots,d)
\end{align*}
as $k'\to\infty$.
Here, the convergences mean the weak convergence on the space of signed Borel measures on $\R^d$. 
We define $\kp=(1/d)\sum_{i=1}^d \kp_{i,i}$.
Then,
\[
\bigl(\b_{k'} \bfh^{(k')}|_{O_{k'}}\bigr)_*\bigl(e_{k'}^2\nu_{\bfh^{(k')}}\bigr)
=\frac1d\sum_{i=1}^d\bigl(\b_{k'} \bfh^{(k')}|_{O_{k'}}\bigr)_*\bigl(e_{k'}^2\nu_{h_i^{(k')}}\bigr)
\to \kp\quad\text{as $k'\to\infty$}.
\]
Since $e_{k'}=1$ q.e.\ on $F_{k'}$ and $\nu_{\bfh^{(k')}}$ charges no set of zero capacity,
\[
\int_K e_{k'}^2\,\dd\nu_{\bfh^{(k')}}
\ge \nu_{\bfh^{(k')}}(F_{k'})=1,
\]
which implies $\kp(\R^d)\ge 1$.

Let us recall \eqref{eq:Bk}. On $B_{k'}$, it holds that $\|\Phi_{\bfh^{(k')}}-I\|_{\HS}<1/k'$; since $\nu_{h_i^{(k')},h_j^{(k')}}=c_{k'}^2\nu_{f_i,f_j}$ on $U_{k'}$, this implies that
\[
\left|\left|\frac{\dd\nu_{h_i^{(k')},h_j^{(k')}}}{\dd\nu_{\bfh^{(k')}}}\right|-\dl_{ij}\right|
<\frac1{k'}
\quad\text{$\nu_{\bfh^{(k')}}$-a.e.\ on }B_{k'}\cap U_{k'},
\qquad i,j=1,2,\dots,d.
\]
Let $F\in C_c(\R^d)$ with $F\ge0$. For $i=1,2,\dots,d$, we have
\begin{equation}\label{eq:F}
\int_{\R^d} F\,\dd\Bigl[\bigl(\b_{k'}\bfh^{(k')}|_{O_{k'}}\bigr)_*\bigl(e_{k'}^2\nu_{h_i^{(k')}}\bigr)\Bigr]
=\int_{O_{k'}}\{F\circ(\b_{k'}\bfh^{(k')})\}e_{k'}^2\,\dd\nu_{h_i^{(k')}}.
\end{equation}
The left-hand side of \eqref{eq:F} converges to $\int_{\R^d}F\,\dd\kp_{i,i}$ as $k'\to\infty$, while
\begin{align*}
&\text{RHS of \eqref{eq:F}}\\*
&\ge\left(1-\frac1{k'}\right)\int_{O_{k'}\cap B_{k'}}\{F\circ(\b_{k'}\bfh^{(k')})\}e_{k'}^2\,\dd\nu_{\bfh^{(k')}}\\
&\ge\left(1-\frac1{k'}\right)\int_{O_{k'}} \{F\circ(\b_{k'}\bfh^{(k')})\}e_{k'}^2\,\dd\nu_{\bfh^{(k')}}
-\left(1-\frac1{k'}\right)\|F\|_\infty\nu_{\bfh^{(k')}}(O_{k'}\setminus B_{k'})\\
&\ge\left(1-\frac1{k'}\right)\int_{\R^d} F\,\dd\Bigl[\bigl(\b_{k'}\bfh^{(k')}|_{O_{k'}}\bigr)_*\bigl(e_{k'}^2\nu_{\bfh^{(k')}}\bigr)\Bigr]
-C_0\|F\|_\infty 2^{-k'} \qquad\text{(from \eqref{eq:Y})}\\
&\to \int_{\R^d}F\,\dd\kp \quad\text{as }k'\to\infty
\end{align*}
and
\begin{align*}
&\text{RHS of \eqref{eq:F}}\\*
&\le\left(1+\frac1{k'}\right)\int_{O_{k'}\cap B_{k'}}\{F\circ(\b_{k'}\bfh^{(k')})\}e_{k'}^2\,\dd\nu_{\bfh^{(k')}}
+d\|F\|_\infty\nu_{\bfh^{(k')}}(O_{k'}\setminus B_{k'})\\
&\le \left(1+\frac1{k'}\right)\int_{\R^d} F\,\dd\Bigl[\bigl(\b_{k'}\bfh^{(k')}|_{O_{k'}}\bigr)_*\bigl(e_{k'}^2\nu_{\bfh^{(k')}}\bigr)\Bigr]
+d C_0\|F\|_\infty 2^{-k'}\qquad\text{(from \eqref{eq:Y})}\\
&\to \int_{\R^d}F\,\dd\kp \quad\text{as }k'\to\infty.
\end{align*}
Therefore, $\int_{\R^d}F\,\dd\kp_{i,i}=\int_{\R^d}F\,\dd\kp$.
This implies that $\kp_{i,i}=\kp$.
Moreover, for $i,j=1,2,\dots,d$ with $i\ne j$,
\begin{equation}\label{eq:F2}
\int_{\R^d} F\,\dd\Bigl[\bigl(\b_{k'}\bfh^{(k')}|_{O_{k'}}\bigr)_*\bigl|e_{k'}^2\nu_{h_i^{(k')},h_j^{(k')}}\bigr|\Bigr]=\int_{O_{k'}}\{F\circ(\b_{k'}\bfh^{(k')})\}e_{k'}^2\,\dd\bigl|\nu_{h_i^{(k')},h_j^{(k')}}\bigr|.
\end{equation}
The left-hand side of \eqref{eq:F2} converges to $\int_{\R^d}F\,\dd\kp_{i,j}$ as $k'\to\infty$, while
\begin{align*}
0&\le \text{RHS of \eqref{eq:F2}}\\
&\le \frac1{k'}\int_{O_{k'}\cap B_{k'}}\{F\circ(\b_{k'}\bfh^{(k')})\}e_{k'}^2\,\dd\nu_{\bfh^{(k')}}
+\|F\|_\infty\bigl|\nu_{h_i^{(k')},h_j^{(k')}}\bigr|(O_{k'}\setminus B_{k'})\\
&\le\frac1{k'}\|F\|_\infty\nu_{\bfh^{(k')}}(O_{k'})+\|F\|_\infty\cdot d\nu_{\bfh^{(k')}}(O_{k'}\setminus B_{k'})\\
&\le \frac{C_0}{k'}\|F\|_\infty+d C_0 2^{-k'}\|F\|_\infty \qquad\text{(from \eqref{eq:Y})}\\
&\to 0 \quad\text{as }k'\to\infty.
\end{align*}
Therefore, $\int_{\R^d}F\,\dd\kp_{i,j}=0$. This implies that $\kp_{i,j}=0$.
In particular, the signed measure $\bigl(\b_{k'} \bfh^{(k')}|_{O_{k'}}\bigr)_*\bigl(e_{k'}^2\nu_{h_i^{(k')},h_j^{(k')}}\bigr)$ converges weakly to the zero measure.

Then, letting $k\to\infty$ along the subsequence $\{k'\}$ in \eqref{eq:phi}, we obtain that
\begin{equation}\label{eq:phi2}
0=\int_{\R^d}\ph\,\dd\hat\kp_i+\int_{\R^d}\frac{\partial\ph}{\partial x_i}\,\dd\kp_{i,i}
=\int_{\R^d}\ph\,\dd\hat\kp_i+\int_{\R^d}\frac{\partial\ph}{\partial x_i}\,\dd\kp.
\end{equation}
By noting that
\begin{align*}
\hat\kp_i(\R^d)
&=\lim_{k'\to\infty}\b_{k'}^{-1}\nu_{e_{k'}^2,h_i^{(k')}}(O_{k'})\\
&=\lim_{k'\to\infty}\b_{k'}^{-1}\nu_{e_{k'}^2,h_i^{(k')}}(K)\\
&=\lim_{k'\to\infty}2\b_{k'}^{-1}\cE(e_{k'}^2,h_i^{(k')})\qquad\text{(from Proposition~\ref{prop:sl}~(iii))}\\
&=0\qquad\text{(from Lemma~\ref{lem:3.2}, since $e_{k'}^2=0$ q.e.\ on $K\setminus U_{k'}$),}
\end{align*}
\eqref{eq:phi2} holds for all $\ph\in C^1_b(\R^d)$.
Therefore, $\frac{\partial}{\partial x_i}\kp=\hat\kp_i$ holds in the distribution sense.
From \cite[Chapter~I, Lemma~7.2.2.1]{BH91}, $\kp$ is absolutely continuous with respect to the $d$-dimensional Lebesgue measure $\cL^d$.
\smallskip

\textbf{Step 6 (Proof of the regularity of the limit measure). }
Let $\xi\in L^1(\R^d,\cL^d)$ denote the Radon--Nikodym derivative $\dd\kp/\dd\cL^d$.

Let $\ph\in C_c(\R^d)$ with $\ph\ge0$. For $i=1,2,\dots,d$, we have
\begin{align}
&\left|\int_{\R^d}\ph\,\dd\hat\kp_i\right|\nonumber\\*
&=\lim_{k'\to\infty}\left|\int_{O_{k'}} \ph(\b_{k'}\bfh^{(k')})\b_{k'}^{-1}\,\dd\nu_{e_{k'}^2,h_i^{(k')}}\right|\nonumber\\
&=\lim_{k'\to\infty}\b_{k'}^{-1}\left|\int_{O_{k'}} \ph(\b_{k'}\bfh^{(k')})\cdot 2e_{k'}\,\dd\nu_{e_{k'},h_i^{(k')}}\right|
\quad\text{(from Proposition~\ref{prop:sl}~(v))}\nonumber\\
&\le\varliminf_{k'\to\infty} 2\b_{k'}^{-1}\left(\int_{O_{k'}} \ph(\b_{k'}\bfh^{(k')})e_{k'}^2\,\dd\nu_{h_i^{(k')}}\right)^{1/2}\left(\int_{O_{k'}} \ph(\b_{k'}\bfh^{(k')})\,\dd\nu_{e_{k'}}\right)^{1/2}\nonumber\\*
&\hspace{5cm}\text{(from Proposition~\ref{prop:sl}~(ii)).}
\label{eq:kpi}
\end{align}
As $k'\to\infty$,
\begin{align*}
\int_{O_{k'}} \ph(\b_{k'}\bfh^{(k')})e_{k'}^2\,\dd\nu_{h_i^{(k')}}
&=\int_{\R^d}\ph\,\dd\Bigl[\bigl(\b_{k'}\bfh^{(k')}|_{O_{k'}}\bigl)_*\bigl(e_{k'}^2\,\dd\nu_{h_i^{(k')}}\bigr)\Bigr]\\
&\to \int_{\R^d}\ph\,\dd\kp_{i,i}=\int_{\R^d}\ph\,\dd\kp.
\end{align*}
We also have
\[
\b_{k'}^{-1}\left(\int_{O_{k'}} \ph(\b_{k'}\bfh^{(k')})\,\dd\nu_{e_{k'}}\right)^{1/2}
=\left\{\int_{\R^d}\ph \,\dd\Bigl[\bigl(\b_{k'}\bfh^{(k')}|_{O_{k'}}\bigr)_*\bigl(\b_{k'}^{-2}\nu_{e_{k'}}\bigr)\Bigr]\right\}^{1/2}.
\]
Here, we note that
\[
\b_{k'}^{-2} \nu_{e_{k'}}(K)=2\b_{k'}^{-2}\cE(e_{k'},e_{k'})=2
\]
from Proposition~\ref{prop:sl}~(iii), \eqref{eq:bkdef}, and \eqref{eq:ekdef}.
Thus, the total mass of $\bigl(\b_{k'}\bfh^{(k')}|_{O_{k'}}\bigr)_*(\b_{k'}^{-2}\nu_{e_{k'}})$ is at most $2$.
Keeping \eqref{eq:unif} in mind, we can extract a subsequence $\{k''\}$ of $\{k'\}$ such that $\bigl(\b_{k''}\bfh^{(k'')}|_{O_{k''}}\bigr)_*(\b_{k''}^{-2}\nu_{e_{k''}})$ converges weakly as $k''\to\infty$ to some finite Borel measure $\tilde\kp$ on $\R^d$ whose support is included in $W_{R_0}$.
In particular, $\tilde\kp(\R^d)\le2$.
Therefore, \eqref{eq:kpi} implies
\begin{equation}\label{eq:kp}
\left|\int_{\R^d}\ph\,\dd\hat\kp_i\right|
\le 2\left(\int_{\R^d}\ph\,\dd\kp\right)^{1/2}\left(\int_{\R^d}\ph\,\dd\tilde\kp\right)^{1/2}
\end{equation}
for any $\ph\in C_c(\R^d)$ with $\ph\ge0$.
For an open set $A\subset\R^d$, we may approximate $\bfone_A$ pointwise by a non-decreasing sequence of non-negative functions in $C_c(\R^d)$ (recall that all the measures under consideration are supported in the common compact set $W_{R_0}$); by the dominated convergence theorem with respect to the finite measures $|\hat\kp_i|$, $\kp$, and $\tilde\kp$, the inequality \eqref{eq:kp} then holds with $\ph=\bfone_A$, so that $|\hat\kp_i(A)|\le2\bigl(\kp(A)\tilde\kp(A)\bigr)^{1/2}$ for every open set $A$. 
Next, let $A$ be a Borel set of $\R^d$ and $\eps>0$. 
Since the finite Borel measures $|\hat\kp_i|$, $\kp$, and $\tilde\kp$ on $\R^d$ are outer regular, there exists an open set $O\supset A$ such that $|\hat\kp_i|(O\setminus A)\le\eps$, $\kp(O)\le\kp(A)+\eps$, and $\tilde\kp(O)\le\tilde\kp(A)+\eps$. Hence,
\[
|\hat\kp_i(A)|
\le|\hat\kp_i(O)|+|\hat\kp_i|(O\setminus A)
\le2\bigl\{(\kp(A)+\eps)(\tilde\kp(A)+\eps)
\bigr\}^{1/2}+\eps.
\]
Letting $\eps\to0$, we obtain
\[
|\hat\kp_i(A)|\le2\bigl(\kp(A) \tilde\kp(A)\bigr)^{1/2}
\quad\text{for every Borel set }A\subset\R^d.
\]
If $\kp(A)=0$, then $\kp(B)=0$ and hence $\hat\kp_i(B)=0$ for every Borel $B\subset A$, so that $|\hat\kp_i|(A)=0$.
Therefore, $|\hat\kp_i|\ll\kp$, in particular, $\hat\kp_i\ll\cL^d$. Thus, $\xi\in W^{1,1}(\R^d)$.

Let $\xi_i=\dd\hat\kp_i/\dd\cL^d$. From \cite[Section~1.6, Theorems~1.30--1.32]{EG25}, 
\[
|\xi_i|\le 2\xi^{1/2}\left(\frac{\dd\tilde\kp_{\mathrm{ac}}}{\dd\cL^d}\right)^{1/2}\quad \cL^d\text{-a.e.},
\]
where $\tilde\kp_{\mathrm{ac}}$ is the absolutely continuous part of $\tilde\kp$ in the Lebesgue decomposition.
For $\eps>0$, let $\gm_\eps(t)=\sqrt{t+\eps}-\sqrt{\eps}$, $t\ge0$.
Then,
\[
\left(\frac{\partial(\gm_\eps\circ\xi)}{\partial x_i}\right)^2
=\left(\frac1{2\sqrt{\xi+\eps}}\frac{\partial\xi}{\partial x_i}\right)^2
=\frac{\xi_i^2}{4(\xi+\eps)}
\le \frac{\dd\tilde\kp_{\mathrm{ac}}}{\dd\cL^d}.
\]
Thus,
\[
\int_{\R^d}\left(\frac{\partial(\gm_\eps\circ\xi)}{\partial x_i}\right)^2 \dd x
\le \tilde\kp_{\mathrm{ac}}(\R^d)\le\tilde\kp(\R^d)<\infty.
\]
Since $\gm_\eps(t)\to\sqrt{t}$ as $\eps\to0$, $\partial\sqrt{\xi}/{\partial x_i}\in L^2(\R^d,\cL^d)$. Therefore, we conclude that $\sqrt{\xi}\in W^{1,2}(\R^d)$.

Finally, by relabeling the subsequence $\{k''\}$ as $\{k\}$, all the convergences asserted in Proposition~\ref{prop:key} hold. This completes the proof.
\end{proof}
\begin{remark}
We proved that $\sqrt\xi\in W^{1,2}(\R^d)$ in Proposition~\ref{prop:key}, but we do not use this property fully for proving Theorem~\ref{th:main}: only $\xi\in L^1(\R^2,\cL^2)$ is used for Theorem~\ref{th:main}. 
This additional regularity may be useful in variants of the
present argument under weaker capacity assumptions.
\end{remark}
We now prove the main theorem.
\begin{proof}[Proof of Theorem~\protect\ref{th:main}]
From the non-triviality assumption (A1), the AF-martingale dimension $d_{\mathrm{m}}$ is greater than 0 from Theorem~\ref{th:char} and \cite[Proposition~2.11]{Hi10}. 
Therefore, it suffices to prove that $d_{\mathrm{m}}<2$.
We deduce a contradiction by assuming $d_{\mathrm{m}}\ge2$.
By applying Proposition~\ref{prop:key} with $d=2$, there exist a strictly increasing sequence $\{n_k\}_{k=1}^\infty$ of natural numbers, $\lm_k\in \Lm_{n_k}$, pairs $(F_k,O_k)$, and $\bfh^{(k)}=(h_1^{(k)},h_2^{(k)})\in(\cH_{U_{\lm_k}^{(n_k)}})^2$ for $k\in\N$ as claimed in Proposition~\ref{prop:key}.
In particular,
\begin{equation}\label{eq:frombelow}
\Cp(\{x\};O_k)\ge C^{-1}\b_k^2\quad
\text{for q.e.}\,x\in F_k.
\end{equation}

For $z\in\R^2$ and $r>0$, let $B(z,r)$ (resp.\ $\bar B(z,r)$) denote the open (resp.\ closed) ball in $\R^2$ with center $z$ and radius $r$.

Let 
\[
N_1=\{z\in\R^2\mid \tilde\kp(\{z\})>0\}
\]
and
\[
N_2=\left\{z\in\R^2\mathrel{}\middle|\mathrel{} \sup_{r\in(0,1]}\fint_{\bar B(z,r)}\xi(w)\,\dd w=+\infty\right\}.
\]
Then, $N_1$ is at most countable, and $\cL^2(N_2)=0$ from Lebesgue's differentiation theorem.

Let
\[
\eta_k=\bigl(\b_k \bfh^{(k)}|_{O_k}\bigr)_*\bigl(\bfone_{F_k}\nu_{\bfh^{(k)}}\bigr).
\]
Then, $\eta_k(\R^2)=\nu_{\bfh^{(k)}}(F_k)=1$, and the support of $\eta_k$ is included in the compact set $W_{R_0}=\{z\in\R^2\mid |z|\le R_0\}$. We can take a subsequence $\{k'\}$ of $\{k\}$ such that $\eta_{k'}$ converges weakly to some Borel measure $\eta_\infty$ on $\R^2$ whose support is included in $W_{R_0}$.
Also, $\eta_\infty(\R^2)=\break\lim_{k'\to\infty}\eta_{k'}(\R^2)=1$.
For simplicity, we write just $k$ instead of $k'$ in what follows.

Since $e_k=1$ q.e.\ on $F_k$ and $\nu_{\bfh^{(k)}}$ charges no set of zero capacity, we have $\bfone_{F_k}\nu_{\bfh^{(k)}}\le e_k^2 \nu_{\bfh^{(k)}}$ as measures; hence, from the construction, $\eta_\infty\le \kp:=\xi\cL^2$.
Therefore, $\eta_\infty$ is absolutely continuous with respect to $\cL^2$ and its Radon--Nikodym derivative $\xi_\infty$ satisfies that $\xi_\infty\le\xi$ $\cL^2$-a.e.
Let $S$ be the support of $\eta_\infty$. Since $\xi_\infty$ is a non-zero function, $\cL^2(S)>0$ from Lebesgue's differentiation theorem.
Choose a point $a$ from $S\setminus(N_1\cup N_2)\,(\ne\emptyset)$ and define
\[
b:=\sup_{r\in(0,1]}\fint_{\bar B(a,r)}\xi(w)\,\dd w<\infty.
\]
From Lemma~\ref{lem:measure}, there exist $k_1\in\N$ and a sequence $\{a^{(k)}\}_{k=k_1}^\infty$ in $W_{R_0}\subset \R^2$ such that $a^{(k)}$ converges to $a$ as $k\to\infty$ and $a^{(k)}$ belongs to the support of $\eta_{k}$ for each $k$.

For $\eps>0$, let $g_\eps\colon[0,\infty)\to[0,1]$ be a smooth function such that
\[
g_\eps(t)=\begin{cases}
1&\text{on }[0,e^{-2/\eps}]\\
-3\eps\log t-4 & \text{on }[e^{-14/(9\eps)},e^{-13/(9\eps)}]\\
0 & \text{on }[e^{-1/\eps},\infty)
\end{cases}
\]
and $-3\eps/t\le g_\eps'(t)\le0$ for all $t>0$.
Let $\Psi^{(k)}_\eps(z):=g_\eps(|z-a^{(k)}|)$ for $z\in\R^2$.
Since $g_\eps$ is constant on a neighborhood of $0$, we have
$\Psi^{(k)}_\eps\in C^1_b(\R^2)$.
Let $\ph^{(k)}_\eps=\Psi^{(k)}_\eps(\b_k\bfh^{(k)})$.
Note that $\Psi^{(k)}_\eps(0)=g_\eps(|a^{(k)}|)$ may differ from $0$,
so that $\ph^{(k)}_\eps$ itself does not necessarily belong to $\cF$.
To handle this, set $\tilde\Psi^{(k)}_\eps:=\Psi^{(k)}_\eps-\Psi^{(k)}_\eps(0)$ and
$\tilde\ph^{(k)}_\eps :=\tilde\Psi^{(k)}_\eps(\b_k\bfh^{(k)}) =\ph^{(k)}_\eps-\Psi^{(k)}_\eps(0)$.
Since $\tilde\Psi^{(k)}_\eps\in C^1_b(\R^2)$ and
$\tilde\Psi^{(k)}_\eps(0)=0$, we have
$\tilde\ph^{(k)}_\eps\in\cF\cap L^\infty(K,\mu)$, and hence
\begin{equation*}
  e_k^2\ph^{(k)}_\eps
  = e_k^2\tilde\ph^{(k)}_\eps + \Psi^{(k)}_\eps(0)e_k^2
  \in\cF\cap L^\infty(K,\mu).
\end{equation*}
In what follows, $\nu_{\ph^{(k)}_\eps}$ and
$\nu_{e_k^2,\ph^{(k)}_\eps}$ are understood as
$\nu_{\tilde\ph^{(k)}_\eps}$ and
$\nu_{e_k^2,\tilde\ph^{(k)}_\eps}$, respectively.
With this convention, \eqref{eq:2dim} below follows from
Proposition~\ref{prop:sl}~(v) applied to $\Phi(x,y)=x(y+\Psi^{(k)}_\eps(0))$ (after a suitable truncation, since $e_k^2$ and $\tilde\ph^{(k)}_\eps$ are bounded), because
$\partial_x\Phi(e_k^2,\tilde\ph^{(k)}_\eps)=\ph^{(k)}_\eps$
and $\partial_y\Phi(e_k^2,\tilde\ph^{(k)}_\eps)=e_k^2$.
Moreover, since
$\nabla\tilde\Psi^{(k)}_\eps=\nabla\Psi^{(k)}_\eps$,
the expression of $\nu_{\tilde\ph^{(k)}_\eps}$ given by \eqref{eq:derivation} is the same as that formally obtained from $\ph^{(k)}_\eps$, so that the computation in \eqref{eq:2dim} and the subsequent estimates are unaffected.

Then, $\ph_\eps^{(k)}$ is quasi-continuous, $e_k^2\ph_\eps^{(k)}=1$ q.e.\ on $G_\eps^{(k)}:=F_k\cap \{x\in K\mid |\b_k\bfh^{(k)}(x)-a^{(k)}|<e^{-2/\eps}\}$, and $e_k^2\ph_\eps^{(k)}=0$ q.e.\ on $K\setminus O_k$. We note that $G_\eps^{(k)}$ is non-empty because $a^{(k)}$ belongs to the support of $\eta_k$ and thus $\nu_{\bfh^{(k)}}(G_\eps^{(k)})=\eta_k(B(a^{(k)},e^{-2/\eps}))>0$.
From Proposition~\ref{prop:qe}~(i), 
\begin{equation}\label{eq:CapG}
\Cp(G_\eps^{(k)};O_k)\le \cE(e_k^2\ph_\eps^{(k)},e_k^2\ph_\eps^{(k)}).
\end{equation}
We also have
\begin{align}
\b_{k}^{-2}\cE(e_{k}^2\ph_\eps^{(k)},e_{k}^2\ph_\eps^{(k)})
&=\frac12 \b_{k}^{-2}\int_K \dd\nu_{e_{k}^2 \ph_\eps^{(k)}}\nonumber\\
&= \frac12\b_{k}^{-2}\left\{\int_K e_{k}^4\,\dd\nu_{\ph_\eps^{(k)}}+2\int_K e_{k}^2\ph_\eps^{(k)}\,\dd\nu_{e_{k}^2,\ph_\eps^{(k)}}+\int_K (\ph_\eps^{(k)})^2\,\dd\nu_{e_{k}^2}\right\}\nonumber\\
&\le \b_{k}^{-2}\int_K e_{k}^4\,\dd\nu_{\ph_\eps^{(k)}}+\b_{k}^{-2}\int_K (\ph_\eps^{(k)})^2\,\dd\nu_{e_{k}^2}\nonumber\\
&=\b_{k}^{-2}\int_K e_{k}^4\,\dd\nu_{\ph_\eps^{(k)}}+4\b_{k}^{-2}\int_K e_k^2(\ph_\eps^{(k)})^2\,\dd\nu_{e_{k}},
\label{eq:2dim}
\end{align}
where we used Proposition~\ref{prop:sl}~(ii) and (v).

We give an upper bound of the right-hand side of \eqref{eq:2dim}.
First,
\begin{align}
&\b_k^{-2}\int_K e_k^4\,\dd\nu_{\ph_\eps^{(k)}}\nonumber\\
&=\sum_{i,j=1}^2\int_K e_k^4 g_\eps'(|\b_k\bfh^{(k)}-a^{(k)}|)^2\frac{(\b_k h_i^{(k)}-a_i^{(k)})(\b_k h_j^{(k)}-a_j^{(k)})}{|\b_k\bfh^{(k)}-a^{(k)}|^2} \,\dd\nu_{h_i^{(k)},h_j^{(k)}}\nonumber\\
&\le\sum_{i,j=1}^2\int_{\R^2}g_\eps'(|z-a^{(k)}|)^2\,\bigl[\bigl(\b_k\bfh^{(k)}|_{O_k}\bigr)_*\bigl|e_k^2\nu_{h_i^{(k)},h_j^{(k)}}\bigr|\bigr](\dd z)\nonumber\\
&\to 2\int_{\R^2}g_\eps'(|z-a|)^2\,\kp(\dd z)
\quad\text{(as $k\to\infty$)}\nonumber\\
&= 2\int_0^\infty g_\eps'(r)^2\,(\psi_*\kp)(\dd r)
\quad\text{(Here, $\psi\colon \R^2\ni z\mapsto |z-a|\in\R$)}\nonumber\\
&\le 2\int_{e^{-2/\eps}}^{e^{-1/\eps}}9\eps^2 r^{-2}\,\dd\Theta(r),
\label{eq:ak2}
\end{align}
where 
\[
  \Theta(r):=(\psi_*\kp)([0,r])=\int_{\bar B(a,r)}\xi(x)\,\dd x\le b\cL^2(\bar B(a,r))=b\pi r^2.
\]
Here, in the third line of \eqref{eq:ak2} from below, we used the fact that $g_\eps'(|\cdot-a^{(k)}|)^2$ converges to $g_\eps'(|\cdot-a|)^2$ uniformly as $k\to\infty$.

Then, the last term of \eqref{eq:ak2} is dominated by
\begin{align*}
18\eps^2\left\{\left[r^{-2}\Theta(r)\right]_{e^{-2/\eps}}^{e^{-1/\eps}}+\int_{e^{-2/\eps}}^{e^{-1/\eps}}2r^{-3}\Theta(r)\,\dd r\right\}
&\le 18\eps^2\left(b\pi+2b\pi\int_{e^{-2/\eps}}^{e^{-1/\eps}}r^{-1}\,\dd r\right)\\
&=18b\pi(\eps^2+2\eps),
\end{align*}
which converges to $0$ as $\eps\to0$.
Next, since $e_k = 0$ q.e.\ on $K\setminus O_k$,
\begin{align*}
4\b_k^{-2}\int_K e_k^2(\ph_\eps^{(k)})^2\,\dd\nu_{e_k}
&\le 4\int_{\R^2}g_\eps(|z-a^{(k)}|)^2\,\bigl[\bigl(\b_k\bfh^{(k)}|_{O_k}\bigr)_*(\b_k^{-2}\nu_{e_k})\bigr](\dd z)\\
&\to 4\int_{\R^2}g_\eps(|z-a|)^2\,\tilde\kp(\dd z)\quad\text{(as $k\to\infty$)}\\
&\to 4\tilde\kp(\{a\})=0\quad\text{(as $\eps\to0$)}.
\end{align*}
Therefore, in view of \eqref{eq:2dim},
\[
\lim_{\eps\to0}\varlimsup_{k\to\infty}\b_k^{-2}\cE(e_k^2\ph_\eps^{(k)}, e_k^2\ph_\eps^{(k)})=0.
\]
This implies that there exists a sequence of positive numbers $\{\eps_{k}\}$ converging to $0$ such that
\[
\b_{k}^{-2}\cE(e_{k}^2\ph_{\eps_{k}}^{(k)},e_{k}^2\ph_{\eps_{k}}^{(k)})\to0
\quad\text{as $k\to\infty$}.
\]
On the other hand, $\nu_{\bfh^{(k)}}(G_{\eps_k}^{(k)})>0$ and $\nu_{\bfh^{(k)}}$ charges no set of zero capacity, so that $\Cp_1(G_{\eps_k}^{(k)})>0$.
In particular, $G_{\eps_k}^{(k)}$ is not included in the exceptional set of \eqref{eq:frombelow}; hence, there exists $x\in G_{\eps_k}^{(k)}$ with $\Cp(\{x\};O_k)\ge C^{-1}\b_k^2$.
By the monotonicity of the relative capacity and \eqref{eq:CapG},
\[
C^{-1}\b_k^2\le \Cp(\{x\};O_k)\le\Cp(G_{\eps_k}^{(k)};O_k)
\le \cE(e_k^2\ph_{\eps_k}^{(k)},e_k^2\ph_{\eps_k}^{(k)}),
\]
that is, $1\le C\b_k^{-2}\cE(e_k^2\ph_{\eps_k}^{(k)},e_k^2\ph_{\eps_k}^{(k)})$ for every $k$.
Letting $k\to\infty$ yields $1\le0$, which is a contradiction.
Therefore, $d_{\mathrm{m}}<2$, and the proof is complete.
\end{proof}

\section{Examples}\label{sec:example}
In this section, we discuss some examples that meet the assumptions of Theorem~\ref{th:main}.
The first one is a class of self-similar sets, by which we confirm that our main theorem is consistent with earlier studies. 
Our second class consists of inhomogeneous Sierpinski gaskets. It includes both self-similar level-$l$ gaskets and genuinely non-self-similar members beyond the scope of earlier results for self-similar sets.
We verify Assumption~\ref{assumption} throughout the class and conclude that the AF-martingale dimension is one.
The third one is the two-dimensional standard Sierpinski carpet, a typical example of infinitely ramified self-similar sets; there, we verify Assumption~(A3') instead of (A3), so that Theorem~\ref{th:main} still applies.

In all examples, every element of $\cF$ has continuous modification, so we use $\osc_A f:=\sup_A f-\inf_A f$ instead of $\muosc_A f$ for $f\in\cF$.
We retain the notation in Section~\ref{sec:results}.

\subsection{P.\,c.\,f.\ self-similar sets}\label{sec:pcf}
We introduce a class of self-similar fractals, following \cite{Ki01}.
Let $K$ be a compact metric space and $N$ an integer greater than one.
Set $S=\{1,2,\ldots,N\}$ and $\Sg=S^\N$.
For $i\in S$, the shift operator $\sg_i\colon \Sg\to\Sg$ is defined as $\sg_i(\om_1\om_2\cdots)=i\om_1\om_2\cdots$.
Suppose that we are given a continuous injective map $\psi_i\colon K\to K$ for each $i\in S$.
We assume that there exists a continuous surjective map $\pi\colon \Sg\to K$ such that $\psi_i\circ \pi=\pi\circ\sg_i$ for every $i\in S$.
The triplet $\cL=(K,S,\{\psi_i\}_{i\in S})$ is called a self-similar structure.

We define $W_0=\{\emptyset\}$ and $W_m=S^m$ for $m\in \N$.
The set $W_*:=\bigcup_{m\in\Z_+}W_m$ is the totality of words consisting of elements of $S$ with finite length.
For $w=w_1w_2\cdots w_m\in W_m$, we define $\psi_w=\psi_{w_1}\circ\psi_{w_2}\circ\cdots\circ\psi_{w_m}$
and $K_w=\psi_w(K)$.
Here, by definition, $\psi_\emptyset$ is the identity map from $K$ to $K$.
For $w=w_1w_2\cdots w_m\in W_m$ and $w'=w'_1w'_2\cdots w'_{m'}\in W_{m'}$, $ww'$ denotes $w_1w_2\cdots w_mw'_1w'_2\cdots w'_{m'}\in W_{m+m'}$.
For $m\in\Z_+$, let $\cB_m$ be the $\sg$-field on $K$ generated by $\{K_w\mid w\in W_m\}$.
Then, $\{\cB_m\}_{m=0}^\infty$ is a filtration on $K$ and the $\sg$-field generated by $\{\cB_m\mid m\in\Z_+\}$ is equal to the Borel $\sg$-field on $K$. 

Let 
\[
\cP=\bigcup_{m\in\N} \sg^m\left(\pi^{-1}\left(\bigcup_{i,j\in S,\,i\ne j}(K_i\cap K_j)\right)\right)\quad\text{and}\quad V_0=\pi(\cP),
\]
where $\sg^m\colon\Sigma\to\Sigma$ is the shift operator that is defined by $\sg^m(\om_1\om_2\cdots)=\om_{m+1}\om_{m+2}\cdots$.
The set $\cP$ is called the post-critical set.
We assume that $K$ is connected and the self-similar structure $(K,S,\{\psi_i\}_{i\in S})$ is {\em post-critically finite} (p.c.f.), that is, $\cP$ is a finite set.
\begin{figure}[t]\centering
\includegraphics[width=0.85\textwidth]{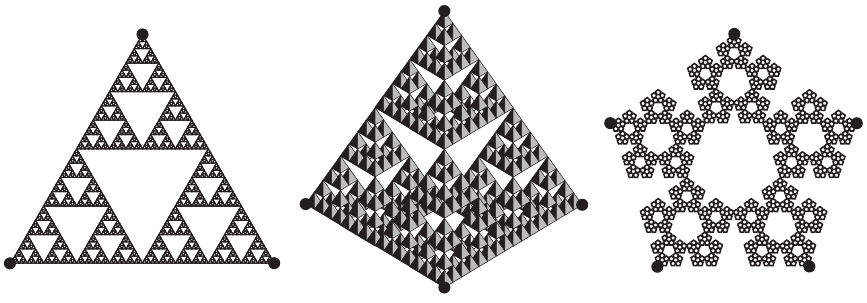}\medskip

\includegraphics[width=0.85\textwidth]{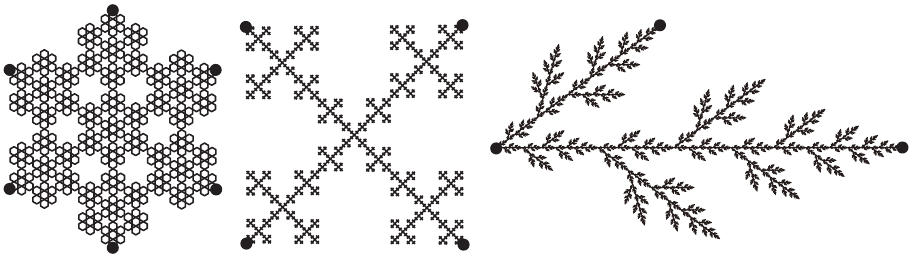}
\caption{(Adapted from \cite[Figure~1]{Hi10}.) Examples of p.c.f.\ self-similar sets. From the upper left, two- and three-dimensional standard Sierpinski gasket, Pentakun (pentagasket), snowflake, the Vicsek set, and Hata's tree-like set.}
\label{fig:pcf}
\end{figure}
Fig.~\ref{fig:pcf} shows several examples of p.c.f.\ self-similar sets $K$.
The set of black points denotes $V_0$ for each fractal.

Let $V_m=\bigcup_{w\in W_m}\psi_w(V_0)$ for $m\in\N$ and $V_*=\bigcup_{m\in\Z_+} V_m$.
For each $x\in K\setminus V_*$, there exists a unique element $\om=\om_1\om_2\cdots\in \Sg$ such that $\pi(\om)=x$.
For each $m\in\N$, $[x]_m$ denotes $\om_1\om_2\cdots\om_m\in W_m$ and $[x]_0$ denotes $\emptyset$.
For $x\in K$, the sequence $\{\bigcup_{w\in W_m;\,x\in K_w}K_w\}_{m=0}^\infty$ is a fundamental system of neighborhoods of $x$ (\cite[Proposition~1.3.6]{Ki01}). Note that\break  $\bigcup_{w\in W_m;\,x\in K_w}K_w=K_{[x]_m}$ if $x\in K\setminus V_*$.

For a finite set $A$, let $\ell(A)$ denote the space of all real-valued functions on $A$, equipped with the inner product $(\cdot,\cdot)_{\ell(A)}$ defined as $(u,v)_{\ell(A)}=\sum_{q\in A}u(q)v(q)$.
Let $D=(D_{qq'})_{q,q'\in V_0}$ be a symmetric linear operator on $\ell(V_0)$ (also regarded as a square matrix of size $\#V_0$) such that the following conditions hold:
\begin{itemize}
\item $D$ is non-positive definite,
\item $Du=0$ if and only if $u$ is constant on $V_0$,
\item $D_{qq'}\ge0$ for all $q\ne q'\in V_0$.
\end{itemize}
We define $\cE^{(0)}(u,v)=(-Du,v)_{\ell(V_0)}$ for $u,v\in \ell(V_0)$.
This is a Dirichlet form on $\ell(V_0)$, where $\ell(V_0)$ is identified with the $L^2$ space on $V_0$ with the counting measure (\cite[Proposition~2.1.3]{Ki01}).
For $\bfr=\{r_i\}_{i\in S}\in(0,\infty)^S$, we define a bilinear form $\cE^{(m)}$ on $\ell(V_m)$ for $m\in\N$ as
\[
  \cE^{(m)}(u,v)=\sum_{w\in W_m}\frac{1}{r_w}\cE^{(0)}(u\circ\psi_w|_{V_0},v\circ\psi_w|_{V_0}),\quad
  u,v\in \ell(V_m).
\]
Here, $r_w:=r_{w_1}r_{w_2}\cdots r_{w_m}$ for $w=w_1w_2\cdots w_m\in W_m$ and $r_\emptyset=1$.
We refer to $(D,\bfr)$ as a \emph{regular harmonic structure} if $0<r_i<1$ for every $i\in S$ and 
\[
\cE^{(0)}(v,v)=\inf\{\cE^{(1)}(u,u)\mid u\in \ell(V_1)\mbox{ and }u|_{V_0}=v\}
\]
for every $v\in \ell(V_0)$.
Then, for $m\in\Z_+$ and $v\in \ell(V_m)$, we obtain 
\[
\cE^{(m)}(v,v)=\inf\{\cE^{(m+1)}(u,u)\mid u\in \ell(V_{m+1})\text{ and }u|_{V_m}=v\}.
\]
The existence of regular harmonic structures is a nontrivial problem.
See, e.g., \cite{Li90,Kus93,HMT06,Pe07,Ki01} for related studies. 
In particular, a class of nested fractals, which are self-similar sets that are realized in Euclidean spaces and have some good symmetries, have canonical regular harmonic structures.
For the precise definition of nested fractals, see \cite{Li90,Ki01}. All the fractals shown in Fig.~\ref{fig:pcf} except Hata's tree-like set are nested fractals and have regular harmonic structures.
Note that Hata's tree-like set also has regular harmonic structures (\cite[Example~3.1.6]{Ki01}).

Here we assume that a regular harmonic structure $(D,\bfr)$ is given.
Take a Borel probability measure $\mu$ on $K$ with full support and $\mu(V_*)=0$.
We can then define a strongly local regular Dirichlet form $(\cE,\cF)$ on $L^2(K,\mu)$ associated with $(D,\bfr)$ as
\begin{align*}
\cF&=\left\{f\in C(K)\subset L^2(K,\mu)\;\middle|\;
\lim_{m\to\infty}\cE^{(m)}(f|_{V_m},f|_{V_m})<\infty\right\},\\
\cE(f,g)&= \lim_{m\to\infty}\cE^{(m)}(f|_{V_m},g|_{V_m}),\quad f,g\in\cF.
\end{align*}
For further details, see the beginning of \cite[Section~3.4]{Ki01}.

The Dirichlet form $(\cE,\cF)$ constructed above satisfies the self-similarity: $\psi_i^* f\in\cF$ for each $i\in S$ and $f\in\cF$ and
\[
  \cE(f,g)=\sum_{i\in S}\frac1{r_i}\cE(\psi_i^* f,\psi_i^* g),\quad f,g\in\cF.
\]
This implies that 
\begin{equation}\label{eq:selfsimilar}
\cE(f,g)=\sum_{w\in W_m}\frac1{r_w}\cE(\psi_w^* f,\psi_w^* g),\quad f,g\in\cF,\ m\in\N
\end{equation}
and
\begin{equation}\label{eq:selfsimilarity2}
\nu_{f,g}=\sum_{w\in W_m}\frac1{r_w}(\psi_w)_* \nu_{\psi_w^* f,\psi_w^* g},\quad f,g\in\cF,\ m\in\N.
\end{equation}
For every $f\in\cF$, $\nu_f$ does not charge any one-point sets. This holds for arbitrary strongly local regular Dirichlet forms; see, e.g., \cite[Theorem~4.3.8]{CF12}.
Fix any minimal energy-dominant measure $\nu$ with $\nu(K)<\infty$. $\nu$ also does not charge any one-point sets. In particular, $\nu(V_*)=0$.

Let $U$ denote $K\setminus V_0$. We remark the following estimate.
\begin{proposition}\label{prop:equivalence}
There exists a positive constant $c_1$ such that
\begin{equation}\label{eq:equivalence1}
2\cE(f,f)=\nu_f(U)\ge c_1\Bigl(\osc_{U}f\Bigr)^2=c_1\Bigl(\osc_{K}f\Bigr)^2\quad\text{for any $f\in\cF$.}
\end{equation}
\end{proposition}
\begin{proof}
See \cite[Theorem~3.3.4]{Ki01}.
\end{proof}

We will confirm that (A1)--(A3) in Assumption~\ref{assumption} hold.
(A1) is obvious. For $n\in\N$, we define $\Lm_n=W_n$ and $U_w^{(n)}=\psi_w(U)$ for $w\in\Lm_n$.
Then, $\{U_w^{(n)}\}_{w\in\Lm_n}$ are disjoint open sets. Since $K\setminus\bigsqcup_{w\in\Lm_n}U_w^{(n)}\subset V_*$ and $(\mu+\nu)(V_*)=0$, $\{U_w^{(n)}\}_{w\in\Lm_n}$ is a partition of $K$.
(A2)~(a) holds from the construction of $\{U_w^{(n)}\}_{w\in\Lm_n}$.
(A2)~(b) holds for the following reason: for each $x\in K$, the sequence $\{\bigcup_{w\in W_m;\,x\in K_w}K_w\}_{m=0}^\infty$ is a fundamental system of neighborhoods of $x$. From the compactness of $K$, their diameters converge to $0$ uniformly; in particular, $\max_{w\in \Lm_n}\diam U_w^{(n)}$ converges to $0$ as $n\to\infty$.

Recall that $\cH_U$ is the space of functions in $\cF$ that are harmonic on $U$.
From the self-similarity of $(\cE,\cF)$, $\psi_w^* h\in\cH_U$ holds for any $w\in W_*$ and $h\in\cH_{U_w^{(n)}}$.
Let $\cD=\{h\in\cH_U\mid \cE(h,h)=1/2\text{ and }\int_K h\,\dd\mu=0\}$. Since $\cH_U$ is a finite dimensional space (indeed, the dimension of $\cH_U$ is $\#V_0$) and $h\mapsto \cE(h,h)^{1/2}+\left|\int_K h\,\dd\mu\right|$ provides a norm on $\cH_U$ from Proposition~\ref{prop:equivalence}, $\cD$ is a compact subset of $\cF$.
For $n\in\N$, let $Y_n$ denote $\bigcup_{w\in\hat W_n}K_w$, where $\hat W_n=\{w\in W_n\mid K_w\cap V_0=\emptyset\}$.

Let $h\in\cD$. Since $\bigcup_{n\in\N}Y_n=U$ and $\nu_h(U)=1$ from $\nu_h(V_0)=0$, there exists $n_0(h)\in\N$ such that $\nu_h(Y_n)>1/2$ for all $n\ge n_0(h)$. From the continuity of $\nu_h$ in $h$ (Proposition~\ref{prop:sl}~(i)), there exists $\dl>0$ such that $\nu_{\hat h}(Y_n)>1/2$ for all $\hat h\in\cF$ with $\cE(h-\hat h,h-\hat h)<\dl$.
From the compactness of $\cD$, there exists $n_1\in\N$ such that $\nu_h(Y_{n_1})>1/2$ for all $h\in\cD$. 
Now let $h \in\cH_U$ be arbitrary. If $\cE(h,h) = 0$, then $\nu_h = 0$ and there is nothing to prove. Otherwise, set $\tilde h := (2\cE(h,h))^{-1/2}(h - \int_K h \,\dd\mu)$. Then $\tilde h\in\cD$ and $\nu_{\tilde h}= (2\cE(h,h))^{-1} \nu_h$. Therefore,
$\nu_h(Y_{n_1}) = 2\cE(h,h) \nu_{\tilde h}(Y_{n_1}) > \cE(h,h) = \nu_h(K)/2$.
Thus we have
\begin{equation}\label{eq:UV}
 \nu_h(Y_{n_1})\ge\frac12\nu_h(K)\quad\text{for any $h\in\cH_U$}.
\end{equation}
We write $V$ for $Y_{n_1}$.

To confirm (A3), we provide some estimates.
It is easy to show from \eqref{eq:equivalence1} that every non-empty subset of $K$ has positive capacity.

For $n\in\N$ and $w\in\Lm_n$, we define $V_w^{(n)}:=\psi_w(V)\subset U_w^{(n)}$.
Then, for $h\in\cH_{U_w^{(n)}}$,
\[
  \nu_h(U_w^{(n)})=\frac1{r_w}\nu_{\psi_w^* h}(U)
  \le \frac2{r_w}\nu_{\psi_w^* h}(V)
  =2\nu_h(V_w^{(n)})
\]
from \eqref{eq:selfsimilarity2} and \eqref{eq:UV}. Thus, (A3) (a) holds.

We fix a function $g\in\cF$ such that $0\le g\le 1$, $g=1$ on $V$, and $g=0$ on $K\setminus U=V_0$.
Let $f_w\in\cF$ be defined as
\[
f_w(x)=\begin{cases}
g(\psi_w^{-1}(x)) & (x\in U_w^{(n)})\\
0 & (x\notin U_w^{(n)}).
\end{cases}
\]
Since $f_w=1$ on $V_w^{(n)}$ and $f_w=0$ on $K\setminus U_w^{(n)}$, we have
\begin{align}
\Cp(V_w^{(n)};U_w^{(n)})&\le \cE(f_w,f_w)\quad\text{(from Proposition~\ref{prop:qe}~(i))}\nonumber\\
&=\frac1{r_w}\cE(g,g)\quad\text{(from \eqref{eq:selfsimilar})}.
\label{eq:VU}
\end{align}
Moreover, for any $h\in\cH_{U_w^{(n)}}$,
\begin{align*}
\nu_h(U_w^{(n)})
&=\frac1{r_w}\nu_{\psi_w^*h}(U)\quad\text{(from \eqref{eq:selfsimilarity2})}\\
&\ge \frac{c_1}{r_w}\Bigl(\osc_{U}\psi_w^*h\Bigr)^2\quad\text{(from \eqref{eq:equivalence1})}\\
&=\frac{c_1}{r_w}\Bigl(\osc_{U_w^{(n)}}h\Bigr)^2.
\end{align*}
Therefore,
\[
\Cp(V_w^{(n)};U_w^{(n)}) \Bigl(\osc_{U_w^{(n)}}h\Bigr)^2\le \frac{\cE(g,g)}{c_1}\nu_h(U_w^{(n)})
\]
and (A3)~(b) holds.

Let $x\in V_w^{(n)}$ and $e_x$ denote $e_{\{x\};U_w^{(n)}}$. 
Then,
\begin{align*}
\Cp(\{x\};U_w^{(n)})
&=\cE(e_x,e_x)\\
&=\frac1{r_w}\cE(\psi_w^*e_x,\psi_w^*e_x)\quad\text{(from \eqref{eq:selfsimilar})}\\
&\ge \frac{c_1}{2 r_w}\Bigl(\osc_U \psi_w^*e_x\Bigr)^2\quad\text{(from \eqref{eq:equivalence1})}\\
&=\frac{c_1}{2 r_w}\\
&\ge \frac{c_1}{2\cE(g,g)}\Cp(V_w^{(n)};U_w^{(n)}).\quad\text{(from \eqref{eq:VU})}
\end{align*}
Therefore, (A3)~(c) holds.

In conclusion, we can apply Theorem~\ref{th:main} and conclude that the AF-martingale dimension is one. 
This shows that our arguments recover \cite[Theorem~4.10]{Hi13} within the framework of the present paper.

\subsection{Inhomogeneous Sierpinski gaskets}\label{sec:inh}
In this subsection, we introduce a class of inhomogeneous Sierpinski gaskets that includes genuinely non-self-similar members beyond the scope of earlier results for self-similar sets.

Fix an integer $d$ greater than $1$.
We take a closed regular $d$-simplex $\tilde K$ in $\R^d$.
The vertices of $\tilde K$ are denoted by $p_1,p_2,\dots,p_{d+1}$.
Let $l$ be an integer greater than $1$.
Let $K_i^{(l)} \subset \tilde K$, $i = 1, 2, \dots$, be the closed regular $d$-simplices obtained as follows:
we divide each edge of $\tilde K$ into $l$ equal parts, join the division points by line segments, and form hyperplanes parallel to the faces of $\tilde K$.
Among the resulting closed simplices whose sizes are $1/l$ times that of $\tilde K$, we retain those whose orientation is consistent with that of $\tilde K$, and remove all the others.
The number of simplices is denoted by $N(l)$. We can confirm that
\[
N(l)=\sum_{j_1=1}^l \sum_{j_2=1}^{j_1}\sum_{j_3=1}^{j_2}\cdots\sum_{j_{d}=1}^{j_{d-1}}1
=\frac{1}{d!}l(l+1)\cdots(l+d-1).
\]
In particular, $N(l)=l(l+1)/2$ when $d=2$ and $N(l)=l(l+1)(l+2)/6$ when $d=3$.
Concerning the indexing, $K_i^{(l)}$ $(i=1,2,\dots,d+1)$ is chosen so that $p_i\in K_i^{(l)}$, while the remaining indices $(i=d+2,d+3,\dots,N(l))$ are assigned arbitrarily. See Fig.~\ref{fig:triangles}.
\begin{figure}[t]
\centering
\includegraphics[width=1\textwidth]{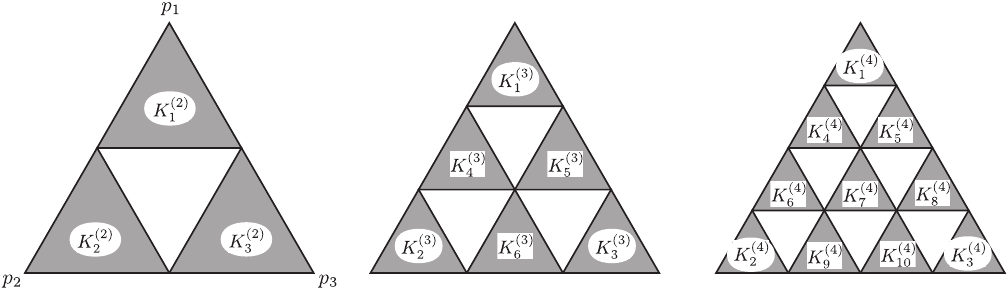}
\caption{(Quoted from \cite[Fig.~1]{HY22}) Illustration of $K_i^{(l)}$ ($i=1,2,\dots,N(l)$) when $d=2$ and $l=2,3,4$, respectively. Pay attention to the choice of $K_1^{(l)}$, $K_2^{(l)}$, and $K_3^{(l)}$.}
\label{fig:triangles}
\end{figure}
Let $\psi_i^{(l)}$, $i=1,2,\dots,N(l)$, be the contractive affine map from $\tilde K$ onto $K_i^{(l)}$ of type $\psi_i^{(l)}(z)=l^{-1}z+\a_i^{(l)}$ for some $\a_i^{(l)}\in \R^d$, which is uniquely determined. 
Note that $p_i$ is the fixed point of $\psi_i^{(l)}$ for $i=1,2,\dots,d+1$.

For a set $A$, $\ell(A)$ denotes the space of all real-valued functions on $A$.
Let 
\[
V_0=\{p_1,p_2,\dots,p_{d+1}\}\quad\text{and}\quad V^{(l)}=\bigcup_{i=1}^{N(l)}\psi_i^{(l)}(V_0).
\]
We further define
\begin{align*}
E_0&=\{\{p,q\}\mid p,q\in V_0\text{ and } p\ne q\},\\
E^{(l)}&=\left\{\{\psi_i^{(l)}(p),\psi_i^{(l)}(q)\}\mathrel{}\middle|\mathrel{} i\in\{1,2,\dots,N(l)\}\text{ and }\{p,q\}\in E_0\right\}.
\end{align*}
We introduce the following quadratic forms:
\begin{align*}
\cQ(f,g)&=\sum_{\{p,q\}\in E_0}(f(p)-f(q))(g(p)-g(q)),\quad f,g\in \ell(V_0),\\
\cQ^{(l)}(f,g)&=\sum_{\{p,q\}\in E^{(l)}}(f(p)-f(q))(g(p)-g(q)),\quad f,g\in \ell(V^{(l)}).
\end{align*}
Then, there exists a unique $r^{(l)}\in(0,1)$ such that for all $f\in \ell(V_0)$,
\[
\cQ(f,f)=\frac1{r^{(l)}}\inf\{\cQ^{(l)}(g,g)\mid g\in \ell(V^{(l)})\text{ and }g|_{V_0}=f\}.
\]
Indeed, the infimum on the right-hand side defines a quadratic form on $\ell(V_0)$ (the trace of $\cQ^{(l)}$ on $\ell(V_0)$), which should be equal to $\cQ$ up to a positive constant multiple by symmetry. For the proof of $r^{(l)}<1$, see e.g., \cite[Corollary~6.28]{Ba98} or \cite[Proposition~3.1.8]{Ki01}.
For example, $r^{(2)}=3/5$, $r^{(3)}=7/15$, and $r^{(4)}=41/103$ when $d=2$.
The asymptotics of $r^{(l)}$ as $l\to\infty$ is found in \cite[Theorem~2.2]{HK02}.

We now fix a non-empty finite subset $T$ of $\{l\in\N\mid l\ge 2\}$. For each $l\in T$, let $S^{(l)}$ denote the set of letters $i^l$ for $i=1,2,\dots,N(l)$. We should remark that $i^l$ is a letter and does not represent $\underbrace{i\cdots i}_l$.
We set $S=\bigcup_{l\in T} S^{(l)}$ and $\Sg= S^\N$. 

For each $v\in S$ the shift operator $\sg_v\colon \Sg\to\Sg$ is defined as $\sg_v(\om_1\om_2\cdots)=v\om_1\om_2\cdots$.
Let $W_0=\{\emptyset\}$ and $W_m=S^m$ for $m\in\N$, and define $W_*=\bigcup_{m\in\Z_+} W_m$. 
As in Section~\ref{sec:pcf}, $ww'\in W_{m+n}$ denotes $w_1w_2\cdots w_m w'_1w'_2\cdots w'_n$ for $w=w_1w_2\cdots w_m\in W_m$ and $w'=w'_1w'_2\cdots w'_n\in W_n$. 
For $\om=\om_1\om_2\cdots\in\Sg$ and $n\in\N$, let $[\om]_n$ denote $\om_1\om_2\cdots\om_n\in W_n$.
By convention, $[\om]_0:=\emptyset\in W_0$ for $\om\in\Sg$.

For $i^l\in S$ we define $\psi_{i^l}:=\psi_i^{(l)}$.
For $w=w_1w_2\cdots w_m\in W_m$, $\psi_w$ denotes $\psi_{w_1}\circ\psi_{w_2}\circ\dots\circ\psi_{w_m}$. Here $\psi_\emptyset$ is the identity map by definition.
For $w\in W_*$, $\tilde K_w$ denotes $\psi_w(\tilde K)$.
For $\om\in \Sg$, $\bigcap_{m\in\Z_+}\tilde K_{[\om]_m}$ is a one-point set, say $\{p\}$. The map $\Sg\ni \om\mapsto p\in \tilde K$ is denoted by $\pi$. The relation $\psi_{v}\circ\pi =\pi\circ \sg_v$ holds for $v\in S$.

We fix $L=\{L_w\}_{w\in W_*}\in T^{W_*}$. In other words, we assign $L_w\in T$ to each $w\in W_*$.
We set $\tilde W_0=\{\emptyset\}$ and
\[
\tilde W_m=\bigcup_{w\in\tilde W_{m-1}}\bigl\{w v\bigm|v\in S^{(L_w)}\bigr\},\quad m\in\N,
\]
inductively.
Define $\tilde W_*=\bigcup_{m\in\Z_+}\tilde W_m\subset W_*$, $\tilde\Sg=\{\om\in\Sg\mid [\om]_m\in\tilde W_m\text{ for all }m\in\Z_+\}$ and $G(L)=\pi(\tilde\Sg)$. Then it holds that
\[
G(L)=\bigcap_{m\in\Z_+}\bigcup_{w\in\tilde W_m}\tilde K_w.
\]
We call $G(L)$ the \emph{inhomogeneous Sierpinski gasket} generated by $L$, see Fig.~\ref{fig:inhsg}.
\begin{figure}[!htbp]
\centering
\includegraphics[width=0.5\textwidth]{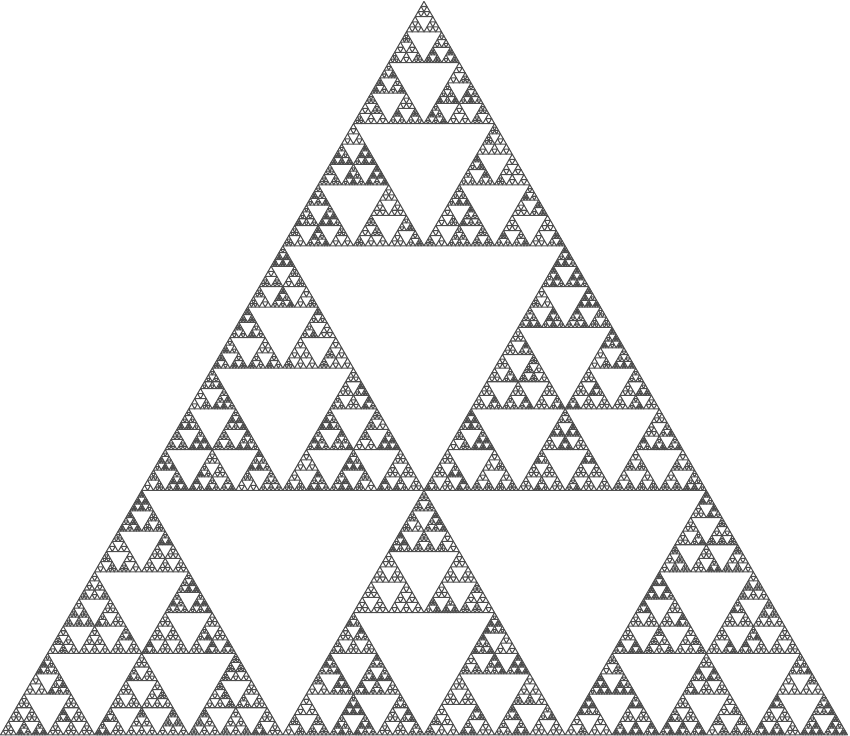}\bigskip\\
\includegraphics[width=0.5\textwidth]{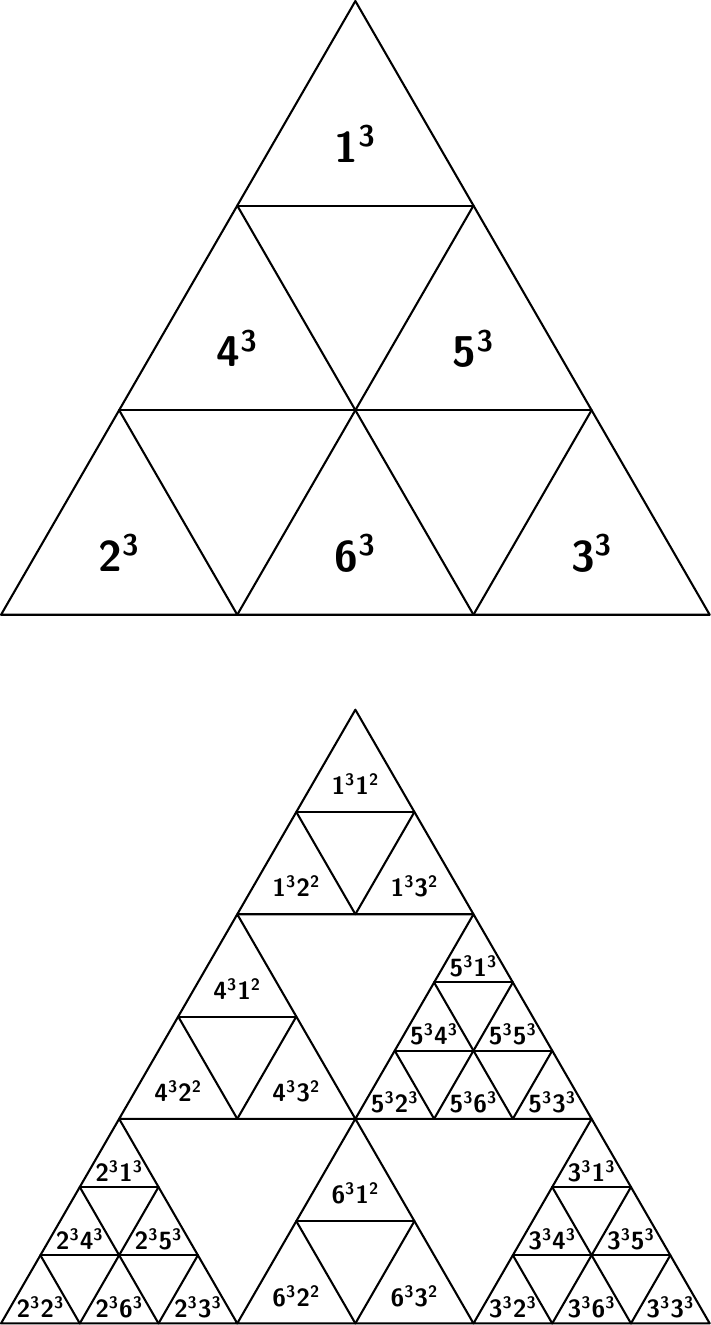}
\caption{(Adapted from \cite[Fig.~4]{Hi25}.) An example of inhomogeneous Sierpinski gaskets with $d=2$ and $T=\{2,3\}$ (the upper figure). Here, $L=\{L_w\}_{w\in W_*}$ is given by $L_\emptyset=3$, $L_{1^3}=L_{4^3}=L_{6^3}=2$, $L_{2^3}=L_{3^3}=L_{5^3}=3$, $L_{1^3 1^2}=L_{1^3 2^2}=L_{1^3 3^2}=2$, $L_{4^3 1^2}=2$, $L_{4^3 2^2}=L_{4^3 3^2}=3$, etc. The indices are indicated in the middle and lower figures.}
\label{fig:inhsg}
\end{figure}
We should note that only $\{L_w\}_{w\in\tilde W_*}$ among $\{L_w\}_{w\in W_*}$ is essential to define $G(L)$. When $L_w=l$ for all $w\in W_*$ for some $l\ge2$, the corresponding $G(L)$ is called the level $l$ Sierpinski gasket.

For $w\in\tilde W_*$, $K_w$ denotes $\tilde K_w\cap G(L)$.
Let 
\[
V_m=\bigcup_{w\in\tilde W_m}\psi_w(V_0)\quad\text{and}\quad
E_m=\left\{\{\psi_w(p),\psi_w(q)\}\mathrel{}\middle|\mathrel{} w\in \tilde W_m\text{ and }\{p,q\}\in E_0\right\}
\]
for $m\in\N$, and $V_*=\bigcup_{m\in\Z_+}V_m$. The closure of $V_*$ in $\R^d$ coincides with $G(L)$.

In the following, we write $K$ for $G(L)$.
Fix a finite Borel measure $\mu$ on $K$ with full support.
We can construct a canonical, strongly local regular Dirichlet form $(\cE,\cF)$ on $L^2(K,\mu)$~\cite{Ha97} in a similar manner to the previous subsection. For readers' convenience and later discussion, we explain some more details here, by following and modifying the arguments in \cite{Kus93}.

We set $r_{i^l}=r^{(l)}$ for $i^l\in S$ and $r_w=r_{w_1}r_{w_2}\cdots r_{w_m}$ for $w=w_1w_2\cdots w_m\in W_m$.
By convention, $r_\emptyset=1$.
For $m\in\Z_+$, we define a quadratic form $\cE^{(m)}$ on $\ell(V_m)$ as
\[
\cE^{(m)}(f,g)=\sum_{w\in\tilde W_m}\frac1{r_w}\cQ(f\circ\psi_w|_{V_0}, g\circ\psi_w|_{V_0}).
\]
Then, for $m\in\Z_+$ and $f\in \ell(V_m)$ it holds that
\[
\cE^{(m)}(f,f)=\inf\{\cE^{(m+1)}(g,g)\mid g\in \ell(V_{m+1})\text{ and }g|_{V_m}=f\}.
\]
In particular, $\{\cE^{(m)}(f|_{V_m},f|_{V_m})\}_{m\in\Z_+}$ is nondecreasing for every $f\in \ell(V_*)$.
Define
\[
\cF_*=\left\{f\in \ell(V_*)\mathrel{}\middle|\mathrel{}\lim_{m\to\infty}\cE^{(m)}(f|_{V_m},f|_{V_m})<\infty\right\}
\]
and for $f,g\in\cF_*$,
\begin{align*}
&\cE_*(f,g)\\*
&:=\lim_{m\to\infty}\cE^{(m)}(f|_{V_m},g|_{V_m})\\
&=\lim_{m\to\infty}\frac12\{\cE^{(m)}((f+g)|_{V_m},(f+g)|_{V_m})-\cE^{(m)}(f|_{V_m},f|_{V_m})-\cE^{(m)}(g|_{V_m},g|_{V_m})\}.
\end{align*}
Note that $(\cE_*,\cF_*)$ is a resistance form on $V_*$ from \cite[Definition~2.2.1, Theorem~2.2.6, Definition~2.3.1]{Ki01}.

The following proposition (the piecewise harmonic extension and maximum principle) is standard and the proof is omitted.
\begin{proposition}\label{prop:1}
Let $m\le n$. For each $f\in \ell(V_m)$ there exists a unique $h\in \ell(V_n)$ such that $h|_{V_m}=f$ and
\[
\cE^{(n)}(h,h)=\inf\{\cE^{(n)}(g,g)\mid g\in \ell(V_n)\text{ and }g|_{V_m}=f\}.
\]
Moreover, for any $w\in\tilde W_m$ and $x\in V_n\cap \tilde K_w$,
\[
\min_{\psi_w(V_0)}f\le h(x)\le \max_{\psi_w(V_0)}f.
\]
\end{proposition}

\begin{proposition}\label{prop:2}
There exists a positive constant $c_2$ independent of the choice of\break  $L=\{L_w\}_{w\in W_*}$ such that 
\[
  |f(x)-f(y)|\le c_2 r^{m/2}\cE_*(f,f)^{1/2}
\]
for all $f\in\cF_*$, $m\in\Z_+$, $w\in\tilde W_m$, and $x,y\in V_*\cap \tilde K_w$.
Here, $r:=\max_{l\in T}r^{(l)}\in(0,1)$.
\end{proposition}
\begin{proof}
First, we consider the case $x\in\psi_w(V_0)$.
There exist $n\in\Z_+$ and $v_i\in S$ $(i=1,2,\dots,n)$ such that $w v_1v_2\cdots v_n\in\tilde W_{m+n}$ and $y\in \psi_{w v_1v_2\cdots v_n}(V_0)$.
Let $x_0=x$ and $x_{n+1}=y$, and take 
\[
x_1\in \psi_{w}(V_0)\text{ and } 
x_i\in\psi_{w v_1v_2\cdots v_{i-1}}(V_0)\ (i=2,3,\dots,n).
\]
Note that $x_i\in V_{m+i-1}\subset V_{m+i}$ for $i=1,2,\dots,n+1$.
Since $T$ is a finite set, there exists $M\in\N$ depending only on $T$ satisfying the following: for each $i=0,1,2,\dots,n$, there exist $k\le M$ and $z_0,z_1,\dots,z_k\in V_{m+i}\cap \tilde K_{w^{(i)}}$ such that $z_0=x_i$, $z_k=x_{i+1}$, and $\{z_j,z_{j+1}\}\in E_{m+i}$ for all $j=0,1,\dots,k-1$, where $w^{(i)}:=w v_1v_2\cdots v_{i-1}$ for $i\ge1$ and $w^{(0)}:=w$. Then
\begin{align*}
|f(x_i)-f(x_{i+1})|
&\le \sum_{j=0}^{k-1}|f(z_j)-f(z_{j+1})|\\
&\le \sum_{j=0}^{k-1}\left\{r^{m+i}\cE^{(m+i)}(f|_{V_{m+i}},f|_{V_{m+i}})\right\}^{1/2}\\
&\le Mr^{(m+i)/2}\cE_*(f,f)^{1/2}.
\end{align*}
Therefore,
\begin{align}
|f(x)-f(y)|
&\le \sum_{i=0}^{n}|f(x_i)-f(x_{i+1})|\nonumber\\
&\le \sum_{i=0}^{n}Mr^{(m+i)/2}\cE_*(f,f)^{1/2}\nonumber\\
&\le \frac{M}{1-r^{1/2}}r^{m/2}\cE_*(f,f)^{1/2}.\label{eq:2}
\end{align}
For general $x$, we take $z\in\psi_w(V_0)$ and apply \eqref{eq:2} to the pairs $\{z,x\}$ and $\{z,y\}$ to obtain
\[
|f(x)-f(y)|\le |f(z)-f(x)|+|f(z)-f(y)| \le \frac{2M}{1-r^{1/2}}r^{m/2}\cE_*(f,f)^{1/2}.
\]
Thus, it suffices to take $c_2={2M}/(1-r^{1/2})$.
\end{proof}
\begin{proposition}\label{prop:3}
Each $f\in\cF_*$ is $\gamma$-H\"older continuous on $V_*$, where $\gamma=-(\log r)/(2\log l_{\max})$ and $l_{\max}=\max T$.
In particular, it is uniformly continuous on $V_*$.
\end{proposition}
\begin{proof}
Let $\diam A$ denote the Euclidean diameter of $A\subset\R^d$.
We first justify the geometric observation used below.
Since $T$ is finite, we can choose $a\in(0,1]$, depending only on $d$ and $T$, with the following properties.
For any simplex arising in the construction, let $A$ and $B$ be distinct first-generation sub-simplices.
If $A\cap B=\emptyset$, then
$\dist(A,B) \ge a\min\{\diam A,\diam B\}$.
If $p$ is a vertex of the parent simplex and $p\notin A$, then
$\dist(p,A)\ge a\diam A$.
Lastly, if $A\cap B=\{p\}$, then
$|z-z'|\ge a\max\{|z-p|,|z'-p|\}$ for $z\in A$ and $z'\in B$.
This follows from the fact that the closed tangent cones of $A$ and $B$ are separated by a positive angle.

The second estimate also holds for every descendant simplex $D$ not containing a given vertex $p$ of an ancestor simplex.
Indeed, apply it at the first step at which the descendant ceases to contain $p$, and use the fact that subsequent descendants have smaller diameter.

Now let $\tilde K_w$ and $\tilde K_{w'}$ be disjoint level-$m$ simplices.
Consider their first distinct ancestors $A$ and $B$, which are first-generation sub-simplices of their common parent.
If $A$ and $B$ are disjoint, the first estimate applies.
Otherwise, they meet at a common vertex $p$, and at least one of $\tilde K_w$ and $\tilde K_{w'}$ does not contain $p$.
Combining the preceding vertex estimate with the last estimate above gives, in either case,
\[
 \dist(\tilde K_w,\tilde K_{w'})
 \ge a^2 \min\{\diam \tilde K_w,\diam \tilde K_{w'}\}.
\]
Consequently, with $c_3=a^2/2$, the condition
$ |x-y|\le c_3\min_{w\in\tilde W_m}\diam \tilde K_w$
implies that any level-$m$ simplices containing $x$ and $y$ intersect.
Their intersection contains a common vertex, which belongs to both corresponding cells $K_w$ and $K_{w'}$.

Let $\dl_0=c_3 \diam \tilde K$.
Take any $x,y\in V_*$ with $0<|x-y|\le \dl_0$.
Let $\dl=|x-y|$.
Take the largest $m\in\Z_+$ such that $\dl_0 l_{\max}^{-m}\ge \dl$.
Since $\dl_0 l_{\max}^{-m-1}<\dl$, it holds that $m>\{\log(\dl_0/\dl)\}/(\log l_{\max})-1$.
Since the contraction ratio of $\psi_v$ with any $v\in S^{(l)}$ for $l\in T$ is $1/l\,(\ge 1/l_{\max})$, we have $\diam \tilde K_w\ge l_{\max}^{-m}\diam \tilde K$ for all $w\in\tilde W_m$. Then, 
\[
c_3\min_{w\in\tilde W_m}\diam \tilde K_w\ge \dl_0 l_{\max}^{-m}\ge\dl=|x-y|.
\]
This implies that there exist $w,w'\in \tilde W_m$ such that $x\in K_w$, $y\in K_{w'}$, and $K_w\cap K_{w'}\ne\emptyset$.
Choose $z\in \psi_w(V_0)\cap \psi_{w'}(V_0)$.
From Proposition~\ref{prop:2},
\[
|f(x)-f(y)|\le|f(x)-f(z)|+|f(z)-f(y)|\le 2c_2 r^{m/2}\cE_*(f,f)^{1/2}.
\]
Since 
\[
r^{m/2}
<r^{[\{\log(\dl_0/\dl)\}/(\log l_{\max})-1]/2}
=r^{-1/2}\left(\frac{\dl}{\dl_0}\right)^\gm
=r^{-1/2}\dl_0^{-\gm}|x-y|^\gm,
\]
$|f(x)-f(y)|\le c_4 |x-y|^\gm$ with $c_4=2c_2r^{-1/2}\dl_0^{-\gm}\cE_*(f,f)^{1/2}$.

For $x,y\in V_*$ with $|x-y|>\dl_0$, Proposition~\ref{prop:2} with $m=0$ gives
\[
 |f(x)-f(y)|
 \le c_2\cE_*(f,f)^{1/2}
 \le c_2\dl_0^{-\gamma}\cE_*(f,f)^{1/2}|x-y|^\gamma.
\]
Thus, we conclude that $f$ is $\gm$-H\"older continuous on $V_*$.
\end{proof}
From this proposition, any $f\in\cF_*$ can extend to a continuous function on $K$ uniquely.
We set
\[
\cF=\{f\in C(K)\mid f|_{V_*}\in\cF_*\}
\quad\text{and}\quad
\cE(f,g)=\cE_*(f|_{V_*},g|_{V_*}),\ f,g\in\cF.
\]
Proposition~\ref{prop:2} immediately implies the following.
\begin{proposition}\label{prop:poincare}
For any $f\in\cF$,
\[
\osc_{K} f \le c_2 \cE(f,f)^{1/2}.
\]
\end{proposition}
This proposition in particular implies that every non-empty subset of $K$ has positive capacity.

Combining Proposition~\ref{prop:1} and Proposition~\ref{prop:3}, we obtain the following.
\begin{proposition}\label{prop:4}
Let $m\in\Z_+$. For each $f\in \ell(V_m)$ there exists a unique $h\in \cF$ such that $h|_{V_m}=f$ and $\cE(h,h)=\cE^{(m)}(f,f)$.
Moreover, for any $w\in\tilde W_m$ and $x\in K_w$,
\begin{equation}\label{eq:max}
\min_{\psi_w(V_0)}f\le h(x)\le \max_{\psi_w(V_0)}f.
\end{equation}
\end{proposition}
Such $h$ will be called an $m$-harmonic function. Since $\{(V_m, E^{(m)})\}_{m\in\Z_+}$ is a compatible sequence, the $m$-harmonic function is the unique minimizer of $\cE$ among $g\in\cF$ with $g|_{V_m} = f$, exactly as in the p.c.f.\ case (see \cite[Section~2.2]{Ki01}); in particular, $\cH_{K\setminus V_m}$ consists precisely of the $m$-harmonic functions.
\begin{proposition}\label{prop:5}
By regarding $C(K)$ as a subspace of $L^2(K,\mu)$, $(\cE,\cF)$ is a strongly local regular Dirichlet form on $L^2(K,\mu)$.
\end{proposition}
\begin{proof}
First, we prove the closedness.
Suppose that a sequence $\{f_n\}_{n=1}^\infty$ in $\cF$ satisfies
\[
\cE(f_m-f_n,f_m-f_n)+\|f_m-f_n\|_{L^2(K,\mu)}^2\to0
\quad \text{as }m,n\to\infty.
\]
Since $\{\int_K (f_m-f_n)\,\dd\mu\}^2\le \mu(K)\|f_m-f_n\|_{L^2(K,\mu)}^2$, $\{\int_K f_n\,\dd\mu\}_{n=1}^\infty$ is a Cauchy sequence, hence converges as $n\to\infty$.
Let $g_n=f_n-\mu(K)^{-1}\int_K f_n\,\dd\mu$ for $n\in\N$. Then
\begin{align*}
\|g_m-g_n\|_\infty &\le \osc_{K}(g_m-g_n)
=\osc_{K}(f_m-f_n)\\
&\le c_2\cE(f_m-f_n,f_m-f_n)^{1/2}\quad\text{(from Proposition~\ref{prop:poincare})}\\
&\to 0\quad \text{as }m,n\to\infty.
\end{align*}
Therefore, $\{g_n\}_{n=1}^\infty$ converges uniformly, so does $\{f_n\}_{n=1}^\infty$. Let $f\in C(K)$ denote the limit of $\{f_n\}_{n=1}^\infty$.
Then, for every $k\in\Z_+$,
\begin{align*}
\cE^{(k)}((f-f_n)|_{V_k},(f-f_n)|_{V_k})
&=\lim_{m\to\infty}\cE^{(k)}((f_m-f_n)|_{V_k},(f_m-f_n)|_{V_k})\\
&\le\varliminf_{m\to\infty}\cE(f_m-f_n,f_m-f_n).
\end{align*}
Since the right-hand side is independent of $k$, it holds that $f-f_n\in\cF$, in particular, $f\in\cF$. Moreover,
\begin{align*}
\cE(f-f_n,f-f_n)&=\lim_{k\to\infty}\cE^{(k)}((f-f_n)|_{V_k},(f-f_n)|_{V_k})\\
&\le \varliminf_{m\to\infty}\cE(f_m-f_n,f_m-f_n)
\to 0\quad \text{as }n\to\infty.
\end{align*}
Since $f_n\to f$ in $L^2(K,\mu)$ by the uniform convergence, this implies that $f_n$ converges to $f$ in $\cF$, which means the closedness of $(\cE,\cF)$.

The symmetry, the Markov property, and the strong locality of $(\cE,\cF)$ follow from the definition.

Concerning the regularity, it suffices to prove that $\cF$ is dense in $C(K)$ since $\cF\subset C(K)$.
Take an arbitrary $f\in C(K)$. Given $\eps>0$, there exists $\dl>0$ such that $x,y\in K$ with $|x-y|\le\dl$ implies $|f(x)-f(y)|\le \eps$ from the uniform continuity of $f$.
Take $m\in\Z_+$ such that $\max_{w\in\tilde W_m}\diam\tilde K_w\le\dl$.
Take $h\in\cF$ in Proposition~\ref{prop:4} with $f$ replaced by $f|_{V_m}$.
Choose any $x\in K$. There exists $w\in\tilde W_m$ such that $x\in K_w$. Take any $y\in \psi_w(V_0)$.
From \eqref{eq:max},
\[
|h(x)-f(y)|\le\osc_{\psi_w(V_0)}f\le \eps.
\]
Combining this inequality with $|f(x)-f(y)|\le\eps$, we obtain that $|h(x)-f(x)|\le 2\eps$. Therefore, $\|h-f\|_\infty\le 2\eps$. This implies that $\cF$ is dense in $C(K)$.
\end{proof}
For $v\in\tilde W_*$, let $L^{[v]}$ denote $\{L_{vw}\}_{w\in W_*}\in T^{W_*}$. We write $K^{[v]}$ for $G(L^{[v]})$ and define a Borel measure $\mu^{[v]}$ on $K^{[v]}$ by
\[
\mu^{[v]}=\frac1{\mu(K_v)}\bigl((\psi_v|_{K^{[v]}})^{-1}\bigr)_*(\mu|_{K_v}).
\]
By considering $(K^{[v]},\mu^{[v]})$ in place of $(K,\mu)$, we can define a canonical Dirichlet form $(\cE^{[v]},\cF^{[v]})$ on $L^2(K^{[v]},\mu^{[v]})$. The energy measure of $f\in\cF^{[v]}$ associated with $(\cE^{[v]},\cF^{[v]})$ is denoted by $\nu_f^{[v]}$.

For $f\in\cF$, $m\in\Z_+$, and $v\in\tilde W_m$, let $f^{[v]}=f\circ\psi_v|_{K^{[v]}}$.
From the definition, the following holds.
\begin{lemma}[cf.\ {\cite[Lemma~3.2]{HY22}}]\label{lem:6}
$f^{[v]}\in\cF^{[v]}$. Moreover, it holds that
\[
\cE(f,f)=\sum_{v\in\tilde W_m}\frac1{r_v}\cE^{[v]}(f^{[v]},f^{[v]}).
\]
Furthermore, if $f$ is harmonic on $K_v\setminus \psi_v(V_0)$ with respect to $(\cE,\cF)$, then $f^{[v]}$ is harmonic on $K^{[v]}\setminus V_0$ with respect to $(\cE^{[v]},\cF^{[v]})$.
\end{lemma}
From this lemma, it holds that
\[
\int_K g\,\dd\nu_f=\sum_{v\in\tilde W_m}\frac1{r_v}\int_{K^{[v]}}g\circ \psi_v|_{K^{[v]}}\,\dd\nu_{f^{[v]}}^{[v]},\quad g\in\cF,
\]
and therefore
\begin{equation}\label{eq:nu}
\nu_f(A)=\sum_{v\in\tilde W_m}\frac1{r_v}\nu_{f^{[v]}}^{[v]}\bigl((\psi_v|_{K^{[v]}})^{-1}(A\cap K_v)\bigr)
\end{equation}
for Borel sets $A$ of $K$.
Noting that energy measures for strongly local Dirichlet forms do not have mass on singletons, \eqref{eq:nu} implies that
\[
\nu_f(A)=\frac1{r_v}\nu_{f^{[v]}}^{[v]}\bigl((\psi_v|_{K^{[v]}})^{-1}(A)\bigr)
\]
if the Borel set $A$ is a subset of $K_v$ for some $v\in\tilde W_*$.
In particular, by letting $A=K_v$,
\begin{equation}\label{eq:A}
\nu_f(K_v)=\frac1{r_v}\nu_{f^{[v]}}^{[v]}(K^{[v]})
=\frac2{r_v}\cE^{[v]}(f^{[v]},f^{[v]}).
\end{equation}
\begin{lemma}\label{lem:7}
Let $w\in\tilde W_*$ and $h\in\cF$ be harmonic on $K_w\setminus \psi_w(V_0)$ with respect to $(\cE,\cF)$. Then,
\[
\nu_h(K_w)=\frac2{r_w} \cQ(h\circ\psi_w|_{V_0},h\circ\psi_w|_{V_0}).
\]
\end{lemma}
\begin{proof}
By Lemma~\ref{lem:6} and the characterization following Proposition~\ref{prop:4} applied to\break $(K^{[w]},\cE^{[w]},\cF^{[w]})$ with $m = 0$, $\cE^{[w]}(h^{[w]},h^{[w]}) = \cQ(h\circ\psi_w|_{V_0}, h\circ\psi_w|_{V_0})$; the claim follows from \eqref{eq:A}.
\end{proof}
Recall Proposition~\ref{prop:1}.
For $l\ge2$ and $f\in \ell(V_0)$, there exists a unique $h\in \ell(V^{(l)})$ such that $h|_{V_0}=f$ and $\cQ(f,f)=\frac1{r^{(l)}}\cQ^{(l)}(h,h)$. Then, for $v\in S^{(l)}$, the map
\[
A_v\colon \ell(V_0)\ni f\mapsto h\circ\psi_v|_{V_0}\in \ell(V_0)
\]
is linear.
By identifying $\ell(V_0)$ with $\R^{d+1}$, $A_v$ is regarded as a square matrix of size $d+1$.
As usual we write 
\[
A_{v_1v_2\cdots v_m}=A_{v_m}A_{v_{m-1}}\cdots A_{v_1}
\quad (v_1,v_2,\dots,v_m\in S).
\]
\begin{lemma}\label{lem:9}
For every $l\in T$, $i\in\{1,\ldots,d+1\}$, and $u\in\ell(V_0)$,
\[
 \cQ(A_{i^l}u,A_{i^l}u) \le (r^{(l)})^2\cQ(u,u).
\]
\end{lemma}

\begin{proof}
Let $\boldsymbol{1}\in\ell(V_0)$ be the constant function with value one, and set
$ Y_i=\{y\in\ell(V_0)\mid (y,\bfone)_{\ell(V_0)}=0,\ y(p_i)=0\}$.
Define $v_i\in\ell(V_0)$ by
\[
 v_i(p_k)=
 \begin{cases}
  0,   & k=i,\\
  1/d, & k\ne i.
 \end{cases}
\]
Then
$A_{i^l}\bfone=\bfone$ and $A_{i^l}v_i=r^{(l)}v_i$; see \cite[Proposition~A.1.1 and Lemma~A.1.5]{Ki01} and \cite[Lemma~5]{HN06}.
Moreover, symmetry under permutations of the vertices other than $p_i$ implies that $A_{i^l}y=s^{(l)}y$ for $y\in Y_i$ for some $s^{(l)}$; moreover $|s^{(l)}|<r^{(l)}$ by \cite[Theorem~A.1.2]{Ki01}.

Every $u\in\ell(V_0)$ has a unique decomposition
\[
 u=a\bfone+bv_i+y,
 \qquad a,b\in\mathbb{R},\quad y\in Y_i.
\]
For $f,g\in\ell(V_0)$, direct calculation gives
\[
 \cQ(f,g)
 =(d+1)\sum_{k=1}^{d+1}f(p_k)g(p_k)
 -\left(\sum_{k=1}^{d+1}f(p_k)\right)
  \left(\sum_{k=1}^{d+1}g(p_k)\right).
\]
Consequently,
$\cQ(\bfone,g)=0$ for every $g$ and
$\cQ(v_i,y)=0$ for $y\in Y_i$.
It follows that
\begin{align*}
 \cQ(A_{i^l}u,A_{i_l}u)
 &=b^2(r^{(l)})^2\cQ(v_i,v_i)+(s^{(l)})^2\cQ(y,y)\\
 &\le (r^{(l)})^2\bigl(b^2\cQ(v_i,v_i)+\cQ(y,y)\bigr)\\
 &=(r^{(l)})^2\cQ(u,u).\qedhere
\end{align*}
\end{proof}
\begin{lemma}\label{lem:10}
For each $c>0$, there exists $N\in\N$ such that for any $m\in\Z_+$, $w\in\tilde W_m$, $h\in \cH_{K_w\setminus\psi_w(V_0)}$, and $i\in\{1,2,\dots,d+1\}$, it holds
\[
\nu_h(K_{w i^{l_1}i^{l_2}\cdots i^{l_N}})\le c\nu_h(K_w),
\]
where $l_1,l_2,\dots,l_N\in T$ are uniquely determined so that $w i^{l_1}i^{l_2}\cdots i^{l_N}\in\tilde W_{m+N}$.
\end{lemma}

\begin{proof}
Let $r=\max_{l\in T}r^{(l)}\in(0,1)$, and choose $N\in\N$ such that $r^N\le c$.
Fix $m,w,h,i$ as in the statement, and put $u=h\circ\psi_w|_{V_0}$ and $v=i^{l_1}\cdots i^{l_N}$.

By Lemma~\ref{lem:6} and the characterization following Proposition~\ref{prop:4}, the restriction of $h$ to $K_w$ is obtained by harmonic extension from its boundary values $u$.
Hence
$h\circ\psi_{wv}|_{V_0}=A_vu$.
Applying Lemma~\ref{lem:7} to $w$ and $wv$, respectively, we obtain
\[
 \nu_h(K_w)=\frac{2}{r_w}\cQ(u,u),
 \qquad
 \nu_h(K_{wv})=\frac{2}{r_wr_v}\cQ(A_vu,A_vu).
\]
Iterating Lemma~\ref{lem:9} gives $\cQ(A_vu,A_vu) \le r_v^2\cQ(u,u)$.
Therefore,
\[
 \nu_h(K_{wv})
 \le r_v\nu_h(K_w)
 \le r^N\nu_h(K_w)
 \le c\nu_h(K_w). \qedhere
\]
\end{proof}
In what follows, we assume that $\mu(V_*)=0$. One of the natural choices of $\mu$ is given by
\begin{equation}\label{eq:mu}
\mu(K_w)=\prod_{j=1}^m N(l_j)^{-1}\quad\text{for }
w=i_1^{l_1}i_2^{l_2}\cdots i_m^{l_m}\in\tilde W_m.
\end{equation}
We now confirm (A1)--(A3) in Assumption~\ref{assumption}.
(A1) is obvious.
Take $N\in\N$ in Lemma~\ref{lem:10} for $c=\{2(d+1)\}^{-1}$.
We set $\Lm_n=\tilde W_n$ for $n\in\N$.
For each $w\in\Lm_n$, we define
\[
U_w^{(n)}=K_w\setminus \psi_w(V_0)\quad\text{and}\quad
V_w^{(n)}=\text{the closure of }\left(K_w\setminus\bigcup_{i=1}^{d+1}K_{w v^{(i)}}\right)\,(\subset U_w^{(n)}),
\]
where $v^{(i)}=i^{l_1(w,i)}i^{l_2(w,i)}\cdots i^{l_N(w,i)}\in W_N$ and $l_1(w,i),l_2(w,i),\dots,l_N(w,i)\in T$ are uniquely determined so that $w v^{(i)}\in \tilde W_{n+N}$.

Then, (A2)~(a) holds from the construction.
(A2)~(b) holds since $\diam K_w\le 2^{-n}\diam \tilde K$ for $n\in \N$ and $w\in \Lm_n$.

(A3)~(a) holds because for $w\in\Lm_n$ and $h\in\cH_{U_w^{(n)}}$,
\begin{align*}
\nu_h(V_w^{(n)})
&=\nu_h(K_w)-\sum_{i=1}^{d+1}\nu_h(K_{w v^{(i)}})\\
&\ge \nu_h(K_w)-\sum_{i=1}^{d+1}\frac{1}{2(d+1)}\nu_h(K_w)
\qquad\text{(from the choice of $N$)}\\
&=\frac12\nu_h(K_w)
=\frac12\nu_h(U_w^{(n)}).
\end{align*}
For confirming (A3)~(b) and (c), we prepare the following:
\begin{proposition}\label{prop:11}
There exists a positive constant $c_5$ such that, for every $w\in\tilde W_*$ and $f\in\cF$,
\[
\nu_f(K_w)\ge\frac{c_5}{r_w}\Bigl(\osc_{K_w}f\Bigr)^2.
\]
\end{proposition}
\begin{proof}
Let $f^{[w]}=f\circ\psi_w|_{K^{[w]}}$. Then, $f^{[w]}\in\cF^{[w]}$ from Lemma~\ref{lem:6} and
\begin{align*}
\nu_f(K_w)
&=\frac2{r_w}\cE^{[w]}(f^{[w]},f^{[w]})
\quad\text{(from \eqref{eq:A})}\\
&\ge \frac2{r_w}\cdot c_2^{-2}\Bigl(\osc_{K^{[w]}}f^{[w]}\Bigr)^2
\quad\text{(from Proposition~\ref{prop:poincare})}\\
&=\frac{2c_2^{-2}}{r_w}\Bigl(\osc_{K_w}f\Bigr)^2.\qedhere
\end{align*}
\end{proof}
We will confirm (A3)~(b). Suppose $w\in \tilde W_n$ for $n\in\N$. 
Let $f\in \ell(V_{n+N})$ be defined by $f=1$ on $V_{n+N}\cap V^{(n)}_w$ and $f=0$ on $V_{n+N}\setminus V^{(n)}_w$; note that $\psi_w(V_0)\cap V^{(n)}_w=\emptyset$.
Let $g\in\cF$ be the $(n+N)$-harmonic function with $g|_{V_{n+N}}=f$, which exists by Proposition~\ref{prop:4}.
By \eqref{eq:max}, $g=1$ on every cell $K_v$, $v\in\tilde W_{n+N}$, with $K_v\subset V^{(n)}_w$, and $g=0$ on every cell $K_v$ with $K_v\not\subset K_w$; hence $g=1$ on $V^{(n)}_w$ and $g=0$ on $K\setminus U^{(n)}_w$.
Moreover, by Proposition~\ref{prop:4},
\[
  \cE(g,g)
  =\cE^{(n+N)}(f,f)
  =\sum_{v\in\tilde W_{n+N}}\frac{1}{r_v}
    \cQ\bigl(f\circ\psi_v|_{V_0},f\circ\psi_v|_{V_0}\bigr).
\]
If $K_v\not\subset K_w$ or $K_v\subset V^{(n)}_w$, then $f\circ\psi_v|_{V_0}$ is constant and the corresponding term vanishes.
The remaining indices are $v=wv^{(i)}$, $i=1,2,\dots,d+1$, for which $f\circ\psi_{wv^{(i)}}|_{V_0}$ takes the value $0$ at $p_i$ and $1$ at the other $d$ vertices of $V_0$, so that $\cQ\bigl(f\circ\psi_{wv^{(i)}}|_{V_0},f\circ\psi_{wv^{(i)}}|_{V_0}\bigr)=d$.
Therefore,
\[
  \Cp\bigl(V^{(n)}_w;U^{(n)}_w\bigr)
  \le \cE(g,g)
  = \frac{d}{r_w}\sum_{i=1}^{d+1}\frac{1}{r_{v^{(i)}}}
  \le \frac{c_6}{r_w},
  \qquad c_6:=d(d+1)\Bigl(\min_{l\in T}r^{(l)}\Bigr)^{-N},
\]
where the first inequality follows from Proposition~\ref{prop:qe}~(i).
Note that $c_6$ depends only on $d$, $T$, and $N$.

Let $h\in\cH_{U_w^{(n)}}$.
Then,
\begin{align*}
\nu_h(U_w^{(n)})&=\nu_h(K_w)\ge\frac{c_5}{r_w}\Bigl(\osc_{K_w}h\Bigr)^2
\quad\text{(from Proposition~\ref{prop:11})}\\
&\ge\frac{c_5}{c_6}\Cp(V_w^{(n)};U_w^{(n)})\Bigl(\osc_{U_w^{(n)}}h\Bigr)^2.
\end{align*}
Therefore, (A3)~(b) holds.

Lastly, we will confirm (A3)~(c). Let $x\in V_w^{(n)}$ and $e_x$ denote $e_{\{x\};U_w^{(n)}}$. 
Then,
\begin{align*}
\Cp(\{x\};U_w^{(n)})
&=\cE(e_x,e_x)\\
&=\frac1{r_w}\cE^{[w]}(e_x\circ\psi_w|_{K^{[w]}},e_x\circ\psi_w|_{K^{[w]}})\quad\text{(from Lemma~\ref{lem:6})}\\
&\ge \frac{1}{c_2^2 r_w}\left(\osc_{K^{[w]}}e_x\circ\psi_w|_{K^{[w]}}\right)^2\quad\text{(from Proposition~\ref{prop:poincare})}\\
&=\frac{1}{c_2^2 r_w}\\
&\ge \frac{1}{c_2^2 c_6}\Cp(V_w^{(n)};U_w^{(n)}).
\end{align*}
Therefore, (A3)~(c) holds.

Thus, the following holds from Theorem~\ref{th:main}.
\begin{theorem}
Let $\mu$ be a finite Borel measure on $G(L)$ with full support and $\mu(V_*)=0$.
Then, the AF-martingale dimension corresponding to $(\cE,\cF)$ on $L^2(G(L),\mu)$ is one.
\end{theorem} 
\begin{remark}\label{rem:inh}
We provide a few remarks on the inhomogeneous Sierpinski gaskets $G(L)$.
\begin{enumerate}
\item In \cite[Theorem~1.2]{Ha97}, the quantitative estimate of the heat kernel on $G(L)$ is obtained for ``almost all'' $L\in T^{W_*}$ when $d=2$, but is not as sharp as that for the standard Sierpinski gasket. The upper and lower bounds involve terms of different forms.
\item By applying Frostman's lemma to the measure $\mu$ given in \eqref{eq:mu}, we can prove that the Hausdorff dimension of $G(L)$ is greater than or equal to
\[
\min_{l\in T}\frac{\log N(l)}{\log l}\ge \frac{\log(d+1)}{\log2}
\]
(since $\log N(l)$ is a convex function of $\log l$ vanishing at $l=1$),
the right-hand side of which diverges to $+\infty$ as $d\to\infty$. This shows the considerable difference between the Hausdorff dimension and the martingale dimension.
\end{enumerate}
\end{remark}

\subsection{The two-dimensional standard Sierpinski carpet}
\label{sec:carpet}

In this subsection, we verify (A1), (A2), and (A3') for the canonical Dirichlet form on the two-dimensional standard Sierpinski carpet (see Fig.~\ref{fig:carpet}), thereby proving that its AF-martingale dimension is equal to one by Theorem~\ref{th:main}. 
We keep the notation on self-similar structures introduced at the beginning of Section~\ref{sec:pcf}.
\begin{figure}
\includegraphics[height=25ex]{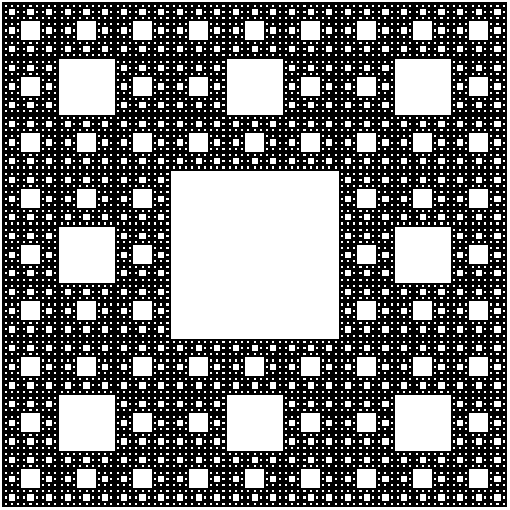}
\caption{The two-dimensional standard Sierpinski carpet.}
\label{fig:carpet}
\end{figure}

Let $q_1,q_2,\dots,q_8$ be an enumeration of
$\{0,1/2,1\}^2\setminus\{(1/2,1/2)\}$ and define
$\psi_i\colon\R^2\to\R^2$ by $\psi_i(z)=(z+2q_i)/3$ for
$i\in S:=\{1,2,\dots,8\}$. The two-dimensional standard Sierpinski
carpet $K$ is the unique non-empty compact set satisfying
$K=\bigcup_{i\in S}\psi_i(K)$; it is connected, and
$\cL=(K,S,\{\psi_i\}_{i\in S})$ is a self-similar structure.
Let $\mu$ be the self-similar probability measure on $K$ with equal
weights $(1/8,\dots,1/8)$, that is, the normalized
$d_{\mathrm f}$-dimensional Hausdorff measure, where
$d_{\mathrm f}=\log8/\log3$; $\mu$ has full support and
$\mu(K_w)=8^{-|w|}$ for $w\in W_*$. For $w\in W_m$ we write
$Q_w:=\psi_w([0,1]^2)$, so that $K_w=K\cap Q_w$. Let
\[
\partial_0K:=[0,1]^2\setminus(0,1)^2,\qquad
\Gamma_m:=\bigcup_{w\in W_m}\psi_w(\partial_0K)\quad(m\in\Z_+),
\qquad
\Gamma_\infty:=\bigcup_{m\in\Z_+}\Gamma_m .
\]

We first record some elementary geometric facts.
For
$w,w'\in W_m$ with $w\ne w'$, write $Q_w=\iota_w+[0,3^{-m}]^2$ with
$\iota_w\in3^{-m}\Z^2$; we say that $w$ and $w'$ are \emph{neighbors} if $\|3^m(\iota_w-\iota_{w'})\|_\infty\le1$, where $\|(a,b)\|_\infty:=|a|\vee|b|$. Then:
\begin{enumerate}
\item[(G1)] $K_w\cap K_{w'}\ne\emptyset$ if and only if $w$ and $w'$
are neighbors. 
In particular, if
$z\in K_w$, $z'\in K_{w'}$, and $|z-z'|<3^{-m}$, then $w$ and $w'$ are neighbors.
\item[(G2)] For $w\in W_m$ and $M\in\Z_+$, let
\[
\cN_M(K_w):=\bigcup\bigl\{K_{w'}\;\big|\;
w'\in W_m,\ D_m(w',w)\le M\bigr\},
\]
where $D_m$ denotes the graph distance on $W_m$ with respect to the
neighbor relation (with $D_m(w',w):=\infty$ if there is no connecting
chain). Since each neighbor step changes $3^m\iota$ by at most one in
each coordinate, $\cN_M(K_w)$ is contained in the square
of side $(2M+1)3^{-m}$ concentric with $Q_w$; in particular,
$\diam\cN_M(K_w)\le(2M+1)\sqrt2\,3^{-m}$.
\item[(G3)] $\cN_M(K_w)\subset \Int_K\cN_{M+1}(K_w)$ for every $M\in\Z_+$,
where $\Int_K$ denotes the interior in $K$. 
Also, for $x\in K_w$, $B(x,3^{-m})\cap K\subset\cN_1(K_w)$, where $B(x,r)$ denotes the open Euclidean ball with center $x$ and radius $r$. Finally, $\cN_M(K_w)\setminus \Int_K\cN_M(K_w)\subset\Gamma_m$.
\end{enumerate}

The analytic input is summarized in the following theorem. Only the
stated properties of $(\cE,\cF)$ will be used in the sequel.

\begin{theorem}\label{th:BBform}
There exists a non-zero, strongly local, regular, conservative symmetric Dirichlet form $(\cE,\cF)$ on $L^2(K,\mu)$, that is invariant under the isometries of $[0,1]^2$ preserving $K$ and has the following properties for some constant $\rho>1$.
\begin{enumerate}
\item (Self-similarity) For every $f\in\cF$ and $i\in S$, $\psi_i^*f\in\cF$ and
\[
\cE(f,f)=\rho\sum_{i\in S}
\cE(\psi_i^*f,\psi_i^*f).
\]
Consequently, $\cE(f,f)=\rho^m\sum_{w\in W_m}\cE(\psi_w^*f,\psi_w^*f)$ for every $m\in\N$ and $f\in\cF$.
\item (Resistance estimate) Every $f\in\cF$ admits a continuous modification, again denoted by $f$. With $\gamma:=\log\rho/\log3\, (>0)$ there exists $C_R>0$ such that
\[
(f(x)-f(y))^2\le C_R|x-y|^\gamma\cE(f,f),
\qquad x,y\in K,\ f\in\cF.
\]
In particular, with $\bar R:=C_R(\diam K)^\gamma<\infty$, $(\osc_K f)^2\le\bar R\cE(f,f)$ for every $f\in\cF$.
\item (Gluing) Let $m\in\N$ and $f\in C(K)$. If $\psi_w^*f\in\cF$ for every $w\in W_m$, then $f\in\cF$.
\end{enumerate}
\end{theorem}

This is the standard Dirichlet form on the planar Sierpinski carpet: it was constructed in \cite{KZ92} (see also \cite{BB89} for the construction of the standard diffusion process).
The planar carpet satisfies the
standing assumptions of \cite[Section~4.3]{Hi13} (symmetry, connectedness, non-diagonality, and borders included), and item~(i) is condition~(A3) there, with $r_i=1/\rho$ for all $i\in S$ in the notation of \cite{Hi13}.
Moreover, $\rho=3^{d_\mathrm{w}-d_\mathrm{f}}$, where $d_\mathrm{w}=\log(8\rho)/\log3$ is the walk dimension and $\rho$ coincides with the resistance scale factor $\rho_F$ of \cite[(5.2)--(5.3)]{BB99}. Since $d_\mathrm{w}\ge2$ by \cite[Corollary~5.3]{BB99} and $d_\mathrm{f}=\log8/\log3<2$, we have $\rho>1$ and $\gamma=d_\mathrm{w}-d_\mathrm{f}>0$.

Item (ii) follows from the pointwise resistance estimate in \cite[proof of Theorem~2.8]{GY19}, since $d_\mathrm{w}-d_\mathrm{f}=\gm$; the underlying resistance scaling goes back to \cite{BB90}.

Item~(iii) follows from \cite[Proposition~5.1]{Hi13}. 

We note two consequences of Theorem~\ref{th:BBform}~(ii). First, for every $f\in\cF$ and $x\in K$,
\[
f(x)^2\le2(f(x)-f(y))^2+2f(y)^2 
\le2\bar R\cE(f,f)+2f(y)^2,\qquad y\in K;
\]
integrating in $y$ with respect to $\mu$ gives $f(x)^2\le2(\bar R\vee1)\bigl(\cE(f,f)+\|f\|_{L^2(K,\mu)}^2\bigr)$.
Hence,
\[
\Cp_1(\{x\})\ge\bigl\{2(\bar R\vee1)\bigr\}^{-1}>0,
\qquad x\in K.
\]
In particular, every non-empty subset of $K$ has positive $\Cp_1$-capacity, and properties stated to hold quasi-everywhere hold
everywhere.

The energy measures inherit the self-similarity.
\begin{lemma}[{\cite[Lemma~4.4~(ii)]{Hi13}}]\label{lem:emss}
For every $f\in\cF$,
\[
\nu_f=\rho\sum_{i\in S}(\psi_i)_* \nu_{\psi_i^*f}
\]
as Borel measures on $K$. Consequently, for every $m\in\N$, $\nu_f=\rho^m\sum_{w\in W_m}(\psi_w)_*\nu_{\psi_w^*f}$.
\end{lemma}

The key input from \cite{Hi13} is that energy measures do not charge the grid lines.

\begin{lemma}\label{lem:boundary}
\begin{enumerate}
\item $\nu_f(\Gamma_\infty)=0$ for every $f\in\cF$.
\item The minimal energy-dominant measure $\nu$ fixed in Section~\ref{sec:results} satisfies $\nu(\Gamma_\infty)=0$.
\item For every $f\in\cF$, $m\in\N$, and
$w\in W_m$,
\[
\nu_f(K_w)=2\rho^m\cE(\psi_w^*f,\psi_w^*f).
\]
Moreover, for any $I\subset W_m$, $\nu_f\bigl(\bigcup_{w\in I}K_w\bigr)=\sum_{w\in I}\nu_f(K_w)$.
\end{enumerate}
\end{lemma}

\begin{proof}
(i) This follows from \cite[Proposition~4.15 and Lemma~4.6~(i)]{Hi13}; the invariance of $(\cE, \cF)$ under the isometries of $[0,1]^2$ preserving $K$ (Theorem~\ref{th:BBform}) is used there.

(ii) By \cite[Proposition~2.7]{Hi10}, there exists $g\in\cF$ such that $\nu_g$ is a minimal energy-dominant measure of $(\cE,\cF)$. Any two minimal energy-dominant measures are mutually absolutely continuous by the definition of minimality. Thus (i) implies the assertion.

(iii) By \cite[Lemma~4.6~(ii)]{Hi13},
$\nu_f(B)=\rho^m\nu_{\psi_w^*f}\bigl(\psi_w^{-1}(B)\bigr)$ for every Borel set $B\subset K_w$; taking $B=K_w$ and using $\nu_{\psi_w^*f}(K)=2\cE(\psi_w^*f,\psi_w^*f)$ gives the identity. The additivity holds because the pairwise intersections $K_w\cap K_{w'}$ ($w,w'\in I$, $w\ne w'$) are contained in $\Gamma_m$.
\end{proof}

\begin{proposition}\label{prop:A1A2}
{\rm(A1)} holds. Moreover, {\rm(A2)} holds with
$\Lambda_n:=W_n$ and
\[
U^{(n)}_w:=K\cap\psi_w\bigl((0,1)^2\bigr),\qquad w\in W_n,
\]
and with the Euclidean metric $\sfd$.
\end{proposition}

\begin{proof}
(A1) is immediate from Theorem~\ref{th:BBform}. We show that $\{U^{(n)}_w\}_{w\in W_n}$ is a partition of $K$ for each $n\in\N$.
Each $U^{(n)}_w$ is non-empty, open, and relatively compact.
The sets are pairwise disjoint because the open squares $\psi_w\bigl((0,1)^2\bigr)$, $w\in W_n$, are. 
Also,
$K\setminus\bigsqcup_{w\in W_n}U^{(n)}_w=\Gamma_n$.
It is easy to see that $\mu(\Gamma_n)=0$.
From Lemma~\ref{lem:boundary}~(ii), $\nu(\Gamma_n)=0$.
Therefore, $\{U^{(n)}_w\}_{w\in W_n}$ is a partition of $K$.

(A2)~(a): for $w\in W_n$ and $i\in S$, $U^{(n+1)}_{wi}\subset
U^{(n)}_w$. 

(A2)~(b): with the Euclidean metric,
$\sup_{w\in W_n}\diam U^{(n)}_w\le\sqrt2\,3^{-n}\to0$ as
$n\to\infty$.
\end{proof}

We now construct the window families. The construction applies to an
arbitrary non-empty open set $U\subsetneq K$ and produces constants
independent of $U$; this uniformity is what makes
(A3') hold with a single set of constants for all
$U^{(n)}_\lambda$ simultaneously.

Let $U$ be a non-empty open subset of $K$ with $U\ne K$. A word $v\in W_m$ ($m\in\N$) is called
\emph{$U$-admissible} if
\[
\dist(K_v,K\setminus U)>20\cdot3^{-m},
\]
where $\dist$ denotes the Euclidean distance. Since
$K\setminus U\ne\emptyset$, we have
$\dist(K_v,K\setminus U)\le\diam K=\sqrt2$;
hence every $U$-admissible $v\in W_m$ satisfies
$20\cdot3^{-m}<\sqrt2$, that is, $m\ge3$. In particular, the empty word is
not $U$-admissible. A word $v$ is called a \emph{Whitney word} for
$U$ if $v$ is $U$-admissible and its parent $v^-$ (the word obtained
by deleting the last letter) is not $U$-admissible. We write
$\Wh(U)$ for the totality of Whitney words for $U$ and
define
\[
\cW(U):=\bigl\{(F_v,O_v)\;\big|\;v\in\Wh(U)
\bigr\},\qquad
F_v:=K_v,\quad
O_v:=\Int_K\cN_2(K_v).
\]

\begin{lemma}\label{lem:whitney}
Let $U$ be a non-empty open subset of $K$ with $U\ne K$. Then, the following hold.
\begin{enumerate}
\item For every $v\in\Wh(U)$ with $|v|=m$, the set $F_v$ is a compact subset of the open set $O_v$ and $\cN_2(K_v)\subset U$ (in particular, $O_v\subset U$).
\item $\bigcup_{v\in\Wh(U)}F_v=U$. In particular, {\rm(A3')~(b)} holds for $\cW(U)$.
\item If $v,v'\in\Wh(U)$ and $O_v\cap O_{v'}\ne\emptyset$, then $\bigl||v|-|v'|\bigr|\le2$. Moreover, $\sum_{v\in\Wh(U)}\mathbf 1_{O_v}\le 147$ on $K$; in particular, {\rm(A3')~(c)} holds for $\cW(U)$.
\end{enumerate}
\end{lemma}

\begin{proof}
(i) $F_v=K_v$ is compact, and $K_v\subset \Int_K\cN_1(K_v)\subset O_v$ by (G3). If $z\in\cN_2(K_v)$, then $\dist(z,K_v)\le \diam\cN_2(K_v)\le5\sqrt2\,3^{-m} <20\cdot3^{-m}<\dist(K_v,K\setminus U)$ by (G2), so that $z\notin K\setminus U$; hence $\cN_2(K_v)\subset U$.

(ii) The inclusion $\subset$ follows from $F_v\subset\cN_2(K_v)\subset U$ by (i). Conversely, let $x\in U$ and $\delta:=\dist(x,K\setminus U)>0$. Choose $m\in\N$ with $22\cdot3^{-m}<\delta$ and $v\in W_m$ with $x\in K_v$. Then
\[
\dist(K_v,K\setminus U)
\ge\delta-\diam K_v\ge\delta-\sqrt2\,3^{-m}
>(22-\sqrt2)3^{-m}>20\cdot3^{-m},
\]
so $v$ is $U$-admissible. Among the prefixes of $v$ that are $U$-admissible (a non-empty set containing $v$), let $\tilde v$ be one of minimal length; then $\tilde v\ne\emptyset$ and $\tilde v^-$ is not $U$-admissible, so $\tilde v\in\Wh(U)$ and $x\in K_v\subset K_{\tilde v}=F_{\tilde v}$. Thus $U\subset\bigcup_{v\in \Wh(U)}F_v$, and (A3')~(b) holds with $\Cp_1\bigl(U\setminus\bigcup_{v\in \Wh(U)}F_v\bigr)=\Cp_1(\emptyset)=0$.

(iii) Let $m:=|v|\le m':=|v'|$ and take $z\in O_v\cap O_{v'}$. By (G2),
\[
\dist(K_v,K_{v'})\le5\sqrt2\,3^{-m}+5\sqrt2\,3^{-m'}
\le10\sqrt2\,3^{-m}.
\]
Since the parent $p:=(v')^-\in W_{m'-1}$ is not $U$-admissible and $K_{v'}\subset K_p$,
\begin{align*}
\dist(K_{v'},K\setminus U)
&\le\dist(K_p,K\setminus U)+\diam K_p\\
&\le 20\cdot3^{-(m'-1)}+\sqrt2\,3^{-(m'-1)}\\
&=(60+3\sqrt2)3^{-m'}.
\end{align*}
Combining these with the $U$-admissibility of $v$ and the triangle inequality,
\begin{align*}
20\cdot3^{-m}
&<\dist(K_v,K\setminus U)\\
&\le\dist(K_v,K_{v'})+\diam K_{v'}+\dist(K_{v'},K\setminus U)\\
&\le10\sqrt2\,3^{-m}+(60+4\sqrt2)3^{-m'}.
\end{align*}
This gives 
\[
3^{m'-m}<\frac{60+4\sqrt2}{20-10\sqrt2}<12<27,
\]
whence $m'-m\le2$.

For the overlap bound, fix $x\in K$ and let $\Wh_x:=\{v\in\Wh(U)\mid x\in O_v\}$. Any two elements of $\Wh_x$ contain $x$, so their levels differ by at most $2$ by the above; in particular, the levels occurring in $\Wh_x$ lie in a set of the form $\{m_0,m_0+1,m_0+2\}$. 
For a fixed level $m$, let $v\in W_m$ satisfy $x\in O_v$. Since
$O_v\subset\cN_2(K_v)$, there exists $u\in W_m$ (depending on $v$) with $x\in K_u$ and $D_m(u,v)\le 2$. As each neighbor step changes the square index by at most one in each coordinate, $\|3^m(\iota_v-\iota_u)\|_\infty\le 2$.
Fix one $u_0\in W_m$ with $x\in K_{u_0}$, which is independent of $v$. Since $x\in K_u\cap K_{u_0}$, either $u=u_0$ or $u$ and $u_0$ are neighbors by (G1); in both cases $\|3^m(\iota_u-\iota_{u_0})\|_\infty\le 1$. Hence $\|3^m(\iota_v-\iota_{u_0})\|_\infty\le 3$. There are at most $7^2=49$ such $v$, and therefore $\#\Wh_x\le 3\cdot 49=147$.
\end{proof}

The next lemma is the first step to confirm (A3')~(d). 
Note that it holds for all functions in $\cF$, not only for harmonic ones.

\begin{lemma}\label{lem:osc}
For every $m\ge1$, $v\in W_m$, and $f\in\cF$,
\[
\Bigl(\osc_{O_v}f\Bigr)^2
\le5\bar R\rho^{-m}\nu_f(O_v),
\qquad\text{where }O_v=\Int_K\cN_2(K_v).
\]
\end{lemma}
\begin{proof}
For a single cell $K_w$, $w\in W_m$, we have $\osc_{K_w}f=\osc_K(\psi_w^*f)$, so Theorem~\ref{th:BBform}~(ii) and Lemma~\ref{lem:boundary}~(iii) give
\begin{equation}
\Bigl(\osc_{K_w}f\Bigr)^2
\le\bar R\cE(\psi_w^*f,\psi_w^*f)
=\bar R(2\rho^m)^{-1}\nu_f(K_w).
\label{eq:celloscil}
\end{equation}
Now let $x,y\in\cN_2(K_v)$, say $x\in K_u$ and $y\in K_{u'}$ with $D_m(u,v)\le2$ and $D_m(u',v)\le2$. Concatenating a geodesic chain from $u$ to $v$ with one from $v$ to $u'$, we obtain a chain $v_1,\dots,v_p$ in $W_m$ with $v_1=u$, $v_p=u'$, $p\le5$, consecutive members being neighbors, each $v_i$ satisfying $D_m(v_i,v)\le2$, and each word occurring at most twice in the chain.
Since $K_{v_i}\cap K_{v_{i+1}}\ne\emptyset$, by telescoping,
\[
|f(x)-f(y)|\le\sum_{i=1}^p\osc_{K_{v_i}}f.
\]
By the Cauchy--Schwarz inequality, \eqref{eq:celloscil}, and Lemma~\ref{lem:boundary}~(iii),
\begin{align*}
(f(x)-f(y))^2
&\le p\sum_{i=1}^p\Bigl(\osc_{K_{v_i}}f\Bigr)^2
\le5\cdot2\!\!\sum_{\substack{w\in W_m\\D_m(w,v)\le2}}\!\!\Bigl(\osc_{K_w}f\Bigr)^2\\
&\le10\bar R(2\rho^m)^{-1}\!\!\sum_{\substack{w\in W_m\\D_m(w,v)\le2}}\!\!\nu_f(K_w)
=5\bar R\rho^{-m}
\nu_f\bigl(\cN_2(K_v)\bigr).
\end{align*}
Since $\cN_2(K_v)\setminus O_v\subset\Gamma_m$ by (G3) and $\nu_f(\Gamma_m)=0$ by Lemma~\ref{lem:boundary}~(i), we have $\nu_f\bigl(\cN_2(K_v)\bigr)=\nu_f(O_v)$. Taking the supremum over $x,y\in O_v\subset\cN_2(K_v)$ completes the proof.
\end{proof}

We turn to the capacity bounds.
\begin{lemma}\label{lem:cap}
There exist constants $c_7>0$ and $c_8>0$ such that for every
$m\ge2$ and $v\in W_m$, with $O_v=\Int_K
\cN_2(K_v)$, the following hold.
\begin{enumerate}
\item $\Cp(K_v;O_v)\le c_7\rho^m$.
\item $\Cp(\{x\};O_v)\ge c_8\rho^m$ for
every $x\in K_v$.
\end{enumerate}
In particular, $0<\Cp(K_v;O_v)<\infty$.
\end{lemma}

\begin{proof}
(ii) Let $x\in K_v$ and let $f\in\cF$ satisfy that $f\ge1$ $\mu$-a.e.\ on some open set containing $x$ and $f=0$ q.e.\ on $K\setminus O_v$. Since $f$ is continuous, $f(x)\ge1$ and $f=0$ on $K\setminus O_v$. 
Since $\max_{y\in K}|x-y|\ge(\diam K)/2=\sqrt2/2>2/3\ge 6\cdot3^{-m}$ and $K$ is connected, the continuous function $y\mapsto|x-y|$ attains the value $6\cdot3^{-m}$ at some $y_0\in K$.
For $z\in O_v\subset\cN_2(K_v)$, choose $u\in W_m$ with $z\in K_u$ and $D_m(u,v)\le2$. Then the square $Q_u$ has index distance at most $2$ from $Q_v$, while $z\in Q_u$ and $x\in Q_v$, so that each coordinate of $z-x$ is bounded in absolute value by $3\cdot3^{-m}$. Hence $|z-x|\le3\sqrt2\,3^{-m}<6\cdot3^{-m}$.
This means that $y_0\in K\setminus O_v$ and $f(y_0)=0$. Theorem~\ref{th:BBform}~(ii) now yields
\[
1\le(f(x)-f(y_0))^2
\le C_R(6\cdot3^{-m})^\gamma\cE(f,f)
=C_R 6^\gamma\rho^{-m}\cE(f,f).
\]
Taking the infimum over $f$ gives (ii) with $c_8=C_R^{-1}6^{-\gamma}$.

(i) For $v\in W_m$, define the \emph{pattern} $\varpi(v)$ of $v$ by
\[
  \varpi(v)\colon\{-3,-2,-1,0,1,2,3\}^2\to\{0,1\},\qquad
  \varpi(v)(j):=\begin{cases}1&\text{if }K_v+3^{-m}j\subset K\\
  0&\text{otherwise.}\end{cases}
\]
Let $\Pi$ be the set of all patterns $\varpi(v)$, $v\in\bigcup_{m\ge2}W_m$;
clearly $\#\Pi\le 2^{49}$.
Write $N_v$ for the interior of the corresponding $7\times7$ region, an open square of side $7\cdot3^{-m}$ concentric with $Q_v$. Since a level-$m$ square at $\|\cdot\|_\infty$-index distance at least $4$ from $v$ does not meet $N_v$, we have
\[
K\cap N_v=\bigcup\bigl\{K_{v'}\cap N_v\;\big|\;
v'\in W_m,\ \|3^m(\iota_{v'}-\iota_v)\|_\infty\le3\bigr\},
\]
and each such $K_{v'}$ is a similar copy of $K$. Consequently, if $\varpi(v)=\varpi(v^*)$ for $v\in W_m$, $v^*\in W_{m^*}$, then the homothety $\phi$ of $\R^2$ with ratio $3^{m^*-m}$ mapping $Q_{v^*}$ onto $Q_v$ satisfies $\phi(K\cap N_{v^*})=K\cap N_v$, $\phi(K_{u^*})=K_u$ for the corresponding words in the two $7\times7$ arrays, $\phi\bigl(\cN_2(K_{v^*})\bigr)=\cN_2(K_v)$, and $\phi(O_{v^*})=O_v$. Moreover, $\phi\circ\psi_{u^*}=\psi_u$ on $K$ for corresponding words $u^*$, $u$, both sides being the orientation-preserving similarity $z\mapsto3^{-m}z+\iota_u$.

For each $\varpi\in\Pi$, fix a representative $v^*(\varpi)$ of level $m^*(\varpi)\,(\ge2)$ with $\varpi(v^*(\varpi))=\varpi$, and set
\[
O^*:=\Int_K\cN_2(K_{v^*})\setminus\bigcup\bigl\{K_u\;\big|\;u\in W_{m^*},\ D_{m^*}(u,v^*)\ge3\bigr\},
\]
which is an open subset of $K$ containing the compact set
$K_{v^*}$. By the regularity of $(\cE,\cF)$, there
exists $g_\varpi\in\cF\,(\subset C(K))$ with $0\le g_\varpi\le1$,
$g_\varpi=1$ on an open set containing $K_{v^*}$, and
$\supp g_\varpi\subset O^*$.
Set
\[
c_7:=\max_{\varpi\in\Pi}\rho^{-m^*(\varpi)}\cE(g_\varpi,g_\varpi)<\infty .
\]

Now let $v\in W_m$, $\varpi:=\varpi(v)$, $v^*:=v^*(\varpi)$,
$m^*:=m^*(\varpi)$, and let $\phi$ be the homothety above. Define $g_v\in C(K)$ as
\[
g_v:=
\begin{cases}
g_\varpi\circ\phi^{-1}&\text{on }\cN_2(K_v),\\
0&\text{on }K\setminus\cN_2(K_v).
\end{cases}
\]
Then Theorem~\ref{th:BBform}~(iii) implies $g_v\in\cF$.
The identity in Theorem~\ref{th:BBform}~(i) gives
\begin{align*}
\cE(g_v,g_v)
&=\rho^m\sum_{u\in W_m}
\cE(\psi_u^*g_v,\psi_u^*g_v)\\
&=\rho^m\!\!\sum_{\substack{u^*\in W_{m^*}\\D_{m^*}(u^*,v^*)\le2}}\!\!\cE(\psi_{u^*}^*g_\varpi,\psi_{u^*}^*g_\varpi)\\
&\le\rho^{m-m^*}\cE(g_\varpi,g_\varpi)\\
&\le c_7\rho^m .
\end{align*}
Therefore $\Cp(K_v;O_v)\le\cE(g_v,g_v)\le c_7\rho^m$.
\end{proof}

\begin{proposition}\label{prop:A3W}
{\rm(A3')} holds for $(\cE,\cF)$ with the partitions of Proposition~{\rm\ref{prop:A1A2}}. 
More precisely, {\rm(A3')} holds with $J=147$ and $C=(5\bar R c_7)\vee(c_7/c_8)$, and $\Lm_n=W_n$, $\Theta_w=\Wh(U_w^{(n)})$, and $\cW_w=\cW(U_w^{(n)})$ for each $w\in W_n$.
Moreover, {\rm(A3')~(d)} holds for all $h\in\cF$, and {\rm(A3')~(e)} holds with ``q.e.'' replaced by ``every.''
\end{proposition}

\begin{proof}
Let $n\in\N$ and $w\in W_n$.
Let $v\in\Wh(U_w^{(n)})$ and $m:=|v|\,(\ge3)$. 
(A3')~(a): $F_v=K_v$ is a compact subset of the open set $O_v$, and $O_v\subset\cN_2(K_v)\subset U_w^{(n)}$ by Lemma~\ref{lem:whitney}~(i); moreover,
$0<\Cp(F_v;O_v)<\infty$ by Lemma~\ref{lem:cap}. (A3')~(b) and (c)
are Lemma~\ref{lem:whitney}~(ii) and (iii). 
(A3')~(d): for every $f\in\cF$, by Lemmas~\ref{lem:cap}~(i) and \ref{lem:osc},
\[
\Cp(F_v;O_v)\Bigl(\muosc_{O_v}f\Bigr)^2
\le c_7\rho^m\cdot5\bar R\rho^{-m}\nu_f(O_v)
\le C\nu_f(O_v).
\]
(A3')~(e): for every $x\in F_v=K_v$, by Lemma~\ref{lem:cap}~(ii),
\[
\Cp(\{x\};O_v)\ge c_8\rho^m
\ge\frac{c_8}{c_7}\Cp(F_v;O_v)
\ge C^{-1}\Cp(F_v;O_v).\qedhere
\]
\end{proof}

We arrive at the main result of this subsection.
\begin{theorem}\label{thm:SC}
Let $K$ be the two-dimensional standard Sierpinski carpet, $\mu$
the self-similar probability measure with equal weights, and
$(\cE,\cF)$ the Dirichlet form on $L^2(K,\mu)$ given in
Theorem~{\rm\ref{th:BBform}}. Then, the AF-martingale dimension $d_{\mathrm m}$ of
$(\cE,\cF)$ is equal to one.
\end{theorem}

\begin{proof}
Assumptions (A1) and (A2) hold by Proposition~\ref{prop:A1A2}, and
(A3') holds by Proposition~\ref{prop:A3W}. The claim follows from
Theorem~\ref{th:main}.
\end{proof}

\begin{remark}\label{rem:SC}
Since $1\le d_{\mathrm m}\le d_{\mathrm s}$ by \cite[Theorem~4.16]{Hi13} and the spectral
dimension $d_{\mathrm s}=2\log 8/\log(8\rho)$ of the planar Sierpinski carpet is smaller
than $2$ (equivalently, $\rho>1$), the conclusion of Theorem~\ref{thm:SC} itself is
contained in \cite{Hi13}; the same is true for the higher-dimensional carpets discussed
below. The point here is the method: the main property required in the argument is the
local resistance estimate (Theorem~\ref{th:BBform}~(ii)).

In particular, the argument extends to higher-dimensional generalized Sierpinski carpets in
the sense of \cite{BB99} (with the strengthened non-diagonality condition, see
\cite[Remark~2.10]{BBKT10}) whose spectral dimension $d_{\mathrm s}$ is less than $2$,
equivalently $d_{\mathrm w}>d_{\mathrm f}$; such carpets exist, for example, the
Menger-sponge-type carpet $MS(6,4)$ in \cite[Section~9]{BB99}. We only outline the
modifications. Items (i) and (iii) of Theorem~\ref{th:BBform} and
Lemmas~\ref{lem:emss}--\ref{lem:boundary} hold for these carpets by
\cite[Section~4.3]{Hi13}. The pointwise resistance estimate in
Theorem~\ref{th:BBform}~(ii), now with exponent $d_{\mathrm w}-d_{\mathrm f}>0$, follows
from the two-sided sub-Gaussian heat kernel estimates on the compact carpet (see
\cite[Section~4.3]{Hi13} and the references therein) by \cite[(3.4)]{Hu08}, which is
established there for irreducible regular Dirichlet forms on compact metric spaces with the
chain condition (cf.\ \cite[Theorem~2.2]{Hu08}); the chain condition holds for the
Euclidean metric by the borders-included hypothesis \cite[Remark~2.2]{BB99}. The
cube-based Whitney windows again have uniformly bounded overlap, and the finite-pattern
cutoff construction in Lemma~\ref{lem:cap} carries over, with all constants depending on
the carpet. In this way one can verify (A3') for these carpets.

Window-based verifications of the same type are expected to be applicable to some
non-self-similar spaces on which uniform local resistance estimates are available; we
leave such extensions for future investigation.
\end{remark}



\begin{funding}
The author was supported in part by Japan Society for the Promotion of Science (JSPS) KAKENHI Grant Number 25K07056 and The Kyoto University Foundation.
\end{funding}

\newenvironment{useofai}[1][Use of AI tools]{\small\paragraph*{#1}}{\par}
\begin{useofai}
The author used ChatGPT (OpenAI) and Claude (Anthropic) to assist in reviewing mathematical arguments, checking proofs, and improving the presentation of the manuscript. All mathematical statements and proofs were independently verified by the author, who takes full responsibility for the content of the paper.
\end{useofai}

\end{document}